\documentclass[10pt,a4paper,oneside,reqno]{amsart}

\usepackage{setspace}
\usepackage[totalwidth=14.1cm,totalheight=22.6cm]{geometry}
\usepackage[T1]{fontenc}
\usepackage{graphicx}
\usepackage{caption}
\usepackage{epsfig}
\usepackage{amsmath,amssymb,amscd,amsthm}
\usepackage{subfigure}
\usepackage{latexsym}
\usepackage{graphics}
\usepackage{colortbl} 
\usepackage{color}
\usepackage{ae}
\usepackage{enumitem}
\usepackage[all]{xy}
\usepackage{scalerel}
\usepackage{tikz}
\usepackage{cancel}
\usepackage{bbm,amsbsy}
\usepackage{mathrsfs}  
\usepackage[colorlinks=true,pagebackref=true]{hyperref}
\usepackage[normalem]{ulem}
\usepackage{nicefrac}
\usepackage{xfrac}
\usepackage{faktor}
\usepackage{tikz-cd} 

\usepackage{nomencl}
\usepackage[normalem]{ulem}

\makenomenclature

\makeatletter
\newtheorem*{rep@theorem}{\rep@title}
\newcommand{\newreptheorem}[2]{%
\newenvironment{rep#1}[1]{%
 \def\rep@title{#2 \ref{##1}}%
 \begin{rep@theorem}}%
 {\end{rep@theorem}}}
\makeatother

\makeatletter
\newtheorem*{rep@cor}{\rep@title}
\newcommand{\newrepcor}[2]{%
\newenvironment{rep#1}[1]{%
 \def\rep@title{#2 \ref{##1}}%
 \begin{rep@cor}}%
 {\end{rep@cor}}}
\makeatother

\makeatletter
\newtheorem*{rep@prop}{\rep@title}
\newcommand{\newrepprop}[2]{%
\newenvironment{rep#1}[1]{%
 \def\rep@title{#2 \ref{##1}}%
 \begin{rep@prop}}%
 {\end{rep@prop}}}
\makeatother

\newtheorem{cor}{Corollary}[section]

\newtheorem{claim}[cor]{Claim}

\newtheorem{theorem}[cor]{Theorem}

\newtheorem{prop}[cor]{Proposition}

 \allowdisplaybreaks
\newrepcor{cor}{Corollary}
\newreptheorem{theorem}{Theorem}
\newrepprop{prop}{Proposition}

\newtheorem{lemma}[cor]{Lemma}

\theoremstyle{definition}
\newtheorem{defi}[cor]{Definition}
\theoremstyle{remark}
\newtheorem{remark}[cor]{Remark}
\newtheorem*{remark*}{Remark}
\newtheorem{example}[cor]{Example}

\newtheorem*{notation*}{Notation}

\newlist{steps}{enumerate}{1}
\setlist[steps, 1]{itemsep=8pt,leftmargin=0cm,itemindent=.5cm,labelwidth=\itemindent,labelsep=0cm,align=left,label = \textbf{\emph{Step \arabic*}:\,}}

\makeatletter
\newcommand{\myitem}[1]{%
\item[#1]\protected@edef\@currentlabel{#1}%
}
\makeatother

\newcommand{\A}{\mathrm{A}}
\newcommand{\psl}{\mathrm{P}\mathrm{S}\mathrm{L}(2,\mathbb{R})}
\newcommand{\ads}{\mathbb{A}\mathrm{d}\mathbb{S}^3}
\newcommand{\HP}{\mathbb{HP}^3}

\newcommand{\isom}{\mathrm{Isom}}

\newcommand{\R}{\mathrm{R}}

\newcommand{\X}{\mathrm{X}}
\newcommand{\B}{\mathrm{B}}

\begin{document}\raggedbottom

\setcounter{secnumdepth}{3}
\setcounter{tocdepth}{2}

\title[Transition of quasifuchsian structures]{TRANSITION OF CONVEX CORE DOUBLES FROM HYPERBOLIC TO ANTI-DE SITTER GEOMETRY}

\author[Farid Diaf]{Farid Diaf}
\address{Farid Diaf: Univ. Grenoble Alpes, CNRS, IF, 38000 Grenoble, France.} \email{farid.diaf@univ-grenoble-alpes.fr}

\thanks{}

\maketitle

\begin{abstract}
Let $\Sigma$ be a surface of negative Euler characteristic, homeomorphic to a closed surface, possibly with a finite number of points removed. In this paper, we present a construction method for a wide range of examples of geometric transition from hyperbolic to Anti-de Sitter structures via Half-pipe geometry on $\Sigma\times\mathbb{S}^1$, with cone singularities along a link. The main ingredient lies in studying the deformation of a convex core structure as the bending laminations of the upper and lower boundary components of the convex core uniformly collapse to zero. 
\end{abstract}
\tableofcontents
\section{Introduction}
In his notes \cite{thurstop}, Thurston introduced the concept of degeneracy of $(\mathrm{G},\mathrm{X})$ structures. Since then, important contributions have been made on this subject \cite{HOG86,POR98,HPS01,POR02,PW07,POR13,KOZ16}, and this idea was notably used to prove the famous orbifold theorem \cite{BLP05, CHK00}.

A geometric transition consists of a deformation of a $(\mathrm{G}, \X)$-structure on a manifold $M$ that degenerates but admits a limit in a different type of geometric structure after "stretching" in the direction of the collapse. It has been known since Klein's time \cite{Klein} that it is possible to transition from hyperbolic to spherical geometry through Euclidean geometry. In his PhD thesis, Danciger \cite{danciger_thesis} introduced a geometric transition between hyperbolic geometry and Anti-de Sitter (AdS) geometry, which is the analogue of hyperbolic space in Lorentzian geometry. To accomplish this, Danciger introduced a geometry known as the \textit{Half-pipe geometry} (HP), which is a limit geometry inside projective geometry of both hyperbolic and Anti-de Sitter geometry in the sense of \cite{CDW} (see also \cite{dancigerGT,dancigerJT,andreafrancois}). Specifically, a transition from hyperbolic to AdS geometry via HP geometry is a path $\mathcal{P}_t$ of real projective structures on a manifold $M$ such that $\mathcal{P}_t$ is conjugate to a hyperbolic structure if $t>0$, or to an AdS structure if $t < 0$, and to a Half-pipe structure when $t=0$.
Several examples of such transitions are given, see \cite{dancigerGT,dancigerGT} for examples in dimension three and \cite{riolo_seppi,riolo_seppi2} for examples in dimension four. The above results strengthen the similarity between hyperbolic and AdS geometry. The first motivation of this paper will be to provide more examples of such geometric transitions.

In dimension three, there are remarkable similarities between hyperbolic and AdS geometry illustrated by Bers' Simultaneous Uniformization Theorem \cite{Bers} for quasi-Fuchsian hyperbolic manifolds and Mess's classification of maximal globally hyperbolic AdS space-times \cite{Mess}. These classes of structures share many features, such as having a \textit{convex core}. An intriguing question for these structures is whether the geometry of the convex core contains all the information about the geometry of the global manifold. Several contributions have been made on this subject \cite{Labourie1992MtriquesP,bonahonotal,Lecuire,Bonsante2006AdsMW,Fixedpoint,Diallo2013PrescribingMO}. Later, Barbot and Fillastre \cite{barbotfillastre} introduced \textit{Quasi-Fuchsian co-Minkowski manifolds}, which are the analogues in Half-pipe geometry of quasi-Fuchsian and maximal globally hyperbolic AdS space-time manifolds. The second motivation of this paper is to provide a connection between the geometry of the convex core of quasi-Fuchsian hyperbolic manifolds and their analogues in AdS and HP geometry through the geometric transition.

\subsection{Examples of transition in dimension 3.}
Danciger provides two explicit (infinite) classes of examples of $3$-manifolds supporting a transition from hyperbolic geometry to AdS geometry through HP geometry. The first class, as stated in \cite[Theorem 1.1]{dancigerGT}, consists of the unit tangent bundle of the $(2, m, m)$ triangle orbifold. The second class \cite[Theorem 3]{dancigerJT} consists of a suspension of a punctured torus by an Anosov diffeomorphism.  It seems natural to ask whether there are another classes of $3$-manifolds supporting such transition. In this paper we provide a new large class of examples. Our main result is the following: 
\begin{theorem}\label{mainthm}
Let $\Sigma$ be a surface of negative Euler characteristic, homeomorphic to a closed surface, possibly with a finite number of points removed. Let $\lambda$ and $\mu$ be two weighted multicurves that fill $\Sigma$. Consider $L=\left(\vert \lambda\vert \times \{1\}\right)\cup\left(\vert \mu\vert \times \{-1\}\right)$ and $$M:=(\Sigma\times\mathbb{S}^1)\setminus L.$$ Then there exists a continuous path $\{\mathcal{P}_t\}_{(-\epsilon,\epsilon)}$ of real projective structures on $M$ with conical singularities on $L$ such that:
\begin{itemize}
\item If $t > 0$, $\mathcal{P}_t$ is conjugate to the hyperbolic doubled convex core structure with bending data $(\vert t\vert \lambda,\vert t\vert \mu)$,
\item If $t < 0$, $\mathcal{P}_t$ is conjugate to the Anti-de Sitter doubled convex core structure with bending data $(\vert t\vert \lambda,\vert t\vert \mu)$,
\item For $t=0$, $\mathcal{P}_0$ to corresponds to the Half-pipe doubled convex core structure with bending data $(\lambda,\mu)$.
\end{itemize}
The cone angle $\theta_t$ around the link $L$ is given by \begin{center}
    $\begin{cases}
\theta_t(\alpha\times\{1\})=2(\pi-\vert t\vert \lambda(\alpha)) &\text{if}\ t>0\ \mathrm{and}\ \alpha\in \vert \lambda\vert\\
\theta_t(\alpha\times\{-1\})=2(\pi-\vert t\vert \mu(\alpha))     &\text{if}\ t>0\ \mathrm{and}\ \alpha\in \vert \mu\vert\\
\theta_t(\alpha\times\{1\})=-2\vert t\vert \lambda(\alpha)      &\text{if}\ t<0\ \mathrm{and}\ \alpha\in \vert \lambda\vert\\
\theta_t(\alpha\times\{-1\})=-2\vert t\vert \mu(\alpha)          &\text{if}\ t<0\ \mathrm{and}\ \alpha\in \vert \mu\vert
\end{cases}$\end{center}
Moreover, if the surface $\Sigma$ has punctures, then for a neighborhood $V$ around a puncture, the structure of $\mathcal{P}_t$ on $V\times \mathbb{S}^1$ is conjugate to a cusp in $\mathbb{H}^3$ (if $t>0$), in $\ads$ (if $t<0$), or in $\HP$ (if $t=0$).
\end{theorem}

Let us explain some terminology used in the statement of Theorem \ref{mainthm}. Roughly speaking, a hyperbolic (resp. AdS or HP) \textit{convex core structure} on $\Sigma\times[0,1]$ is a hyperbolic (resp. AdS or HP) structure on $\Sigma\times [0,1]$ for which the holonomy representation $\rho$ can be deformed in a suitable sense to a Fuchsian representation. This structure makes $\Sigma\times[0,1]$ isometric to $\mathrm{CH}(\Lambda_{\rho})/\rho(\pi_1(\Sigma))$, where $\mathrm{CH}(\Lambda_{\rho})$ is the convex hull of the limit set $\Lambda_{\rho}$ of $\rho$. The boundary $\partial\mathrm{CH}(\Lambda_{\rho})$ has two connected components that we will denote by $\partial_+\mathrm{CH}(\Lambda_{\rho})$ and $\partial_-\mathrm{CH}(\Lambda_{\rho})$. This gives rise to an identification of $\Sigma\times\{1\}$ with $\partial_+\mathrm{CH}(\Lambda_{\rho})/\rho(\pi_1(\Sigma))$ and $\Sigma\times\{0\} $ with $\partial_-\mathrm{CH}(\Lambda_{\rho})/\rho(\pi_1(\Sigma))$. These components are almost everywhere totally geodesic, except on two measured geodesic laminations $\lambda$ and $\mu$ which are supported where the surface is bent. We refer to $(\lambda,\mu)$ as the \textit{bending data} of the convex core structure. It is well known that $\lambda$ and $\mu$ must satisfy the \textit{filling} condition, which means that every component of the complement of the support of $\lambda$ and $\mu$ in $\Sigma$ contains at most one puncture, and it is simply connected after adding the puncture if needed.

Now, let us consider a convex core structure with bending data $(\lambda,\mu)$ and holonomy $\rho$, we furthermore assume that $\lambda$ and $\mu$ are weighted multicurves (see Definition \ref{multicurve}) which are a particular case of geodesic laminations. In that case, the hyperbolic (resp. AdS or HP) \textit{doubled convex core structure} with bending data $(\lambda,\mu)$ is the singular hyperbolic (resp. AdS or HP) structure obtained by doubling the convex core $\mathrm{CH}(\Lambda_{\rho})$ along its faces. As a result, the doubled manifold is homeomorphic to $\Sigma\times\mathbb{S}^1$, and the singular locus is $L:=(\vert\lambda\vert\times\{1\})\cup(\vert\mu\vert\times\{-1\})$, where $\vert \cdot\vert$ denotes the support of a weighted multicurve. The holonomy of a meridian $\gamma\in \pi_1(M)$ that encircles a curve $\alpha\times\{*\}$ in $L$ is a rotation of angle $\theta$. This angle $\theta$ is called the \textit{cone angle} around $\alpha\times\{*\}$. When the surface $\Sigma$ has punctures, the ends of $M$ are \textit{cusps}, which are a well-known notion in hyperbolic geometry and have been extended by Riolo and Seppi \cite{riolo_seppi} in Anti-de Sitter and Half-Pipe geometry.

\subsection{The strategy of the proof.}
The idea of the proof is divided into four main steps. Before explaining a rough idea of each step, let us recall the philosophy behind transitioning from a hyperbolic to AdS structure through HP structure. The concept is that whenever we have a path of hyperbolic or AdS structures that collapse to a hyperbolic structure onto a co-dimension one totally geodesic surface, one can hope to produce a geometric transition via half-pipe geometry after a suitable rescaling along the direction of collapse. Now, the starting point of the proof is the following: let $\lambda$ and $\mu$ be two weighted multicurves that fill $\Sigma$, Kerckhoff established the following result concerning length functions defined on Teichm\"uller space which we will only state for weighted multicurves even it holds for general measured laminations:

\begin{theorem}[\cite{kerker}]
The length function $l_{\lambda}+l_{\mu}$ defined over the Teichm\"uller space of $\Sigma$ has a unique minimum $k_{\lambda,\mu}.$
\end{theorem}
Later, Series \cite{limitefisch} proved that the path of holonomies of hyperbolic convex core structures on $\Sigma\times[0,1]$ with bending data $(t\lambda,t\mu)$ converges as $t\to 0^+$ to a holonomy of a complete hyperbolic structure on $\Sigma$ given by the \textit{Kerckhoff point} $k_{\lambda,\mu}$. An analogous result in AdS geometry was obtained by Bonsante and Schlenker \cite{Fixedpoint}, namely, the holonomy representation of the AdS convex core structure with bending data $(t\lambda,t\mu)$ converges to the Kerckhoff point as $t\to 0^+$.\\

The first step towards proving the main Theorem \ref{mainthm} is to establish the transition at the level of holonomy for the convex core structure (before doubling). Let $\rho_{(t\lambda,t\mu)}$ be the path of holonomies representations of a hyperbolic (if $t>0$) and AdS (if $t<0$) convex core structure with bending data $(\vert t\vert\lambda,\vert t\vert \mu)$. Then, we prove the following theorem, where a more detailed statement is given in Theorem \ref{H} in Section \ref{TROFHOL}:

\begin{theorem}[Transition of holonomy]\label{1.3}
After conjugating if needed, the rescaled holonomy $\tau_t\rho_{(t\lambda,t\mu)}\tau_t^{-1}$ converges as $t\to 0$ to $\rho_{(\lambda,\mu)}^{\HP}$, which is the holonomy of a Half-pipe convex core structure on $\Sigma\times[0,1]$ with bending data $(\lambda,\mu)$.
\end{theorem}

By rescaling, we mean that we conjugate $\rho_{(t\lambda,t\mu)}$ by an appropriate projective transformation $\tau_t$ which is defined in \eqref{taut} in Section \ref{sec2}. The proof of Theorem \ref{1.3} uses the description of the holonomy $\rho_{(t\lambda,t\mu)}$ in terms of a \textit{bending cocycle}. The notion of bending cocycle was first introduced in the context of hyperbolic geometry by Thurston in \cite{thurstop} and later studied in detail in the Epstein-Marden paper \cite{Epstein}.

The procedure consists of bending a totally geodesic plane in $\mathbb{H}^3$ along a closed set of disjoint geodesics using rotations in $\mathbb{H}^3$. The closed subset corresponds to a lift of a weighted multicurve in $\Sigma$ to a totally geodesic plane $\mathbb{H}^2$ in $\mathbb{H}^3$. A similar construction of a bending cocycle in AdS geometry was initiated by Mess \cite{Mess} and further studied by Benedetti-Bonsante in \cite{canorot}. By combining these two constructions, one can prove that the rescaled holonomy $\tau_t\rho_{(t\lambda,t\mu)}\tau_t^{-1}$ converges to a representation $\rho_{(\lambda,\mu)}^{\HP}:\pi_1(\Sigma)\to \isom(\HP)$.

The remaining part of the proof consists of demonstrating that this representation is indeed the holonomy of a Half-pipe convex core structure with bending data $(\lambda,\mu)$. If the surface $\Sigma$ is closed, the result follows directly from results established in \cite{barbotfillastre} and \cite{affine}. However, when the surface $\Sigma$ has punctures, it becomes necessary to analyze the behavior of the holonomy and the developing map on the Half-pipe convex structure near the punctures. This issue will be addressed in Proposition \ref{half} of Section \ref{TROFHOL}.\\

The second step of the proof consists of studying the convergence of the rescaled pleated surfaces $\tau_t\partial_+\mathrm{CH}(\Lambda_{\rho_{(t\lambda,t\mu)}})$ and $\tau_t\partial_-\mathrm{CH}(\Lambda_{\rho_{(t\lambda,t\mu)}})$. By our construction, the pleated surface $\partial_+\mathrm{CH}(\Lambda_{\rho_{(t\lambda,t\mu)}})$ obtained by bending the totally geodesic plane $\mathbb{H}^2$ of $\mathbb{H}^3$ and $\ads$ will converge after rescaling. However, to prove the convergence of the rescaled lower boundary component $\tau_t\partial_-\mathrm{CH}(\Lambda_{\rho_{(t\lambda,t\mu)}})$, an important additional ingredient is required: we need to control the "distance" between $\partial_+\mathrm{CH}(\Lambda_{\rho_{(t\lambda,t\mu)}})$ and $\partial_-\mathrm{CH}(\Lambda_{\rho_{(t\lambda,t\mu)}})$. Quantitative estimates in this direction are proved by Series in \cite{limitefisch} in hyperbolic geometry and by Seppi in \cite{SEP19} in AdS geometry, using the notion of \textit{width} introduced in \cite{Maximalsurface}. This allows us to prove the convergence of the surfaces $\tau_t\partial_-\mathrm{CH}(\Lambda_{\rho_{(t\lambda,t\mu)}})$.
\begin{theorem}[Transition of pleated surfaces]\label{1.4}
After conjugating if needed, the rescaled pleated surfaces $\tau_t\partial_{\pm}\mathrm{CH}(\Lambda_{\rho_{(t\lambda,t\mu)}})$ converge as $t\to 0$ to the pleated surfaces $\partial_{\pm}\mathrm{CH}(\Lambda_{\rho^{\HP}_{(\lambda,\mu)}})$, where $\rho_{(\lambda,\mu)}^{\HP}$ is the holonomy of a Half-pipe convex core structure on $\Sigma\times[0,1]$ with bending data $(\lambda,\mu)$.
\end{theorem}
Here, by convergence of the pleated surfaces, we mean the convergence of a suitable parametrization $\widetilde{\Sigma}\to \partial_{\pm}\mathrm{CH}(\Lambda_{\rho_{(t\lambda,t\mu)}})$ in the compact-open topology to a parametrization $\widetilde{\Sigma}\to \partial_{\pm}\mathrm{CH}(\Lambda_{\rho^{\HP}_{(\lambda,\mu)}})$, where $\widetilde{\Sigma}$ is the universal cover of $\Sigma$. A more precise statement of Theorem \ref{1.4} is given in Theorem \ref{Sx0}, which is the main result of Section \ref{sec5}.\\

Having established the transition at the level of holonomies and pleated surfaces, the next step is to promote it to convergence at the level of developing maps. More precisely, in Section \ref{TRNDEV}, we prove the following theorem:
\begin{theorem}[Transition of geometric structures]
    Let $\lambda$ and $\mu $ be two weighted multicurves which fill  $\Sigma$ and consider $\rho_{(t\lambda,t\mu)}$ the family of representations as above. Then there is a family of developing maps $\mathrm{Dev}_{(t\lambda,t\mu)}:\widetilde{\Sigma}\times[0,1]\to \X$ where $\X=\mathbb{H}^3$ if $t>0 $ and $\X=\ads$ if $t<0$ such that the convex core structure $(\tau_t\mathrm{Dev}_{(t\lambda,t\mu)},\tau_t\rho_{(t\lambda,t\mu)}\tau_t^{-1})$ converges as $t\to 0$ to the Half-pipe convex core structure  $(\mathrm{Dev}_{(\lambda,\mu)}^{\HP},\rho_{(\lambda,\mu)}^{\HP})$.
\end{theorem}
A naive approach would be to apply the classical Ehresmann-Thurston Principle \cite{thurstop} (see Theorem \ref{ehresman}) to construct a developing map that is equivariant with respect to the holonomy representation $\rho_{(t\lambda,t\mu)}$. However, there are two main obstacles to applying directly such a construction. The first obstacle is that the Ehresmann-Thurston Principle only applies to compact manifolds (possibly with boundary), and our surface $\Sigma\times[0,1]$ is not compact when $\Sigma$ has punctures. The second problem is that the Ehresmann-Thurston Principle does not allow us to control the geometry of the boundaries $\Sigma\times\{1\}$ and $\Sigma\times\{0\}$. In our case, this information is crucial since we ultimately want to glue two copies of the convex core $\mathrm{CH}(\Lambda_{\rho_{(t\lambda,t\mu)}})$ along its boundary. To overcome these problems, the strategy is as follows: using the bending cocycle construction, we can construct two $\rho_{(t\lambda,t\mu)}$-equivariant maps: one from $\widetilde{\Sigma}\times\{1\}$ to $\partial_+\mathrm{CH}(\Lambda_{\rho_{(t\lambda,t\mu)}})$ and the other from $\widetilde{\Sigma}\times\{0\}$ to $\partial_-\mathrm{CH}(\Lambda_{\rho_{(t\lambda,t\mu)}})$. Then, we use a transverse vector field to $\partial\mathrm{CH}(\Lambda_{\rho_{t\lambda,t\mu}})$ to extend these two maps to small neighborhoods $\widetilde{\Sigma}\times[0,\delta]$ and $\widetilde{\Sigma}\times[1-\delta,1]$ for some $\delta>0$. By doing so, we can then apply the Ehresmann-Thurston Principle and classical results on the deformation of geometric structures to obtain a developing map on $\widetilde{\Sigma_{\mathfrak{c}}}\times[0,1]$, where $\Sigma_{\mathfrak{c}}$ is a compact subsurface in $\Sigma$ that is the complement of the union of small neighborhoods around the punctures. As a direct consequence of our methods, we can extend this developing map to $\widetilde{\Sigma}\times[0,1]$, achieving a satisfactory description of the geometry of the cusps.\\

Finally, in Section \ref{sec7}, we complete the proof of the Main Theorem \ref{mainthm} by describing the holonomy and developing map of the doubled convex core structure. This can be expressed explicitly in terms of the holonomy and developing map before doubling.

\subsection{Organization of the paper}
Section \ref{sec2} provides an overview of the different geometric structures discussed in this work. We will recall the geometric transition from hyperbolic to AdS geometry via the Half-pipe geometry. In Section \ref{sec3}, we introduce the concept of a convex core structure. Section \ref{TROFHOL} focuses on the bending map construction, which concludes the proof of Theorem \ref{H}. The proof of the transition of the pleated surfaces is presented in Section \ref{sec5}. In Section \ref{TRNDEV}, we construct the developing map and then prove the transition at the level of the developing map. In Section \ref{sec7}, we recall the notion of a cone singularity and the construction of the double convex core structure and then complete the proof of the main Theorem \ref{mainthm}.

\subsection{Acknowledgments}
This paper owes much to the conversations I had with my PHD advisor, Andrea Seppi. I would like to thank him for his invaluable comments, remarks, and continuous support. His help was crucial in improving the quality of this text with his feedback and suggestions. I am also grateful to Filippo Mazzoli for our fruitful conversations about the subject.

\section{ Geometric transition from \texorpdfstring{$\mathbb{H}^3$}{Lg} to \texorpdfstring{$\ads$}{Lg}}\label{sec2}
In this section we will recall the relevant notions of geometric transition. We will start by recalling the formalism of $(\mathrm{G},\X)$ structures on manifolds.

\subsection{\texorpdfstring{$(\mathrm{G},\X)$}{Lg}-structures}
Let $\X$ be a manifold and $\mathrm{G}$ be a Lie group that acts transitively on $\X$ by analytic diffeomorphisms.
\begin{defi}
A $(\mathrm{G},\X)$-structure on a manifold $M$ is a maximal collection $\{\phi_{i}: U_i\to \X\}$ where $\{U_i\}$ is an open cover of $M$ consisting of connected open sets, and each $\phi_i$ is a homeomorphism onto its image such that each transition map $$\phi_i\circ\phi_j^{-1}:\phi_j(U_i\cap U_j)\to \phi_i(U_i\cap U_j)$$ is the restriction of an element $g_{ij}\in \mathrm{G}.$
\end{defi}
Consider $M$ and $N$ two manifolds endowed with $(\mathrm{G},\X)$-structures and $f: M \to N$ a map. Then $f$ is a \textit{$(\mathrm{G},\X)$-map} if for every charts $(U_i,\phi_i)$, $(V_j,\psi_j)$ for $M$ and $N$ respectively, the composition
$$\psi_j\circ f\circ \phi_i^{-1}\vert_{\phi_i(U_i\cap f^{-1}(V_j))}$$
is the restriction of an element $g\in \mathrm{G}.$
Let $\Pi: \widetilde{M}\to M$ the universal cover of $M$, then there is a canonical $(\mathrm{G},\X)$-structure on $\widetilde{M}$ which makes the covering $\Pi$ a $(\mathrm{G},\X)$-map. An important consequence of the analiticity of the action of $\mathrm{G}$ on $\X$ is the existence of a "global" coordinate on $M$, called a \textit{developing map } $$\mathrm{dev}: \widetilde{M}\to \X, $$ which is a $(\mathrm{G},\X)$-map. It turns out that the map $\mathrm{dev}$ completely determines the $(\mathrm{G},\X)$-structure on $M$. Moreover this map is equivariant with respect to the \textit{holonomy representation} $\mathrm{hol}:\pi_1(M)\to \mathrm{G}$. The pair $(\mathrm{dev},\mathrm{hol})$ is defined up to the action of $\mathrm{G}$, where $\mathrm{G}$ acts by precomposition on the developing map while it acts by conjugacy on the holonomy representation. We finish this section by recalling a fundamental fact in the deformation theory of geometric structures. Before that, we need the following definition.
\begin{defi}
    Let $M$ be a compact manifold possibly with boundary. We say that a family $(\mathrm{dev}_t,\rho_t)$ of $(\mathrm{G},\X)$-structures on $M$ converges to $(\mathrm{dev},\rho)$ in the $\mathcal{C}^k$ topology if we have:
    \begin{itemize}
        \item For all $\gamma$ in $\pi_1(M)$, $\lim_{t \to 0}\rho_t(\gamma)=\rho(\gamma)$ and $t\to\rho_t(\gamma)$ is $\mathcal{C}^k$
        \item $\mathrm{dev}_t$ converges to $\mathrm{dev}$ as $t\to0$ in the $\mathcal{C}^{k}$ topology on any compact subset of $\widetilde{\mathrm{int}(M)}$.
    \end{itemize}
\end{defi}
The following theorem is due to Ehrsemann and Thurston (\cite{ehresman}, \cite{thurstop}). We refer the reader to 
\cite{bergeron} for a more detailed proof.

\begin{theorem}[The Ehresmann-Thurston Principle \cite{thurstop}]\label{ehresman}
Let $M$ be a compact manifold possibly with boundary and $(\mathrm{dev},\rho)$ a $(\mathrm{G},\X)$ structure on $M$. Consider $\rho_t$ a continuous family of representations such that 
$$\lim_{t \to 0}\rho_t(\gamma)=\rho(\gamma),$$ for all $\gamma\in \pi_1(M)$. Then for $t$ small enough, $\rho_t$ is the holonomy of a $(\mathrm{G},\X)$-structure on $M$ given by $(\mathrm{dev}_t,\rho_t)$. Moreover if $\mathrm{dev}$ is a $\mathcal{C}^k$ map, then one can assume that $\mathrm{dev}_t$ converges to $\mathrm{dev}$ in the $\mathcal{C}^k$ topology. 
\end{theorem}

It is worth remarking that Theorem \ref{ehresman} may produce nearby $(\mathrm{G},\X)$-structures on $M$ with different behaviour in the boundary.  

\subsection{Real projective structures}
Real projective structures are an important class of $(\mathrm{G},\X)$-structures.  In this paper, we will be interested in the case where $\X$ is a domain of the projective space $\mathbb{RP}^3$ and $\mathrm{G}$ acts transitively on $\X$ via projective transformations that preserve $\X$. Consider the family of quadratic forms $q_t$ depending on the real parameter $t$ defined on $\mathbb{R}^4$ by: 
\begin{equation}\label{qt}
    q_t(x)=-x_0^2+x_1^2+x_2^2+t\lvert t  \rvert x_3^2
\end{equation}
This family of quadratic forms allows us to define in the next sections the three geometries that interest us.
\subsubsection{\textbf{Hyperbolic structures}}
The projective model of the \textit{hyperbolic space} is given by the negative lines with respect to the quadratic form $q_1$, namely 
$$\mathbb{H}^3 :=\{[x]\in \mathbb{RP}^3, \ q_1(x)<0\}.$$ It is well know that $\mathbb{H}^3$ equipped with the Riemannian metric induced by $q_1$ is the unique complete, simply connected
Riemannian manifold of constant sectional curvature $-1$ up to isometries.  Geodesics lines and totally geodesic planes in $\mathbb{H}^3$ are given by lines and planes in $\mathbb{RP}^3$ that intersect $\mathbb{H}^3$. An example of totally geodesic plane is the \textit{hyperbolic plane} $\mathbb{H}^2$ defined by 
\begin{equation}\label{H2}
    \mathbb{H}^2=\{ [x_0,x_1,x_2,x_3]\in \mathbb{H}^3, x_3=0    \}
\end{equation}
We denote by $\isom(\mathbb{H}^3)$ the group of orientation-preserving isometries of the hyperbolic space $\mathbb{H}^3$, it is identified with the identity component of the group $\mathrm{PO}(1,3)$, where $\mathrm{PO}(1,3)$ is the subgroup of the projective transformations that preserve $\mathbb{H}^3$. The boundary at infinity $\partial\mathbb{H}^3$ of $\mathbb{H}^3$ is given by 
$$\partial\mathbb{H}^3=\{[x]\in \mathbb{RP}^3, \ q_1(x)=0\},$$ which is homeomorphic the sphere $\mathbb{S}^2$. In conclusion we have the following definition.

\begin{defi}
  A hyperbolic structure on a three-manifold $M$ is an $(\isom(\mathbb{H}^3),\mathbb{H}^3)$-structure.  
\end{defi}

\subsubsection{\textbf{Anti-de Sitter structures}}
Anti-de Sitter geometry is the analog of the hyperbolic geometry in Lorentzian geometry. The projective model of \textit{Anti-de Sitter} $3$-space is defined as: $$\ads :=\{[x]\in \mathbb{RP}^3, \ q_{-1}(x)<0\},$$
endowed with the Lorentzian metric induced by the quadratic form $q_{-1}$. The boundary at infinity $\partial\ads$ of $\ads$ is given by 
$$\partial\ads=\{[x]\in \mathbb{RP}^3, \ q_{-1}(x)=0\}.$$
We denote by $\isom(\ads)$ the group of orientation-preserving and time-preserving isometries
of $\ads$. It is identified with the identity component of the group $\mathrm{PO}(2,2)$, where $\mathrm{PO}(2,2)$ is the subgroup of projective transformations that preserve $\ads$.
As for hyperbolic space, geodesics and totally geodesic planes of $\ads$ are obtained as the intersections of lines and planes of $\mathbb{RP}^3$ with $\ads.$ However we distinguish three types of totally geodesic submanifolds in $\ads.$

\begin{defi}
    Let $\mathrm{P}$ be a non-trivial totally geodesic submanifold of $\ads$, namely a geodesic or a plane. Then we say that: 
    \begin{itemize}
    \item $\mathrm{P}$ is \textit{lightlike} if the restriction of the Lorentzian metric to $\mathrm{P}$ is degenerate.
    \item $\mathrm{P}$ is \textit{spacelike} if the restriction of the Lorentzian metric to $\mathrm{P}$ is positive definite.
    \item $\mathrm{P}$ is \textit{timelike} if the restriction of the Lorentzian metric to $\mathrm{P}$ is non-degenerate and not positive definite.
         \end{itemize}
\end{defi}
Let $\mathrm{P}_1$, $\mathrm{P}_2$ two spacelike planes which intersect along a geodesic in $\ads$. For $i=1, 2$ let $\mathrm{H}_i$ be a hyperplane in $\mathbb{R}^4$ such that $\mathrm{P}_i$ is obtained as the intersection of $\ads$ with the projectivization of $\mathrm{H}_i$. Consider $\mathrm{n}_i$ the unit orthogonal of $\mathrm{H}_i$ with respects to the bilinear form $\langle \cdot, \cdot \rangle_{2,2}$ whose associated quadratic form is $q_{-1}.$ Then we define the \textit{angle} between $\mathrm{P}_1$ and $\mathrm{P}_2$ as the non-negative real number $\theta$ satisfying the following equation
\begin{equation}
    \cosh{\theta}=\vert\langle \mathrm{n}_1,\mathrm{n}_2 \rangle_{2,2}  \vert.
\end{equation}
We conclude the preliminaries on the Anti-de Sitter space by the following definition.   
\begin{defi}
An Anti-de Sitter structure on a three-manifold $M$ is an $(\isom(\ads),\ads)$-structure. 
\end{defi}

\subsubsection{\textbf{Half-pipe structures}}\label{halfpipe_structure}
Danciger \cite{danciger_thesis} introduced half-pipe geometry as a transitional geometry between hyperbolic geometry and Anti-de Sitter geometry. \textit{Half-pipe} space is defined as
$$\HP :=\{[x]\in\mathbb{RP}^3, \ q_0(x)<0\}.$$
The boundary at infinity $\partial\HP$ of $\HP$ is given by 
$$\partial\HP=\{[x]\in \mathbb{RP}^3, \ q_{0}(x)=0\}.$$ The Half-pipe space has a natural identification with the dual of Minkowski space, namely the space of spacelike planes of the Minkowski space. Recall that the Minkowski space $\mathbb{R}^{1,2}$ is the vector space $\mathbb{R}^3$ endowed with the Lorentzian metric $\langle,\rangle_{1,2}$ defined by
$$\langle (x_0,x_1,x_2),(y_0,y_1,y_2)\rangle_{1,2} =-x_0y_0+x_1y_1+x_2y_2.        $$
The group $\isom(\mathbb{R}^{1,2} )$ of orientation-preserving and time-preserving isometries of $\mathbb{R}^{1,2}$ is identified with $$\mathrm{O}_{0}(1,2)\ltimes\mathbb{R}^{1,2},$$ where $\mathrm{O}(1,2)$ is the linear transformation that preserve the bilinear form $\langle\cdot,\cdot\rangle_{1,2}$, $\mathrm{O}_0(1,2)$ the identity component of $\mathrm{O}_0(1,2)$, $\mathbb{R}^{1,2}$ acts by translation on itself.
The identification of $\HP$  with the the space of spacelike planes of $\mathbb{R}^{1,2}$ works as follow: for each $[(x,t)]$ in $\HP$ we associate the plane \begin{equation}\label{4}
    \mathrm{P}_{[(x,t)]}=\{  y \in \mathbb{R}^{1,2} : \langle x,y \rangle_{1,2} = t         \}
\end{equation}
The plane $\mathrm{P}_{[(x,t)]}$ is a spacelike plane of $\mathbb{R}^{1,2}$ since its normal vector $x$ is negative for $\langle , \rangle_{1,2}$. 
This gives a diffeomorphism $\HP\to\mathbb{H}^2\times\mathbb{R}$ defined by 
\begin{equation}\label{L}
[(x,t)]\to([x],\mathrm{L}([x,t])),\end{equation}
where $\mathrm{L}([x,t])$ is the signed distance of $\mathrm{P}_{[(x,t)]}$
to the origin along the future normal direction. By an elementary computation one can check that \begin{equation}\label{formule_L}
    \mathrm{L}([x,t])=\frac{t}{\sqrt{-\langle x,x\rangle}}_{1,2}.\end{equation}
Observe that each point $[x,t]$ in $ \partial\mathbb{H}^2\times\mathbb{R}$ corresponds under duality \eqref{4} to a \textit{lightlike} plane in Minokowski space. Therefore $\partial\mathbb{H}^2\times\mathbb{R}$ is identified with  $\partial\HP\setminus\{[0,0,0,1]\}$. Even though the quadratic form $q_0$ defines only a degenerate metric on $\HP$, we will call geodesics (resp. planes) of $\HP$ the intersection of lines (resp. planes) of $\mathbb{RP}^3$ with $\HP.$ We will also use the following terminology
\begin{itemize}
\item[\textbullet] A geodesic in $\HP$ of the form $\{*\}\times \mathbb{R}$ is called a \textit{fiber}.
\item[\textbullet] A geodesic in $\HP$ which is not a fiber is called a \textit{spacelike} geodesic.
\item[\textbullet] A plane in $\HP$ is \textit{spacelike} if it does not contain a fiber.
\end{itemize}
The duality works also in the opposite direction, namely any spacelike plane of $\HP$ corresponds to a point in $\mathbb{R}^{1,2}$. Thus we have the following proposition.
\begin{prop}\cite[Lemma 4.10]{riolo_seppi}\label{hpangle}
    Let $\mathrm{P}_1$, $\mathrm{P}_2$ be two spacelike planes in $\HP$ which correspond dually to $\mathrm{n}_1$ and $\mathrm{n}_2$ in $\mathbb{R}^{1,2}$. Then $\mathrm{P}_1$ and $\mathrm{P}_2$ intersect along a spacelike geodesic in $\HP$ if and only if $\mathrm{n_1}-\mathrm{n}_2$ is a spacelike segment in $\mathbb{R}^{1,2}$.\\ In this situation, we define the angle between $\mathrm{P}_1$ and $\mathrm{P}_2$ as the non-negative real number
    $$\theta=\sqrt{\langle \mathrm{n_1}-\mathrm{n_2}, \mathrm{n_1}-\mathrm{n_2}        \rangle_{1,2}}.$$
\end{prop}
Now, let us denote by $\isom(\HP)$ the group of transformations given by 
$$\begin{bmatrix}
\mathrm{A} & 0  \\
\mathrm{v} & 1 
\end{bmatrix}$$
where $\mathrm{A}\in \mathrm{O}_{0}(1,2)$ and $\mathrm{v}\in \mathbb{R}^2$. Observe that the group $\isom(\HP)$ preserves the orientation of $\HP$ and the oriented fibers. This group has also a natural identification with $\isom(\mathbb{R}^{1,2})$ induced by duality. Indeed any isometry of $\mathbb{R}^{1,2}$ induces a transformation of the space of spacelike planes of $\mathbb{R}^{1,2}$. The isomorphism between the two groups is given by (See \cite[Section 2.8]{riolo_seppi})
\begin{equation}\label{groupeduality}  
\begin{array}{ccccc}
\mathrm{Is}: &  & \isom(\mathbb{R}^{1,2}) & \to & \isom(\HP) \\
 & & (\mathrm{A},\mathrm{v}) & \mapsto & \begin{bmatrix}
\mathrm{A} & 0  \\
^T\mathrm{v}\mathrm{J}\mathrm{A} & 1 
\end{bmatrix}, \\
\end{array}
\end{equation}
where $\mathrm{J}=\mathrm{diag}(-1,1,1).$ Using the affine chart $\{x_0=1\}$, we obtain the \textit{Klein Model} of the Half-pipe space which is identified with the cylinder $\mathbb{D}^2\times\mathbb{R}$, where $\mathbb{D}^2$ is the Klein model of the hyperbolic plane. This identification is given by: \begin{equation}\label{7}
    \begin{array}{ccccc}
&  & \HP & \to & \mathbb{D}^2\times\mathbb{R} \\
 & & [x_0,x_1,x_2,x_3] & \mapsto & (\frac{x_1}{x_0},\frac{x_2}{x_0},\frac{x_3}{x_0}). \\
\end{array}\end{equation}
For $y\in \mathbb{R}^{1,2}$, the spacelike plane in $\HP$ given by $\mathrm{P}_y:=\{[x,t]\in\HP \mid \langle x,y  \rangle_{1,2}=t  \}$ corresponds in the Klein model to the graph of the affine functions over $\mathbb{D}^2$ defined by 
$$\begin{array}{ccccc}
 &  & \mathbb{D}^2 & \to & \mathbb{R} \\
 & & z & \mapsto & \langle y,(1,z)\rangle_{1,2} \\
\end{array}.$$
Note that the group of linear transformations $\mathrm{O}_{0}(1,2)$ acts by isometry on the hyperboloid 
\begin{equation}\mathcal{H}^2:=\{(x_0,x_1,x_2)\in \mathbb{R}^{1,2}\mid -x_0^2+x_1^2+x_2^2=-1, \ x_0>0\}.\end{equation} This hyperboloid is an isometric copy of the hyperbolic plane embedded in Minkowski space. Let us consider the radial projection $\mathrm{Pr}: \mathcal{H}^2\to\mathbb{D}^2$ defined by
\begin{equation}\label{radial}
\mathrm{Pr}(x_0,x_1,x_2)=\left( \frac{x_1}{x_0},\frac{x_2}{x_0}   \right)\end{equation}
For each $z\in \mathbb{D}^2$ and $\mathrm{A}\in \mathrm{O}_0(1,2)$, we will denote by $\mathrm{A}\cdot z$ the image of $z$ by the isometry of $\mathbb{D}^2$ induced by $\mathrm{A}$. More precisely, let $x=\mathrm{Pr}^{-1}(z)$, then $\mathrm{A}\cdot z$ is the unique element in $\mathbb{D}^2$ such that 
$$\mathrm{Pr}(\mathrm{A}x)=\mathrm{A}\cdot\mathrm{Pr}(x).$$
Equivalently, we can check that $\mathrm{A}\cdot z$ satisfies the following equation:
\begin{equation} \label{projecvslinear}\begin{pmatrix}
1 \\
\mathrm{A}\cdot z 
\end{pmatrix}= \frac{\mathrm{A}\begin{pmatrix}
1 \\
z 
\end{pmatrix}}{-\langle \mathrm{A}\begin{pmatrix}
1 \\
z 
\end{pmatrix} ,(1,0,0)\rangle_{1,2}}.  \end{equation}
The following Lemma is an elementary computation about the action of linear transformations and translations on the Klein model of the Half-pipe space. (see \cite[Lemma 2.26]{barbotfillastre}).
\begin{lemma}\label{rotationHP}
    Let $(z, t)\in \mathbb{D}^2\times\mathbb{R}$, $\mathrm{A}\in \mathrm{O}_{0}(1,2)$ and $\mathrm{v}\in \mathbb{R}^{1,2}$. Then the isometry of
Half-pipe space defined by $\mathrm{Is}(\mathrm{A},0)$ and $\mathrm{Is}(\mathrm{Id},\mathrm{v})$ act on the Klein model $\mathbb{D}^2\times\mathbb{R}$ as follows:
$$\mathrm{Is}(\mathrm{A},0)\cdot(z,t)=\left( \mathrm{A}\cdot z, \frac{h}{-\langle \mathrm{A}\begin{pmatrix}
1 \\
z 
\end{pmatrix} ,(1,0,0)\rangle_{1,2}} \right).$$

$$\mathrm{Is}(\mathrm{Id},\mathrm{v})\cdot(z,t)=\left(z,t+ \langle \mathrm{v},(1,z)\rangle_{1,2}    \right).$$
Moreover, if $\mathrm{v}$ is a spacelike vector in $\mathbb{R}^{1,2}$ then $\mathrm{Is}(\mathrm{Id},\mathrm{v})$ is a Half-pipe rotation that fixes pointwise the geodesic $\mathrm{P}_{\mathrm{v}}\cap\mathbb{H}^2$.
\end{lemma}
As Corollary, we obtain the following fact used in Section \ref{TROFHOL}.
\begin{cor}\label{invariantgraph}
    Let $f:\mathbb{D}^2\to\mathbb{R}$ be a function and $\mathrm{A}\in \mathrm{O}_{0}(1,2)$. Then the graph $\mathrm{graph}(f)\subset \mathbb{D}^2\times\mathbb{R}$ is preserved by $\mathrm{Is}(\mathrm{A},0)$ if and only if $f$ satisfies the following $$f(\mathrm{A}\cdot z)=\frac{f(z)}{-\langle \mathrm{A}\begin{pmatrix}
1 \\
z 
\end{pmatrix} ,(1,0,0)\rangle_{1,2}}.$$ 
\end{cor}
We conclude this preliminary discussion by introducing the third type of projective structure that is of interest to us:
\begin{defi}
A Half-pipe structure on a three-manifold $M$ is a $(\isom(\HP),\HP)$-structure.
\end{defi}
\subsection{Geometric transition}\label{GT}
Let us now recall the description of the transition between hyperbolic and Anti-de Sitter geometry through the half-pipe geometry.

We consider the projective transformation $\tau_t$ depending on the real parameter $t\neq0$ defined by: 
\begin{equation}\label{taut}
\tau_t=\begin{bmatrix}
1 & 0 & 0 & 0\\
0 & 1 & 0& 0 \\
0 & 0 & 1 &0 \\
0&0&0&\frac{1}{\vert t\vert}
\end{bmatrix}.\end{equation}
The map $\tau_t$ will be called the \textit{rescaling map}. It satisfies two main properties:
    
\begin{itemize}
    \item[\textbullet] $\tau_t$ fixes pointwise the hyperbolic plane $\mathbb{H}^2$ defined by $x_3=0$ as in \eqref{H2} which is a totally geodesic plane in our three spaces $\mathbb{H}^3$, $\ads$ and $\HP.$
    \item[\textbullet] $\tau_t$ commutes with $\isom(\mathbb{H}^2) $, where $\isom(\mathbb{H}^2)$ is a subgroup of $\isom(\mathbb{H}^3)$, $\isom(\ads)$, $\isom(\HP)$ which is identified with projective transformations of the form
    $$\begin{bmatrix}
\mathrm{A} & 0  \\
0 & 1 
\end{bmatrix},$$
where $\mathrm{A}\in \mathrm{O}_0(1,2)$.
\end{itemize}
One can show that when $t\to0$, the closure of $\tau_t\mathbb{H}^3$ and $\tau_t\ads$ in $\mathbb{RP}^3$ converge to the closure of the half-pipe space $\mathbb{HP}^3$ in the Hausdorff topology. Moreover the groups $\tau_t\mathrm{Isom}(\mathbb{H}^3)\tau_t^{-1}$ and $\tau_t\mathrm{Isom}(\ads)\tau_t^{-1}$ converge in the Chabauty topology to the group $\mathrm{Isom}(\mathbb{HP}^3)$. See \cite{danciger_thesis}, \cite{andreafrancois} for more details. We can now state the transition phenomena that interests us. 
\begin{defi}\cite{danciger_thesis}
 A geometric transition on a three manifold $M$ from hyperbolic to Anti-de Sitter geometry, through half-pipe geometry, is a continuous path of real projective structures $\mathcal{P}_t$ on $M$, defined for $t \in (-\epsilon, \epsilon )$,
which is conjugate to
\begin{enumerate}
\item Hyperbolic structures for $t > 0$;
\item Half-pipe structures for $t = 0$;
\item Anti-de Sitter structures for $t < 0$.
\end{enumerate}
\end{defi}
\begin{remark}
In fact we are interested in the projective structures $\mathcal{P}_t$ which are obtained by this way: we take a family of projective structures $(\mathrm{dev}, \rho_t)$ on $M$ such that  

\begin{enumerate}
    \item For $t>0$, $\mathrm{dev}_t$ takes values in $\mathbb{H}^3$ and $\rho_t $ in $\isom(\mathbb{H}^3)$;
    \item For $t<0$, $\mathrm{dev}_t$ takes values in $\ads$ and $\rho_t$ in $\isom(\ads)$;
    \item When $t\to 0$, the representation $\rho_t$ converges to a representation $\sigma_0$ with value in $\mathrm{Isom}(\mathbb{H}^2)$ and  the developing map $\mathrm{dev}_t$ converges to a submersion $\mathrm{dev}_0$ with values in $\mathbb{H}^2$ which is $\sigma_0$-equivariant.
\end{enumerate}
So if the family $(\tau_{ t}\circ\mathrm{dev}_t, \tau_{t}\rho_t\tau_{t}^{-1})$ converges to a Half-pipe structure $(\mathrm{dev}_0,\rho_0)$, then this produces a geometric transition from $\mathbb{H}^3$ to $\ads$ through $\HP$.
\end{remark}

\subsection{Horospheres and cusps}
In this subsection, we recall the notions of horospheres and cusps. These concepts have a well-established definition in hyperbolic geometry and have also been extended to Anti-de Sitter (AdS) and Half-pipe geometry by Riolo and Seppi in their work \cite{riolo_seppi}. See also \cite[Section 6]{AndreaEnrico} for a detailed exposition on horospheres in pseudo-hyperbolic spaces.

\begin{defi}
A horosphere in $\mathbb{H}^3$ centred at $p\in \partial\mathbb{H}^3$ is a smooth surface $\mathrm{H}$ in $\mathbb{H}^3$ that is orthogonal to all geodesics having the same endpoint $p$. In $\ads$, a horosphere centred at $p\in \partial\ads$ is a smooth timelike surface $\mathrm{H}$ that is orthogonal to all spacelike geodesics with the same endpoint $p$.
\end{defi}

\begin{defi}
A horosphere in $\HP$ is the union of all the fibers passing through a hyperbolic horosphere $\widehat{H}$ contained in a spacelike plane $\mathrm{P}$ in $\HP$.
\end{defi}

Let $\langle\cdot,\cdot \rangle_{1,3}$ (resp. $\langle\cdot,\cdot \rangle_{2,2}$, $\langle,\cdot \rangle_{1,2,0}$) be the bilinear form associated to the quadratic form $q_1$ (resp. $q_{-1}$, $q_0$). Denote by $\mathcal{H}^3$, $\mathcal{A}\mathrm{d}\mathcal{S}^3$ and $\mathcal{HP}^3$ the lift in $\mathbb{R}^4$ of $\mathbb{H}^3$, $\ads$ and $\HP$ respectively, more precisely 
$$\mathcal{H}^3=\{x\in \mathbb{R}^4, \ \mid \ \langle x,x\rangle_{1,3}=-1, \, x_0>0\}.$$
$$\mathcal{A}\mathrm{d}\mathcal{S}^3=\{x\in \mathbb{R}^4, \ \mid \ \langle x,x\rangle_{2,2}=-1\}.$$
$$\mathcal{HP}^3=\{x\in \mathbb{R}^4, \ \mid \ \langle x,x\rangle_{1,2,0}=-1, \, x_0>0\}.$$

The construction of horospheres in different spaces can be described as follows:

\begin{itemize}
\item For a null vector $p$ with respect to $q_{1}$, a horosphere $H$ in hyperbolic space is given by
\begin{equation}
H=\mathbb{P}\left(\{x\in \mathcal{H}^3 \ \mid \ \langle x,p\rangle_{1,3}=a \}\right).\end{equation}

\item For a null vector $p$ with respect to $q_{-1}$, a horosphere $H$ in Anti-de Sitter space is given by
\begin{equation}H=\mathbb{P}\left(\{x\in \mathcal{A}\mathrm{d}\mathcal{S}^3 \ \mid \ \langle x,p\rangle_{2,2}=a \}\right).\end{equation}

\item For a null vector $p$ with respect to $q_0$ that is not collinear to $(0,0,0,1)$ (i.e., the point $p$ corresponds to a lightlike plane in $\mathbb{R}^{1,2}$), a horosphere $H$ in Half-pipe space is given by
\begin{equation}H=\mathbb{P}\left(\{x\in \mathcal{HP}^3 \ \mid \ \langle x,p\rangle_{1,2,0}=a \}\right).\end{equation}
\end{itemize}
where $a$ is a negative real number. We now proceed to define a cusp in our three geometries. Before that, let us recall that for a given horosphere $H$ in $\mathbb{H}^3$ (resp. $\ads$, $\HP$), we denote:
\begin{itemize}
\item $\mathrm{P}^{\mathbb{H}^3}:=\mathrm{Stab}_{\mathbb{H}^3}(H)$ as the subgroup of $\isom(\mathbb{H}^3)$ that stabilizes the horosphere $H$.
\item $\mathrm{P}^{\ads}:=\mathrm{Stab}_{\ads}(H)$ as the subgroup of $\isom(\ads)$ that stabilizes the horosphere $H$.
\item $\mathrm{P}^{\HP}:=\mathrm{Stab}_{\HP}(H)\cap\mathrm{Stab}(p)$ as the subgroup of $\isom(\HP)$ that stabilizes the horosphere $H$ and a point $p\in \overline{H}\setminus H$, where $\overline{H}$ denotes the closure of $H$ in $\mathbb{RP}^3$.
\end{itemize}
It turns out that $\mathrm{P}^{\mathbb{H}^3}$ and $\mathrm{P}^{\ads}$ are isomorphic to the isometry group of the Euclidean and Minkowski planes, respectively. Moreover, $\mathrm{P}^{\HP}$ can be obtained as the limit of $\tau_t\mathrm{P}^{\mathbb{H}^3}\tau_t^{-1}$ and $\tau_t\mathrm{P}^{\ads}\tau_t^{-1}$. For a more detailed explanation, we refer the reader to \cite[Section 3]{riolo_seppi}.

\begin{defi}
Let $p$ be a null vector for either $q_1$, $q_{-1}$, or $q_0$.
\begin{enumerate}
    \item A \textit{cusp} in a hyperbolic manifold is a region isometric to the quotient of $$\mathbb{P}\left(\{x\in \mathcal{H}^3 \ \mid \  \langle x,p\rangle_{1,3}>-1  \}\right)$$ by a subgroup $\Gamma$ of $\mathrm{P}^{\mathbb{H}^3}$ acting freely, properly and co-compactly on $\mathbb{P}\left(\{ \langle x,p\rangle=-1  \}\right)$.
\item A \textit{cusp} in a Anti-de Sitter manifold is a region isometric to the quotient of $$\mathbb{P}\left(\{x\in \mathcal{A}\mathrm{d}\mathcal{S}^3\ \mid \ \langle x,p\rangle_{2,2}>-1  \}\right)$$ by a subgroup $\Gamma$ of $\mathrm{P}^{\ads}$ acting freely, properly and co-compactly on $\mathbb{P}\left(\{ \langle x,p\rangle=-1  \}\right)$.
\item A \textit{cusp} in a Half-pipe manifold is a region isometric to the quotient of $$\mathbb{P}\left(\{ x\in \mathcal{HP}^3 \ \mid\ \langle x,p\rangle_{1,2,0}>-1  \}\right)$$ by a subgroup $\Gamma$ of $\mathrm{P}^{\HP}$ acting freely, properly and co-compactly on $\mathbb{P}\left(\{ \langle x,p\rangle=-1  \}\right)$.
\end{enumerate}
\end{defi}

\section{Convex core structures}\label{sec3}
We fix once and for all an oriented surface $\Sigma$ of negative Euler characteristic homeomorphic to a closed surface with a finite number of points removed. 
\subsection{Preliminaries }
\begin{defi}
A representation $\sigma:\pi_1(\Sigma)\to \mathrm{Isom}(\mathbb{H}^2) $ is $\textit{Fuchsian}$ if $\mathbb{H}^2/\sigma(\pi_1(\Sigma))$ is a hyperbolic surface of finite area. When $\Sigma$ has punctures, we further assume that $\mathbb{H}^2/\sigma(\pi_1(\Sigma))$ is homeomorphic to $\Sigma.$
 
\end{defi}
It is well know that a representation $\sigma$ is Fuchsian if and only if $\sigma$ is a discrete and faithful representation and $\sigma$ sends every loop around punctures to parabolic isometries of $\mathbb{H}^2$. We recall now the definition of the Teichm\"uller space of $\Sigma$.
\begin{defi}
    The \textit{Teichm\"uller space} of $\Sigma$ denoted $\mathcal{T}(\Sigma)$
is the set of Fuchsian representations modulo conjugacy by elements of $\mathrm{Isom}(\mathbb{H}^2)$.
\end{defi}
We move on to defining the notion of weighted multicurves and their length. 
\begin{defi}\label{multicurve}
We say that $\lambda=\sum_{i=1}^{k}a_i\alpha_i$ is a weighted multicurve if $\alpha_i$ 
are homotopy classes of non peripheral, and non-trivial simple closed curves which are pairwise non-homotopic and $a_i>0$ for $0\leq i\leq k$.
\end{defi}
The \textit{support} of $\lambda $, denoted $\vert\lambda\vert$, is the collection of the curves $\alpha_i$. If we further assume that $\alpha_i$ are closed simple geodesics with respect to a complete hyperbolic metric of finite volume on $\Sigma$, then we say that $\lambda$ is a \textit{geodesic weighted multicurve}. 
A weighted multicurve is a particular example of \textit{measured lamination} on $\Sigma$. We will not recall here the general definition of measured laminations since we only need to deal with weighted multicurves. For a detailed description we refer to \cite[Section II.1.11]{Epstein}. We say that two weighted multicurves $\lambda$, $\mu$ fill $\Sigma$ if every component of $\Sigma\setminus \vert\lambda\vert \cup \vert\mu\vert$ contains at most one puncture, and it is simply connected after adding the puncture if needed. Now we state the definition of the length of a weighted multicurve.
 
\begin{defi}\label{length}
Given a weighted multicurve $\lambda=\sum_{i=0}^{k}a_i\alpha_i$ and $h$ a hyperbolic metric on $\Sigma$, the length of $\lambda$ with respect to $h$ is 

$$l_{\lambda}(h)=\sum_{i=0}^{k}a_il_{\alpha_i}(h),$$
where $l_{\alpha_i}(h)$ denotes the length of the $h$-geodesic representative in the homotopy class of $\alpha_i$.
\end{defi}

\subsection{Hyperbolic convex core structures}
Let $\rho:\pi_1(\Sigma) \to \mathrm{Isom}(\mathbb{H}^3)$ be a representation such that $\rho(\pi_1(\Sigma))$ acts freely on $\mathbb{H}^3$. When $\Sigma$ has punctures we assume that $\rho$ sends a loop around punctures to a parabolic isometry. Then the \textit{limit set} $\Lambda_{\rho}$ is the set of accumulation points of the orbits of $\rho(\pi_1(\Sigma))$ in $\mathbb{H}^3$. 
\begin{defi}
A representation $\rho: \pi_1(\Sigma)\to\isom(\mathbb{H}^3)$ is called \textit{$\mathbb{H}^3$-quasi-Fuchsian} if the limit set of $\Lambda_{\rho}$ is a quasi-circle. If $\Sigma$ has punctures, we require that $\mathbb{H}^3/\rho(\pi_1(\Sigma))$ is homeomorphic to $\Sigma\times\mathbb{R}.$
In that case the \textit{convex core} $\mathcal{C}_{\mathbb{H}^3}(\rho)$ of $\rho$ is defined by:
$$\mathcal{C}_{\mathbb{H}^3}(\rho):=\mathrm{CH}(\Lambda_{\rho})/\rho(\pi_1(\Sigma)),$$
where $\mathrm{CH}(\Lambda_{\rho})$ is the convex hull of $\Lambda_{\rho}$ in $\mathbb{H}^3$.
\end{defi}
\begin{remark}
    If $\rho$ is a $\mathbb{H}^3$-quasi-Fuchsian representation then $\mathbb{H}^3/\rho(\pi_1(\Sigma))$ is known as \textit{quasi-Fuchsian hyperbolic manifold}. The convex core $\mathcal{C}_{\mathbb{H}^3}(\rho)$ is the smallest non-empty geodesically convex subset in $\mathbb{H}^3/\rho(\pi_1(\Sigma))$. Note that if we only assume that $\Lambda_{\rho}$ is a quasi-circle, then the quotient $\mathbb{H}^3/\rho(\pi_1(\Sigma))$ is always homeomorphic to the product of a punctured surface and $\mathbb{R}$. However, we need to assume that $\mathbb{H}^3/\rho(\pi_1(\Sigma))$ is homeomorphic to $\Sigma\times\mathbb{R}$ because for punctured surfaces, the fundamental group does not determine the topology of $\Sigma$.
\end{remark}
We say that $\rho$ is \textit{Fuchsian} if $\mathrm{CH}(\Lambda_{\rho})$ is a totally geodesic plane in $\mathbb{H}^3$, this occur precisely when $\rho$ is conjugate in $\isom(\mathbb{H}^3)$ to a Fuchsian representation $\rho_0:\pi_1(\Sigma)\to \isom(\mathbb{H}^2)$. Hence the Teichm\"uller space of $\mathcal{T}(\Sigma)$ can be identified with the Fuchsian representations in $\isom(\mathbb{H}^3)$. If $\rho$ is not Fuchsian, then $\mathrm{CH}(\Lambda_{\rho})$ has non-empty interior in $\mathbb{H}^3$ and its boundary is the disjoint union of two components; the \textit{upper boundary} component $\partial_+\mathrm{CH}(\Lambda_{\rho})$ and the \textit{lower boundary component} $\partial_-\mathrm{CH}(\Lambda_{\rho})$. Each gives in the quotient two surfaces $\partial_{\pm}\mathcal{C}_{\mathbb{H}^3}(\rho):=\partial_{\pm}\mathrm{CH}(\Lambda_{\rho})/\rho(\pi_1(\Sigma))$
homeomorphic to $\Sigma$. Moreover the convex core $\mathcal{C}_{\mathbb{H}^3}(\rho)$ is homeomorphic to $\Sigma\times[0,1]$. This suggest the following definition.  
\begin{defi}
    A \textit{hyperbolic convex core structure} on $\Sigma\times[0,1]$ is a hyperbolic structure on $\Sigma\times[0,1] $ such that the associated developing map $\mathrm{dev}$ and holonomy representation $\rho$ satisfy the following:
    \begin{itemize}
        \item[\textbullet] $\rho:\pi_1(\Sigma)\to\isom(\mathbb{H}^3) $ is a $\mathbb{H}^3$-quasi-Fuchsian representation.
        \item[\textbullet] The developing map $\mathrm{dev}:\widetilde{\Sigma}\times[0,1]\to\mathbb{H}^3$ is a homeomorphism onto $\mathrm{CH}(\Lambda_{\rho})$.
        
    \end{itemize} 
\end{defi}
Now if $\rho$ is not Fuchsian, the geometry of the boundary of the convex core was studied by Thurston in \cite[Chapter 8]{thurstop}. He proved that the components $\partial_{\pm}\mathcal{C}_{\mathbb{H}^3}(\rho)$ are \textit{pleated} surfaces. Namely $\partial_{\pm}\mathrm{CH}(\Lambda_{\rho})$ is the union of totally geodesic pieces which match together to give a complete hyperbolic metric $m_{\pm}$ on  $\partial_{\pm}\mathcal{C}_{\mathbb{H}^3}(\rho)$ which will be called \textit{the induced metric}.
The locus where $\partial_+\mathrm{CH}(\Lambda_{\rho})$ is not totally geodesic defines a geodesic lamination $\lambda_{\pm}$ on $\partial_{\pm}\mathcal{C}_{\mathbb{H}^3}(\rho)\cong \Sigma$. Furthermore these geodesic laminations support a transverse measure called \textit{bending measure}. The description of these measure is simple when the support of $\lambda^{\pm}$ is a collection of finite disjoint non homotopic curves. Indeed in that case, the measure of an arc $c$ transverse to $\lambda^{\pm}$ consists of a sum of the exterior dihedral angles along the leaves
that $c$ meets. For a general description of the bending measure we refer the reader to \cite[Section II.1.11]{Epstein}.

Thurston conjectured that a hyperbolic convex core structure on $\Sigma\times [0,1]$ is uniquely determined by the bending lamination of the boundary. The existence part of this conjecture was proved by Bonahon and Otal. 
\begin{theorem}\cite{bonahonotal}
Let $\lambda$, $\mu$ be two measured geodesic laminations which fill up $\Sigma$ with no closed leaf of weight at least equal to $\pi$. Then there is a hyperbolic convex core structure on $\Sigma\times[0,1]$ for which the 
bending lamination on the upper (resp. lower) boundary component of its convex core is isotopic to $\lambda$ (resp. $\mu$). If $\lambda$, $\mu$ are weighted multicurves, then the hyperbolic convex core structure is unique up to isometry isotopic to the identity.
\end{theorem}
Bonahon proved in \cite{almostfuchsian} that for $t$ small enough the measured laminations $t\lambda$, $t\mu$ can be uniquely obtained as the bending measured laminations of a hyperbolic convex core structure on $\Sigma\times[0,1]$ up to isometry. Series then proved the following degeneration result:

\begin{theorem}[\cite{limitefisch}]\label{series}
For $t>0$ small enough, let $\rho_t$ be the holonomy representation of the unique (up to isometry isotopic to the identity) hyperbolic convex core structure on $\Sigma\times[0,1]$ for which the bending lamination on the upper (resp. lower) boundary component is isotopic $t\lambda$ (resp. $t\mu$). Then after conjugating if needed 
$$ \lim_{t \to 0}\rho_t=k_{\lambda,\mu},$$
where $k_{\lambda,\mu}$ is the Kerckhoff point.
\end{theorem}
Recall that the Kerckhoff point is the unique minimum of the function $l_{\lambda}+l_{\mu}$ over $\mathcal{T}(\Sigma)$ (see \cite{kerker}). 

\subsection{Anti-de Sitter convex core structures}\label{conj}
Let us now move on to define the \textit{Anti-de Sitter convex core structures}, which can be seen as the Lorentzian analogue of hyperbolic convex core structures. Mess \cite{Mess} observed that $\ads$ has a Lie group model, in which $\ads$ is identified with the Lie group $\psl$. This group can be seen as the orientation-preserving isometries of the half plane model of the hyperbolic plane $\mathbb{H}^2$. In this model, the isometry group $\isom(\ads)$ is identified with $\psl\times\psl$ and the boundary at infinity $\partial\ads$ with $\mathbb{RP}^1\times\mathbb{RP}^1$. For a more detailed description of this model, we refer the reader to \cite[Section 3]{adsarticle}.\\
Let us consider $\sigma_l, \sigma_r:\pi_1(\Sigma)\to\isom(\mathbb{H}^2)$
a pair of Fuchsian representations. It is known that there is a unique orientation-preserving homeomorphism $\phi:\mathbb{RP}^1\to \mathbb{RP}^1$ which is $(\sigma_l,\sigma_r)$-equivariant, namely for every $\gamma\in \pi_1(\Sigma)$
\begin{equation} 
\phi\circ\sigma_l(\gamma)=\sigma_r(\gamma)\circ\phi.\end{equation}

\begin{defi} We say that a representation $\rho:\pi_1(\Sigma)\to\isom(\ads)\cong\isom(\mathbb{H}^2)\times \isom(\mathbb{H}^2)$ is an \textit{$\ads$-quasi-Fuchsian representation} if $\rho=(\sigma_l,\sigma_r)$ for some Fuchsian representations $\sigma_l$, $\sigma_r$. Define $\Lambda_{\rho}$ to be the graph in $\mathbb{RP}^1\times\mathbb{RP}^1\cong\partial\ads$ of the unique $(\sigma_l,\sigma_r)$-equivariant
orientation-preserving homeomorphism of $\mathbb{RP}^1$.
Then the \textit{convex core} $\mathcal{C}_{\ads}(\rho)$ of $\rho$ is defined by:
$$\mathcal{C}_{\ads}(\rho):=\mathrm{CH}(\Lambda_{\rho})/\rho(\pi_1(\Sigma)),$$
where $\mathrm{CH}(\Lambda_{\rho})$ is the convex hull of $\Lambda_{\rho}$ in $\ads$.
\end{defi}
\begin{remark}
    Notice that $\ads$ is not convex in $\mathbb{RP}^3$, so the fact that $\mathrm{CH}(\Lambda_{\rho})$ is convex and contained in $\ads$ is not obvious. We refer the reader to \cite[Lemma 5]{Mess} or \cite[Proposition 4.6.1]{adsarticle} for a detailed proof.
\end{remark}
Again, as in hyperbolic geometry, there is a particular case where is $\Lambda_{\rho}$ is the boundary of a totally geodesic copy of $\mathbb{H}^2$. This precisely occurs when $\sigma_l$ and $\sigma_r$ are conjugate in $\isom(\mathbb{H}^2)$ and we say that $\rho$ is a \textit{Fuchsian} representation. Hence the Teichm\"uller space of $\mathcal{T}(\Sigma)$ can be also identified with the Fuchsian representations in $\isom(\ads).$   
If $\mathrm{CH}(\Lambda_{\rho})$  has non empty interior, then its boundary is the disjoint union of two topological disks; the \textit{upper boundary} component $\partial_+\mathrm{CH}(\Lambda_{\rho})$ and the \textit{lower boundary component} $\partial_-\mathrm{CH}(\Lambda_{\rho})$. Each gives in the quotient two surfaces $\partial_{\pm}\mathcal{C}_{\ads}(\rho):=\partial_{\pm}\mathrm{CH}(\Lambda_{\rho})/\rho(\pi_1(\Sigma))$
homeomorphic to $\Sigma$. As in the hyperbolic case, Mess \cite{Mess} (see also \cite{canorot}) proved that the surfaces $\partial_{\pm}\mathcal{C}_{\ads}(\rho)$ are pleated along a \textit{bending lamination} $\lambda_{\pm}$ with hyperbolic \textit{induced metric} $m_{\pm}$.
\begin{remark} In contrast to the hyperbolic case, the representation $\rho$ does not acts properly on $\ads$. However Mess \cite{Mess} (for $\Sigma$ closed) and Barbot \cite{barbot} (for $\Sigma$ possibly with punctures) showed that there is a convex domain $\Omega(\Lambda_{\rho})$ in $\ads$ on which the action of $\rho$ is proper. The quotient $\Omega(\Lambda_{\rho})/\rho(\pi_1(\Sigma))$ is  a particular case of maximal globally hyperbolic Anti-de Sitter manifold. By a theorem of \cite{domain}, global hyperbolicity has a strong consequence on the topology of $\Omega(\Lambda_{\rho})/\rho(\pi_1(\Sigma))$. In fact $\Omega(\Lambda_{\rho})/\rho(\pi_1(\Sigma))$ is homeomorphic to $\Sigma\times \mathbb{R}$. Moreover in that case $\mathcal{C}_{\ads}(\rho)$ is homeomorphic to $\Sigma\times[0,1]$ unless $\rho$ is Fuchsian. 

When the surface $\Sigma$ is closed, Mess \cite{Mess} observed that every maximal globally hyperbolic Anti-de Sitter manifold with Cauchy surface $\Sigma$ is of the form $\Omega(\Lambda_{\rho})/\rho(\pi_1(\Sigma))$. Furthermore the convex core $\mathcal{C}_{\ads}(\rho)$ is dual to $\Omega(\Lambda_{\rho}).$ For a more detailed description of Mess work, we refer the reader to \cite[Section 4]{adsarticle}. It may be true that the Mess's approach might be extended to the case of surfaces with punctures, but we do not consider that question here and this is one of the reasons for choosing the terminology of convex core structure. 
\end{remark}
We give now the following definition.
\begin{defi}
    An \textit{Anti-de Sitter convex core structure} on $\Sigma\times[0,1]$ is an Anti-de Sitter structure on $\Sigma\times[0,1] $ such that the associated developing map $\mathrm{dev}$ and holonomy representation $\rho$ satisfy the following:
    \begin{itemize}
        \item[\textbullet] $\rho:\pi_1(\Sigma)\to\isom(\ads) $ is an $\ads$-quasi-Fuchsian representation.
        \item[\textbullet] The developing map $\mathrm{dev}:\widetilde{\Sigma}\times[0,1]\to\ads$ is a homeomorphism onto $\mathrm{CH}(\Lambda_{\rho})$.
    \end{itemize} 
\end{defi}
Mess asked whether an Anti-de Sitter convex core structure on $\Sigma\times[0,1]$ is uniquely
determined by the bending laminations of the boundary of the convex core. This is the analog of Thurston's conjecture for hyperbolic convex core structures. The existence part of this conjecture was established by Bonsante and Schlenker.
\begin{theorem}[Theorem $1.4$ and Lemma $1.6$]\cite{Fixedpoint}\label{Fixedpoint}
Let $\lambda$, $\mu$ be measured geodesic laminations which fill up $\Sigma$. Then there is an Anti-de Sitter convex core structure on $\Sigma\times[0,1]$ for which the bending lamination on the upper (resp. lower) boundary component is isotopic to $\lambda$ (resp. $\mu$). Moreover there exists $\epsilon>0$ such that the structure is unique for laminations of the form $(t\lambda,t\mu)$ with $t\in(0,\epsilon)$. 
\end{theorem}

\begin{remark}
Bonsante and Schlenker's proof deals with the case of closed surfaces. However, the second part of the theorem concerning the existence and uniqueness of the realization for small bending laminations can be adapted to surfaces with punctures using the same argument.
\end{remark}
 We have also the following Theorem which is similar to Series's Theorem \ref{series} in hyperbolic geometry. 

\begin{theorem}[Lemma $3.6$ in \cite{Fixedpoint}]\label{ads}
For $t>0$ small enough, let $\rho_t$ be the holonomy representation of the unique (up to isometry isotopic to the identity) Anti-de Sitter convex core structure on $\Sigma\times[0,1]$ for which the bending lamination on the upper (resp. lower) boundary component is isotopic to $t\lambda$ (resp. $t\mu$). Then after conjugating if needed 
$$ \lim_{t \to 0}\rho_t=k_{\lambda,\mu},$$
where $k_{\lambda,\mu}$ is the Kerckhoff point.     
\end{theorem}
\subsection{Half-pipe convex core structures}
We finish this section by introducing the half-pipe convex core structures on $\Sigma\times[0,1]$. First we fix a Fuchsian representation $\sigma:\pi_1(\Sigma)\to\isom(\mathbb{H}^2)$. Given a map $\tau:\pi_1(\Sigma)\to \mathbb{R}^{1,2}$, we say that $\tau$ is a $\sigma$\textit{-cocycle} if \begin{equation}\label{affine}
    \tau(\alpha\beta)=\tau(\alpha)+\sigma(\alpha)\cdot\tau(\beta)
\end{equation}
The condition \eqref{affine} allows us to define an isometric action on the Minkowski space given by $(\sigma,\tau)$ where $\sigma$ acts by linear isometry on $\mathbb{R}^{1,2}$ and $\tau$ by translation. Hence by duality \eqref{groupeduality}, it induces a representation $\rho:\pi_1(\Sigma)\to\isom(\HP)$ defined by \begin{equation}\label{affineminkowski}  
\rho(\gamma):=\left [
   \begin{array}{c c c}
      \sigma(\gamma) & 0 \\
      
      ^T\tau(\gamma)\mathrm{J}\sigma(\gamma) & 1\\ 
   \end{array}
\right].
\end{equation}
The representation $\sigma$ is called the \textit{linear} part of $\rho$.
\begin{defi}
    We say that a representation $\rho:\pi_1(\Sigma)\to \isom(\HP) $ is a $\HP$-quasi-Fuchsian representation if there is a continuous function $\phi:\partial\mathbb{D}^2\to\mathbb{R}$ with graph $\Lambda_{\rho}$ invariant by $\rho$. 
\end{defi}
\begin{remark}\label{HPqausi}
    Barbot and Fillastre \cite{barbotfillastre} have shown that when the surface $\Sigma$ is closed, then any affine deformation of a Fuchsian representation defines by duality a $\HP$-quasi-Fuchsian representation. Moreover $\phi$ is unique. When $\Sigma$ has punctures Nie and Seppi showed in \cite{affine} that the same fact holds if and only if the representation $\rho$ given in \eqref{affineminkowski} sends every loop around puncture to a parabolic isometry in $\isom(\HP)$, that is an isometry of $\HP$ which has a fixed point in $\partial\HP$.   
\end{remark}
\begin{defi}
   Let $\rho:\pi_1(\Sigma)\to \isom(\HP) $ be a $\HP$-quasi-Fuchsian representation and $\phi:\partial\mathbb{D}^2\to\mathbb{R}$ the unique continuous function with graph $\Lambda_{\rho}$ invariant by $\rho$. Then we define the \textit{convex core} $\mathcal{C}_{\HP}(\rho)$ of $\rho$  by:
$$\mathcal{C}_{\HP}(\rho):=\mathrm{CH}(\Lambda_{\rho})/\rho(\pi_1(\Sigma)),$$
where $\mathrm{CH}(\Lambda_{\rho})$ is the convex hull of $\Lambda_{\rho}$ in $\HP$.
\end{defi}

A particular case of $\HP$-quasi-Fuchsian representation is given when the affine deformation $\tau$ is conjugated to $0$ through a translation in $\isom(\mathbb{R}^{1,2})$. In that case, one can prove that after conjugation if needed, the function $\phi:=0$ is the unique continuous function with graph invariant under $\rho$. (See \cite[proposition 5.3]{affine} for a proof). Thus $\mathrm{CH}(\Lambda_{\rho})$ is a spacelike plane in $\HP$ isometric to $\mathbb{H}^2$. In all other cases the boundary of $\mathcal{C}_{\HP}(\rho)$ is the union of two surfaces $\partial_{\pm}\mathcal{C}_{\HP}(\rho)$ homeomorphic to $\Sigma$ and bent along measured laminations $\lambda_{\pm}$, and here also $\mathcal{C}_{\HP}(\rho)$ is homeomorphic to $\Sigma\times[0,1]$ (see Lemma \ref{HP3_convexcore}). This yields the following definition.
\begin{defi}
    A \textit{Half-pipe convex core structure} on $\Sigma\times[0,1]$ is a Half-pipe structure on $\Sigma\times[0,1] $ such that the associated developing map $\mathrm{dev}$ and holonomy representation $\rho$ satisfy the following:
    \begin{itemize}
        \item[\textbullet] $\rho:\pi_1(\Sigma)\to\isom(\HP) $ is a $\HP$-quasi-Fuchsian representation.
        \item[\textbullet] The developing map $\mathrm{dev}:\widetilde{\Sigma}\times[0,1]\to\mathbb{HP}^3$ is a homeomorphism onto $\mathrm{CH}(\Lambda_{\rho})$.   
    \end{itemize} 
\end{defi}
\section{Transition of holonomy}\label{TROFHOL}
\subsection{Rotations in hyperbolic space, Anti-de Sitter space and Half-pipe space}
Let us first recall the notion of rotations in $\mathbb{H}^3$, $\ads$ and $\HP.$
\begin{defi}
A rotation in $\mathbb{H}^3$ (resp. in $\ads$ or in $\HP$) is an element of $\isom(\mathbb{H}^3)$ (resp. of $\isom(\ads)$ or $\isom(\HP)$) which fixes pointwise a geodesic in $\mathbb{H}^3$ (resp. in $\ads$ or $\HP$). Here we will only consider rotations that fix a spacelike geodesic in $\ads$ or $\HP.$

\end{defi}
\begin{figure}[htb]
\centering
\includegraphics[width=.8\textwidth]{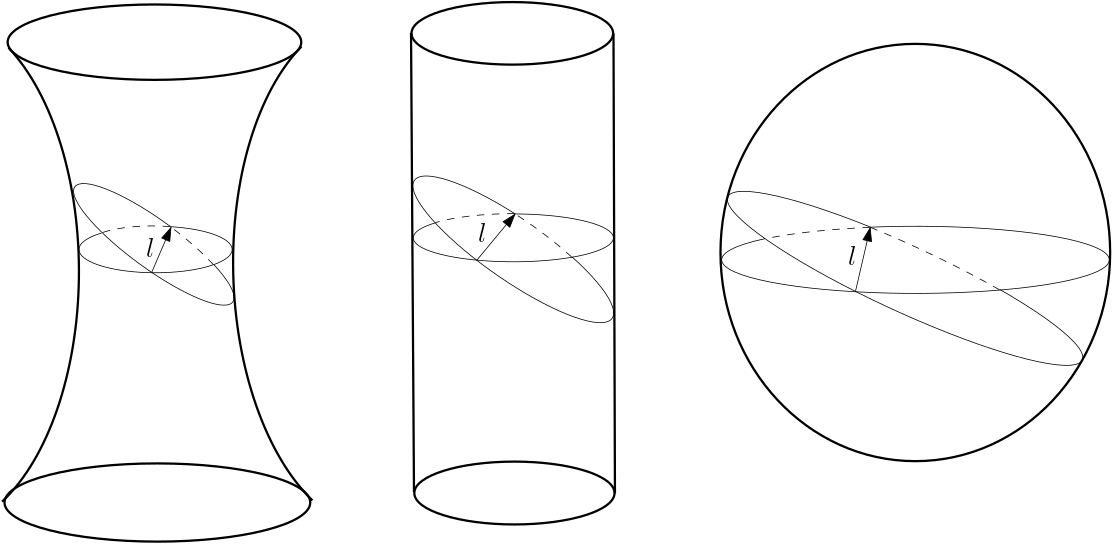}
\caption{The effect of rotations in $\ads$, $\HP$, and $\mathbb{H}^3$ on a plane containing the geodesic fixed by the rotation.}\label{Rotation}
\end{figure}
Fix an orientation on $\mathrm{L}:=\mathrm{Fix}(\mathrm{R})$ which is the geodesic fixed by the rotation $\mathrm{R}$. Consider $\phi $ in $\isom(\mathbb{H}^3)$ (or $\isom(\ads)$, $\isom(\HP)$) such that, as an oriented geodesic 
\begin{equation}   
\phi(\mathrm{L})=\{[\cosh{t},\sinh{t},0,0], \ t\in\mathbb{R}    \},
\end{equation} then we define the \textit{angle} of $\mathrm{R}$ as follows:
\begin{itemize}
    \item In the hyperbolic space: The unique number $\theta\in [-\pi,\pi[$ such that 
    $$\phi\mathrm{R}\phi^{-1}=\begin{bmatrix}
1 & 0 & 0&0 \\
0 & 1 & 0 &0\\
0 & 0 & \cos{\theta}&   \sin{\theta}\\
0&0&   -\sin{\theta}& \cos{\theta}
\end{bmatrix}.$$

 \item In the Anti-de Sitter space: The unique number $\theta\in \mathbb{R}$ such that
    $$\phi\mathrm{R}\phi^{-1}=\begin{bmatrix}
1 & 0 & 0&0 \\
0 & 1 & 0 &0\\
0 & 0 & \cosh{\theta}&   \sinh{\theta}\\
0&0&   \sinh{\theta}& \cosh{\theta}
\end{bmatrix}.$$

 \item In the half-pipe space: The unique number $\theta\in \mathbb{R}$ such that 
 $$\phi\mathrm{R}\phi^{-1}=
 \begin{bmatrix}
1 & 0 & 0&0 \\
0 & 1 & 0 &0\\
0 & 0 & 1&   0\\
0&0&   -\theta & 1
\end{bmatrix}.$$
Note that the last identity represents a half-pipe rotation that fixes the geodesic defined by $x_2=x_3=0$ and sends the spacelike plane $\mathrm{P}$ defined by $x_3=0$ to the spacelike plane $\mathrm{Q}$ defined by $x_3=-\theta x_2$. An elementary computation shows that the angle between $\mathrm{P}$ and $\mathrm{Q}$, in the sense of Proposition \ref{hpangle}, is equal to $\vert \theta\vert.$ 
\end{itemize}

Let us also emphasize the following facts about rotations on $\mathbb{H}^3$, $\ads$ and $\HP$:
\begin{enumerate}
\item The definition of the angle does not depend on the choice of the isometry $\phi$.
    \item The sign of the angle of $\mathrm{R}$ depends not only on $\mathrm{R}$ but also on the choice of an orientation on $\mathrm{L}$.
    \item The rotation $\mathrm{R}$ is uniquely determined by the oriented geodesic $\mathrm{L}$ and the angle of rotation with respect to $\mathrm{L}$.
\end{enumerate}

Rotations in $\HP$ can be considered as infinitesimal rotations on $\mathbb{H}^3$ and $\ads$. More precisely, we have the following property:
\begin{prop}\label{rotation}
Let $l_t$ be a family of oriented geodesics in $\mathbb{H}^2$ such that $\lim_{\vert t\vert \to 0}l_t=l_0 $ where $l_0$ is an oriented geodesic in $\mathbb{H}^2$. Let $\R_t$ be a family of rotations in $\mathbb{H}^3$ (resp. in $\ads$) defined for $t>0$ (resp. $t<0$) of angle $\theta(t)$ around the geodesic $l_t$.  Assume that $\theta$ is a smooth function of $t$ and $\theta(0)=0$. Then

$$\lim_{t \to 0^+} \tau_t \R_t\tau_t^{-1}=\lim_{t \to 0^-} \tau_t \R_t\tau_t^{-1} =\R_0.$$
Where $\R_0$ is the half-pipe rotation of angle $\theta^{'}(0)$ around $l_0$.

\end{prop}
\begin{proof}
By assumption, the family of geodesics $l_t$ converges to $l_0$, so we can choose a family of isometries $\A_{t}\in \isom(\mathbb{H}^2)$ such that $\A_t(l_t)=l_0$ and $\lim_{\vert t\vert\to 0} \A_t=\mathrm{Id}$. Hence $\tau_t\mathrm{A}_t\R_t\mathrm{A}_t^{-1}\tau_t^{-1}$ and $\tau_t\R_t\tau_t^{-1}$ have the same limit since $\tau_t$ commutes with $\isom(\mathbb{H}^2)$.

Let us now focus on the case $t>0$. The isometry $\mathrm{I}_t:=\A_t\R_t\A_t^{-1}$ is the rotation in $\mathbb{H}^3$ with angle $\theta(t)$ around the fixed geodesic $l_0$. Hence we may assume up to the action of the isometry group of $\mathbb{H}^3$ by conjugacy on $\mathrm{I}_t$ that:  
\begin{equation}\label{eq1}
    \mathrm{I_t}=\begin{bmatrix}
1 & 0 & 0&0 \\
0 & 1 & 0 &0\\
0 & 0 & \cos{\theta(t)}&   \sin{\theta(t)}\\
0&0&   -\sin{\theta(t)}& \cos{\theta(t)}
\end{bmatrix}.
\end{equation}By a direct computation one can see that 
$$\lim_{t \to 0^+}\tau_t\mathrm{I}_t\tau_t^{-1}=\begin{bmatrix}
1 & 0 & 0&0 \\
0 & 1 & 0 &0\\
0 & 0 & 1&   0\\
0&0&   -\theta^{'}(0)& 1
\end{bmatrix},$$ which is the half-pipe rotation of angle $\theta^{'}(0)$ around $l_0.$ The case $t<0 $ can be done in similar way by changing the formula \eqref{eq1} by $$\mathrm{I_t}=\begin{bmatrix}
1 & 0 & 0&0 \\
0 & 1 & 0 &0\\
0 & 0 & \cosh{\theta(t)}&   \sinh{\theta(t)}\\
0&0&   \sinh{\theta(t)}& \cosh{\theta(t)}
\end{bmatrix}.$$
\end{proof}
\subsection{Bending cocyle}\label{section4.2}
In this section we describe a recipe to construct $\mathbb{H}^3$/$\ads$/$\HP$-quasi-Fuchsian representations. Throughout the rest of this paper, $\X$ will denote one of the spaces $\mathbb{H}^3$, $\ads$ or $\HP$. We fix $(\mathrm{dev},\sigma)$ a complete hyperbolic structure of finite volume on $\Sigma.$ 
Let $\lambda$ be a weighted multicurve on $\Sigma$. Here we consider $\lambda$ as a geodesic weighted multicurve with respect to the hyperbolic metric $(\mathrm{dev},\sigma)$ and let $\widetilde{\lambda}$ be its lift to $\widetilde{\Sigma}$. Let $x$, $y$ two points in $\widetilde{\Sigma}$ and choose an oriented arc $c$ from $x$ to $y$ transverse to $\widetilde{\lambda}$. Denote by $l_1$, $\cdots l_n$ the geodesics intersecting $c$ among the leaves of $\widetilde{\lambda}$ and let $a_1,\cdots a_n$ be their weights.
For each $k$, denote by $\mathrm{R}^{\X}(a_k,l_k) $ the rotation in $\mathrm{Isom}(\X)$ of angle $a_k$ along the geodesic $\mathrm{dev}(l_k)$ of $\mathbb{H}^2$. Here we orient the geodesic $\mathrm{dev}(l_k)$ in such way that at the intersections points, the orientations of $\mathrm{dev}(c)$ and $\mathrm{dev}(l_i)$ induce the natural orientation of $\mathbb{H}^2$. A simple argument shows that the isometry
$\mathrm{R}^{\X}(a_1,l_1)\circ \cdots \mathrm{R}^{\X}(a_n,l_n)$ depends only on the endpoints $x$, $y$ of $c$. We define
$$\mathrm{B}^{\X}_{\lambda,+}(x,y):=\mathrm{R}^{\X}(a_1,l_1)\circ \cdots \mathrm{R}^{\X}(a_n,l_n).$$
The map $\mathrm{B}^{\X}_{\lambda,+}:\pi_1(\Sigma)\to \isom(\mathrm{X}) $ is the \textit{positive bending cocycle} associated to the weighted multicurve $\lambda$ and the hyperbolic structure $(\mathrm{dev},\rho)$. It is a $\pi_1(\Sigma)$-invariant $\mathrm{Isom}(\mathrm{X})$-valued cocycle. That is
$$ \B^{\X}_{\lambda,+}(x,y)\circ \B^{\X}_{\lambda,+}(y,z)=\B^{\X}_{\lambda,+} (x,z) \ \ \ \ \mathrm{Cocyle \ condition}     $$
$$\B^{\X}_{\lambda,+}(\gamma x,\gamma y)=\sigma(\gamma)\B^{\X}_{\lambda,+}(x,y)\sigma(\gamma)^{-1}\ \ \ \ \mathrm{\pi_1(\Sigma)-invariance},$$ this yields the following proposition. 
\begin{prop}\label{rep}
Let us fix $x_0\in \widetilde{\Sigma}\setminus \widetilde{\lambda}$. Then 
\begin{enumerate}
  \item The map $\rho_{\lambda,+}^{X}: \pi_1(\Sigma)\to \mathrm{Isom}(\X) $, defined by $\rho_{\lambda,+}^{\X}(\gamma)=\B^{\X}_{\lambda,+}(x_0, \gamma x_0)\circ \sigma(\gamma) $ is a representation.
    \item The \textit{positive bending map} $\mathrm{b}_{\lambda,+}^{\X}:\widetilde{\Sigma}\to \X$ defined by 
    $$\mathrm{b}_{\lambda,+}^{\X}(x):=\B^{\X}_{\lambda,+}(x_0,x)(\mathrm{dev}(x)) $$
    is $\pi_1(\Sigma)$-equivariant, that is $\mathrm{b}_{\lambda,+}^{\X}(\gamma x)=\rho_{\lambda,+}^{\X}(\gamma)\mathrm{b}_{\lambda,+}^{\X}(x).$
\end{enumerate}
\end{prop}
\begin{proof}
    Let us check that $\rho_{\lambda,+}^{\X}$ is a homomorphism. We compute
\begin{align*}
   \rho_{\lambda,+}^{\X}(\gamma_1)\rho_{\lambda,+}^{\X}(\gamma_2) & =\B^{\X}_{\lambda,+}(x_0, \gamma_1 x_0)\sigma(\gamma_1)    \B^{\X}_{\lambda,+}(x_0, \gamma_2 x_0)\sigma(\gamma_2)  \\
   & = \B^{\X}_{\lambda,+}(x_0, \gamma_1 x_0)\sigma(\gamma_1)    \B^{\X}_{\lambda,+}(x_0, \gamma_2 x_0)\sigma(\gamma_1)^{-1}\sigma(\gamma_1\gamma_2)  \\
   & =  \B^{\X}_{\lambda,+}(x_0, \gamma_1 x_0)\B^{\X}_{\lambda,+}(\gamma_1 x_0, \gamma_1 \gamma_2 x_0)          \\
   &=   \B^{\X}_{\lambda,+}(x_0, \gamma_1\gamma_2 x_0)=\rho_{\lambda,+}^{\X}(\gamma_1\gamma_2)     
\end{align*}
We used the cocyle condition in the third equation and the $\pi_1(\Sigma)-$invariance in the fourth equation. In the same way we can prove that the bending map is equivariant.
\end{proof}
In the same way, one can define the \textit{negative bending cocycle} $\mathrm{B}^{\X}_{\lambda,-}:\pi_1(\Sigma)\to \isom(\mathrm{X}) $ associated to the weighted multicurve $\lambda$ and the complete hyperbolic structure $(\mathrm{dev},\rho)$ by changing the sign of the rotations:
$$\mathrm{B}^{\X}_{\lambda,-}(x,y):=\mathrm{R}^{\X}(-a_1,l_1)\circ \cdots \mathrm{R}^{\X}(-a_n,l_n).$$
We define also the \textit{negative bending map} $\mathrm{b}_{\lambda,-}^{\X}(x):=\B^{\X}_{\lambda,-}(x_0,x)(\mathrm{dev}(x))$  which is equivariant with respect to the representation $\rho_{\lambda,-}^{X}(\gamma):=\B^{\X}_{\lambda,-}(x_0, \gamma x_0)\circ \sigma(\gamma)$.\\ 
Note that for simplicity we use the notation $\B_{\lambda,\pm}^{\X}$ despite the fact that the cocycle $\B_{\lambda,\pm}^{\X}$.  
depends not only on the weighted multicurve $\lambda$ but also on the hyperbolic structure $(\mathrm{dev}, \sigma)$.

The following proposition gives a relation between the bending cocycle and the holonomy of hyperbolic convex core structures on $\Sigma\times[0,1]$ with small bending lamination.

\begin{prop}\cite[Proposition 2.3]{limitefisch}\label{holHfuch}
Let $\lambda$ be a weighted multicurve and $(\mathrm{dev},\sigma)$ a complete hyperbolic structure of finite volume on $\Sigma$. Consider $\B_{t\lambda,\pm}^{\mathbb{H}^3}$ the associated bending cocyle. Then if $t$ is sufficiently small, the representation $\rho_{t\lambda,\pm}^{\mathbb{H} ^3}: \pi_1(\Sigma)\to \mathrm{Isom}(\mathbb{H}^3)$ is $\mathbb{H}^3-$quasi-Fuchsian. Moreover we have: 
\begin{itemize}
    \item The surface $\partial_{\pm}\mathcal{C}(\rho_{t\lambda,\pm}^{\mathbb{H}^3})$ is pleated along $t\lambda$ and the holonomy of the induced metric is $\sigma.$
    \item The image of the bending map $\mathrm{b}_{t\lambda,\pm}^{\mathbb{H}^3}$ is $\partial_{\pm}\mathrm{CH}(\Lambda_{\rho_{t\lambda,\pm}^{\mathbb{H}^3}})$. 
 
\end{itemize}

\end{prop}
We need also an analogue of Proposition \ref{holHfuch} in the Anti-de Sitter geometry. This is given by Mess \cite{Mess} in the co-compact case and by Benedetti-Bonsante \cite{canorot} in the general case. 
\begin{prop}\cite[Section 6]{canorot}  \label{Mess} Let $\lambda$ be a weighted multicurve and $(\mathrm{dev},\sigma)$ a complete hyperbolic structure of finite volume on $\Sigma$. Consider $\B_{t\lambda,\pm}^{\ads}$ the associated bending cocyle. Then for all $t>0$, the representation $\rho_{t\lambda,\pm}^{\ads}: \pi_1(\Sigma)\to \mathrm{Isom}(\ads)$ is $\ads-$quasi-Fuchsian. Moreover we have: 
\begin{itemize}
    \item The surface $\partial_{\pm}\mathcal{C}(\rho_{t\lambda,\pm}^{\ads})$ is pleated along $t\lambda$ and the holonomy of the induced metric is $\sigma.$
    \item The image of the bending map $\mathrm{b}_{t\lambda,\pm}^{\ads}$ is $\partial_{\pm}\mathrm{CH}(\Lambda_{\rho_{t\lambda,\pm}^{\ads}})$.
\end{itemize}

\begin{remark}
    We emphasize that from now on the component of the upper (resp. lower) boundary component of the convex core of an $\mathbb{H}^3$, $\ads$ or $\HP$-quasi-Fuchsian representation corresponds to a positive (resp.negative) bending map. 
\end{remark}
\end{prop}
Now we will prove a Half-pipe version of Propositions \ref{holHfuch} and \ref{Mess}.
\begin{prop}\label{half}
    Let $\lambda$ be a weighted multicurve and $(\mathrm{dev},\sigma)$ a complete hyperbolic structure of finite volume on $\Sigma$. For $t>0$, let $\B_{\lambda,\pm}^{\HP}$ be the associated bending cocyle. Then 
    \begin{itemize}
        \item The representation $\rho_{\lambda,\pm}^{\HP}$ defined in Proposition \ref{rep} is $\HP-$ quasi-Fuchsian with linear part $\sigma$.
        \item The image of the bending map $\mathrm{b}_{\lambda,\pm}^{\HP}$ is $\partial_{\pm}\mathrm{CH}(\Lambda_{\rho_{\lambda,\pm}^{\HP}})$.
    \end{itemize}
\end{prop}
The proof of Proposition \ref{half} is already known in the case where $\Sigma$ is a closed surface. This was observed in the work of Barbot and Fillastre \cite{barbotfillastre}. They proved that any representation of $\pi_1(\Sigma)$ in $\isom(\HP)$ with Fuchsian linear part is a $\HP$-quasi-Fuchsian.   
For the case where $\Sigma$ has punctures, the proof of Proposition \ref{half} needs some preparation. Therefore, the remaining part of this section will be devoted to proving Proposition \ref{half}. We only focus on the positive bending map since the other can be proved in the same way. To do so, some relevant notions should be recalled. First let us recall that for a bounded convex (resp. concave) function $\phi:\mathbb{D}^2\to\mathbb{R}$, we can define the \textit{boundary value} of $\phi$ to be the extension of $\phi$ to the unit circle $\mathbb{S}^1=\partial\mathbb{D}^2$ by the formula:

$$\phi(z) = \lim_{s \to 0^+}\phi((1-s)z+sx),$$ for some $x\in \mathbb{D}^2$. We note that this is independent of the choice of $x\in \mathbb{D}^2$. Moreover the boundary value of $\phi$ is lower (resp. upper) semi continuous. We will need the following basic fact in convex analysis:
\begin{prop}\cite[Proposition 4.2]{ConvexAnal}\label{boundaryvalueextension}
    Let $\phi: \mathbb{D}^2\to\mathbb{R}$ be a convex function (or concave). Then the boundary value of $\phi$ is a continuous function on $\partial{\mathbb{D}^2}$ if and only if $\phi$ has a continuous extension to $\overline{\mathbb{D}^2}$.   
\end{prop}
For a more detailed exposition on convex analysis, we refer the reader to \cite{Rockconvex} or \cite[section 4.1]{ConvexAnal}. The next Lemma provides a sufficient condition for a real function defined over $\mathbb{D}^2$ to be extended to $\overline{\mathbb{D}^2}$. 
 
\begin{lemma}\label{4.8}
Let $\rho:\pi_1(\Sigma)\to\isom(\HP)$ be a representation whose linear part is a Fuchsian representation $\sigma$. Assume that $\phi: \mathbb{D}^2\to \mathbb{R}$ is a $\mathcal{C}^2$ function such that 
\begin{itemize}
    \item The graph of $\phi$ is invariant by $\rho(\pi_1(\Sigma))$. 
    \item  For each puncture of $\Sigma$, there is  a neighborhood $V$ such that the restriction of $\phi$ to any lift $\widetilde{V}$ in $\widetilde{\Sigma}\cong \mathbb{D}^2$ of $V$ is an affine map.
\end{itemize}
Then $\phi$ extends to a continuous map on $\overline{\mathbb{D}^2}$.
\end{lemma}

\begin{proof}
The proof is inspired by Proposition $5.2$ in \cite{affine}.
Let $U$ be the union of pairwise disjoint punctured disks such that the restriction of $\phi$ to any connected component of $\widetilde{U}$ in $\mathbb{D}^2$ is an affine map. We take $K\subset \mathbb{D}^2\setminus \widetilde{U}$ to be a compact set such that $$\mathbb{D}^2\setminus \widetilde{U}=\bigcup_{\gamma\in \pi_1(\Sigma)}\sigma(\gamma)\cdot K.$$ Consider the function  $\omega_{\mathbb{D}^2}(x)=-\sqrt{1-x^2}$ which is smooth, strictly convex and vanishes on the boundary of the disk. One can prove that the graph of $\omega_{\mathbb{H}^2}$ is invariant by $\isom(\mathbb{H}^2)$ \cite[section 3.2]{affine}. By compactness of $K$, we can take a sufficiently large constant $C>0$ such that the smooth functions
$$\phi_- := \phi + C\omega_{\mathbb{D}^2}, \ \ \phi_+=\phi - C\omega_{\mathbb{D}^2}$$
are strictly convex and strictly concave in $K$, respectively. Moreover it is not hard to check that the graphs of $\phi_-$ and $\phi_+$ are invariant by $\rho(\pi_1(\Sigma))$ \cite[Lemma 2.8]{affine}. Since $\phi$ restricts to an affine map in $\widetilde{U} $, then $\phi_-$ (resp. $\phi_+ $) are strictly convex (resp. strictly concave) on $\mathbb{D}^2$. It is easy to see that the boundary values of $\phi_-$ and $\phi_+$ coincide, this follow from the fact that $\omega_{\mathbb{D}^2}$ vanishes in the boundary and $\phi_+-\phi_-=2C\omega_{\mathbb{D}^2}.$

Hence the common boundary value of $\phi_-$ and $\phi_-$ is a continuous function of $\mathbb{S}^1$, as it is both lower and upper semicontinuous. Therefore $\phi_+$ and $\phi_-$ are continuous functions on $\overline{\mathbb{D}^2}$ by Proposition \ref{boundaryvalueextension}. This implies that $\phi$ extends continuously to $\overline{\mathbb{D}^2}$.
\end{proof} 
Let's now come back to our representation $\rho_{\lambda,+}^{\HP}$. From this point and until the proof of Proposition \ref{half}, we will use the Klein model of the Half-pipe space. Let $(\sigma,\mathrm{dev})$ be a complete hyperbolic metric of finite volume on $\Sigma$.
We consider the positive bending map $\mathrm{b}_{\lambda,+}^{\HP}:\widetilde{\Sigma}\to \HP$ defined by 
    \begin{equation}\mathrm{b}_{\lambda,+}^{\HP}(y)=\B_{\lambda,+}^{\HP}(y_0,y) \label{forme_de_rotation}(\mathrm{dev}(y)),
    \end{equation}
where $y_0$ is a fixed point on $\widetilde{\Sigma}\setminus\widetilde{\lambda}$. Here the map $\mathrm{dev}$ takes value in $\mathbb{D}^2$.  By construction, $\B_{\lambda,+}^{\HP}$ is a composition of Half-pipe rotations. Since rotation in $\HP$ has the form $\mathrm{Is}(\mathrm{Id},\mathrm{v})$ for some $\mathrm{v}\in \mathbb{R}^{1,2}$, then the map $\mathrm{b}_{\lambda,+}^{\HP}$ has the form $(\mathrm{dev},\phi_{\lambda})$ for some function $\phi_{\lambda}:\widetilde{\Sigma}\to \mathbb{R}$. Consider the function $\psi_{\lambda}$ defined on $\mathbb{D}^2$ by 
\begin{equation}
    \psi_{\lambda}(z):=\phi_{\lambda}\circ \mathrm{dev}^{-1}(z)
\end{equation}
The concrete description of $\psi_{\lambda} $ is as follows:
Consider $\overline{\mathrm{dev}}:\widetilde{\Sigma}\to\mathcal{H}^2$ defined by $\overline{\mathrm{dev}}=\mathrm{Pr}^{-1}\circ\mathrm{dev}$ where $\mathrm{Pr}$ is the radial projection defined in \eqref{radial}.

Let $x_0:=\overline{\mathrm{dev}}(y_0)$ and pick $z$ in $ \mathbb{D}^2$. Denote by $x=\mathrm{Pr}^{-1}(z)\in\mathcal{H}^2$. Now consider the
oriented geodesic interval $[x_0, x]$ in $\mathcal{H}^2$. We define a map $\eta : \mathcal{G}[x_0, x] \to \mathbb{R}^{1,2}$ where $\mathcal{G}[x_0, x]$ is the space of geodesics in $\mathcal{H}^2$ intersecting $[x_0,x]$. This map assigns to each geodesic $l$ in $\mathcal{G}[x_0, x]$  the corresponding point in $d\mathbb{S}^2$, namely, the spacelike unit vector in $\mathbb{R}^{1,2}$ orthogonal to $l$ for the Minkowski product, pointing outward with respect to
the direction from $x_0$ to $x$. Assume that $l_1,\cdots l_n$ are the images by $\overline{\mathrm{dev}}$ of the set of the leaves of $\widetilde{\lambda} $ that intersect the oriented segment $[x_0, x]$, then 
\begin{equation}\label{psil}
\psi_{\lambda}(z)=\sum_{i=1}^{n}-\lambda(l_i)\langle \eta(l_i),(1,z)\rangle_{1,2}
\end{equation}

By the choice of the orientation that we have made on $\eta(l_k)$, one can prove that $\psi_{\lambda}$ is a concave function on $\mathbb{D}^2$ (further details can be found in \cite[Section 3.5.1]{canorot}). Additionally  $\psi_{\lambda}$ is smooth on $\mathbb{D}^2\setminus \mathrm{dev}(\widetilde{\lambda})$ because on each connected component of $\mathbb{D}^2\setminus \mathrm{dev}(\widetilde{\lambda})$, $\psi_{\lambda}$ is the restriction of an affine map defined over $\mathbb{R}^2$. However $\psi_{\lambda}$ is not $\mathcal{C}^1$ in $\mathbb{D}^2$ and therefore Lemma \ref{4.8} cannot be applied directly to extend the map $\psi_{\lambda}$ to $\mathbb{S}^1$. The next Lemma treats this issue by smoothing the map $\psi_{\lambda}$ in a neighborhood of $\widetilde{\lambda}$.
\begin{lemma}\label{4.9}
Let $\lambda$ be a weighted multicurve on $\Sigma$ and consider the representation $\rho_{\lambda,+}^{\HP}:\pi_1(\Sigma)\to\isom(\HP)$ (see Proposition \ref{rep}). 
Then there exists a smooth map $\widetilde{\psi_{\lambda}}:\mathbb{D}^2\to \mathbb{R}$ such that 
\begin{itemize}
    \item The graph of $\widetilde{\psi_{\lambda}}$ is invariant by $\rho_{\lambda,+}^{\HP}(\pi_1(\Sigma))$.
    \item  For each puncture of $\Sigma$, there is  a neighborhood $V$ such that the restriction of $\widetilde{\psi_{\lambda}}$ to any lift $\widetilde{V}$ in $\widetilde{\Sigma}\cong \mathbb{D}^2$ of $V$ is an affine map. 
\end{itemize} 
 In particular by Lemma \ref{4.8}, the map $\widetilde{\psi_{\lambda}}$ has a continuous extension to $\mathbb{S}^1$. Moreover the boundary value of the concave map $\psi_{\lambda}$ coincides with the extension of $\widetilde{\psi_{\lambda}}$ to $\mathbb{S}^1$.  
\end{lemma}
\begin{proof}
Let $\alpha_1, \cdots \alpha_n$ be the support of the weighted multicurve $\lambda$. We can find $\epsilon>0$ such that the $\epsilon-$neighborhoods $U_{\alpha_k}(\epsilon)$ of each curve $\alpha_i$ are pairwise disjoint. Let $W$ be the union of $U_{\alpha_i}(\epsilon)$ and let $\widetilde{W}$ be its lift to $\mathbb{D}^2$. We then define $\widetilde{\psi_{\lambda}}$ as follows (see Figure \ref{Smoothing}):
\begin{itemize}
\item For $z\in \mathbb{D}^2 \setminus \widetilde{W}$, we define $\widetilde{\psi_{\lambda}}(z)$ to be equal to $\psi_{\lambda}(z)$. 
\item For $z\in \widetilde{U_{\alpha_i}(\epsilon)}$, Assume that $l_1,\cdots l_{n_i}$ is the images by $\overline{\mathrm{dev}}$ of the set of the leaves of $\widetilde{\lambda} $ that intersect the oriented segment $[x_0, x]$ with $x=\mathrm{Pr}^{-1}(z)$, then: 

$$\psi_{\lambda}(z)=\sum_{j=1}^{n_i}-\lambda(l_j)\langle \eta(l_j),(1,z)\rangle_{1,2}.$$
We choose a smooth real function $f:\mathbb{R}\to[0,1]$ such that
$$\begin{cases}
f=0 &\text{ if}\ \vert t\vert <\frac{\epsilon}{2}\\
f=1 &\text{ if }\ \vert t\vert \geq\epsilon
\end{cases}$$
Then we  associate the following map:

\begin{equation}\label{lissage}     
\widetilde{\psi_{\lambda}}(z):=-\langle \sum_{j=1}^{n_i-1}\lambda(l_j) \eta(l_j)+f(d_{\mathbb{H}^2}(z,l_{n_i}))\lambda(l_{n_i})\eta(l_{n_i}),(1,z)\rangle_{1,2}.\end{equation}
\end{itemize}
Notice that $\widetilde{\psi_{\lambda}}$ coincides with $\psi_{\lambda}$ outside $U_{l_j}(\epsilon)$, where $U_{l_j}(\epsilon)$ is the $\epsilon-$neighborhood of $\mathrm{Pr}(l_j)$ in $\mathbb{D}^2$. Therefore, by our choice of $f$, it is not difficult to see that $\widetilde{\psi_{\lambda}}$ is smooth on $\mathbb{D}^2$.

\begin{figure}[htb]
\centering
\includegraphics[width=.7\textwidth]{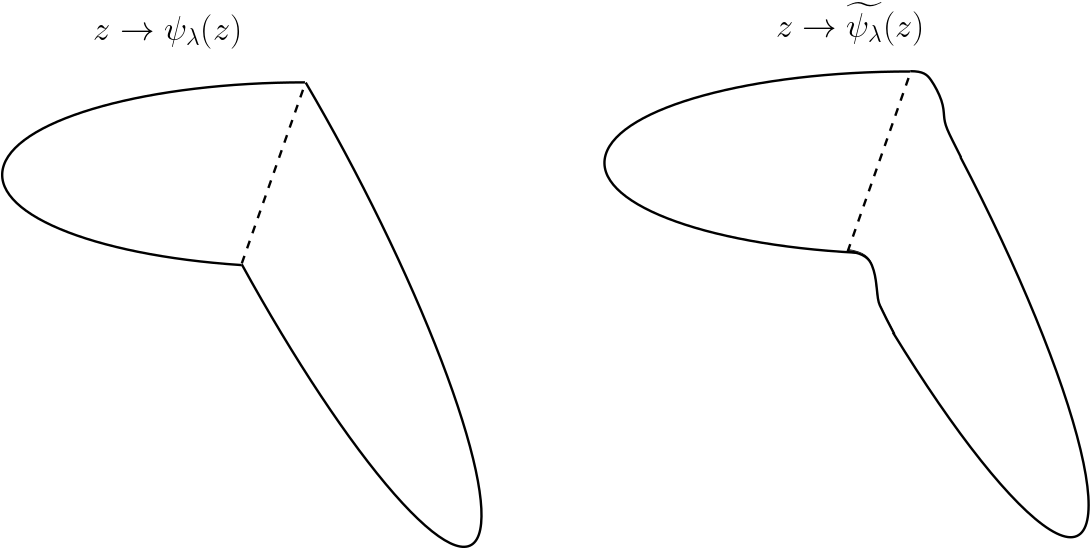}
\caption{In a neighborhood of a bending line, the left picture illustrates the graph of the function $\psi_{\lambda}$, while the right picture illustrates the graph of the function $\widetilde{\psi_{\lambda}}$. Outside a neighborhood of the bending lines, the two functions coincide.}\label{Smoothing}
\end{figure}

Now we will show that the graph of $\widetilde{\psi_{\lambda}}$ is preserved under the action of $\rho_{\lambda,+}^{\HP}(\pi_{1}(\Sigma))$. Since the graph of $\psi_{\lambda}$ invariant by such action, then it is equivalent to show that the graph of $\widetilde{\psi_{\lambda}}-\psi_{\lambda}$ is preserved by $\sigma(\pi_1(\Sigma))$, where $\sigma$ is the affine part of $\rho_{\lambda,+}^{\HP}$ (see \cite[Lemma 2.8]{affine}). Since for all $z$ in $U_{l_j}(\epsilon)$ and $g$ in $\sigma(\pi_1(\Sigma))$, we have $gU_{l_j}(\epsilon)=U_{g(l_j)}(\epsilon)$. Then from Corollary \ref{invariantgraph}, it is enough to show that for all $z\in U_{l_j}(\epsilon)$ and $g\in\sigma(\pi_1(\Sigma))$ we have $$(\widetilde{\psi_{\lambda}}-\psi_{\lambda})(g\cdot z)=\frac{-1}{\langle g (1,z)^T ,(1,0,0)\rangle_{1,2}}(\widetilde{\psi_{\lambda}}-\psi_{\lambda})(z).$$ Now, we compute
\begin{align*}
(\widetilde{\psi_{\lambda}}-\psi_{\lambda})(g\cdot z)&=f(d_{\mathbb{H}^2}(g z,l_j))\langle \lambda(l_j)\eta(l_j),(1,g \cdot z)\rangle_{1,2}\\
&=\frac{-1}{\langle g (1,z)^T ,(1,0,0)\rangle_{1,2}}f(d_{\mathbb{H}^2}(g z,l_k))\langle \lambda(l_j)\eta(l_j),g(1, z)^T\rangle_{1,2}\\
&=\frac{-1}{\langle g (1,z)^T ,(1,0,0)\rangle_{1,2}}f(d_{\mathbb{H}^2}(z,g^{-1}(l_j)))\langle \lambda(g^{-1}l_j)\eta(g^{-1}l_j),(1,z)\rangle_{1,2}\\
&=\frac{-1}{\langle g (1,z)^T ,(1,0,0)\rangle_{1,2}}f(d_{\mathbb{H}^2}(z,l_j))\langle \lambda(l_j)\eta(l_j),(1,z)\rangle_{1,2}\\
&=\frac{-1}{\langle g (1,z)^T ,(1,0,0)\rangle_{1,2}}(\widetilde{\psi_{\lambda}}-\psi_{\lambda})(z).
\end{align*}
We used in the second equality the relation between the linear action $g(1,z)^T$ and the projective action $g\cdot z$ established in \eqref{projecvslinear}. In the third equality, we used the fact that $\sigma$ acts by isometry on $\mathbb{H}^2$ and the map $\eta$ which associates to each geodesic of $\mathcal{H}^2$ its orthogonal unit vector is equivariant under the action of $\mathrm{O}_{0}(1,2)$. \\

The only part remaining to prove the Lemma is to show that the boundary value of the concave function $\psi_{\lambda}$ coincides with the continuous extension of $\widetilde{\psi_{\lambda}}$ to $\mathbb{S}^1$. Note that such extension exists because by construction, $\widetilde{\psi_{\lambda}}$ satisfies the assumptions of Lemma \ref{4.8}. We will distinguish two situations, see Figure \ref{emaray}.
\begin{enumerate}
    \item If $z$ is an endpoint of a geodesic $\mathrm{Pr}(l_j)$, then we observe that for any point $w$ lying on the line $\mathrm{Pr}(l_j)\subset\mathbb{D}^2$, we have $\widetilde{\psi_{\lambda}}(w) = \psi_{\lambda}(w)$ because $\langle\eta(l_{j}),(1,w)\rangle_{1,2}=0$, for $w\in\mathrm{Pr}(l_j)$. This implies that the limit of $\psi_{\lambda}((1-s)z+sw)$ as $s\to0^+$ coincides with $\widetilde{\psi_{\lambda}}(z)$.
    \item If $z$ is not an endpoint of any geodesic $\mathrm{Pr}(l_j)$, then we can find a point $w$ in $\mathbb{D}^2$ such that the interval $[w,z[$ is disjoint from $U_{l_j}(\epsilon)$. Consequently, the boundary value of $\psi_{\lambda}$ at $z$ coincides with $\widetilde{\psi_{\lambda}}(z)$ because $\psi_{\lambda}$ and $\widetilde{\psi_{\lambda}}$ coincide on the $[w,z[$.
\end{enumerate}\end{proof}

\begin{figure}[htb]
\centering
\includegraphics[width=.8\textwidth]{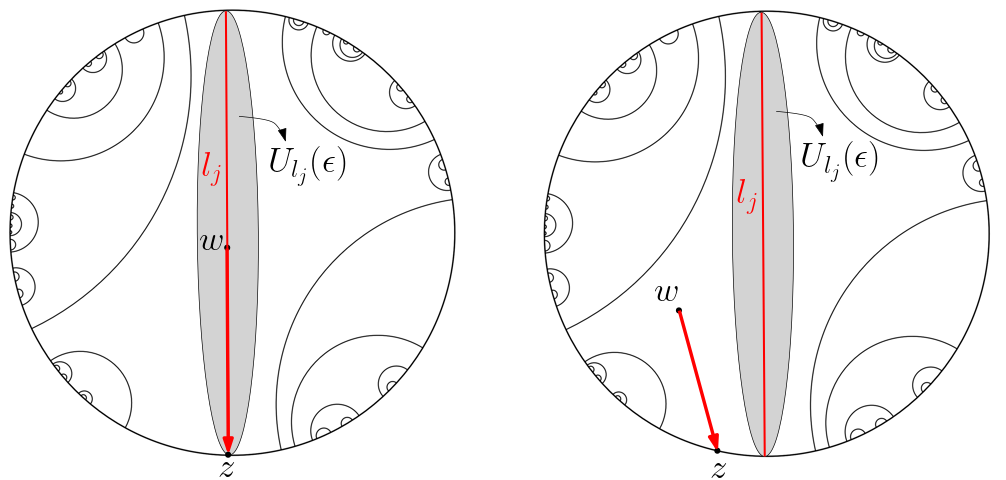}
\caption{The proof that the boundary value $\psi_{\lambda}$ coincides with the continuous extension of $\widetilde{\psi_{\lambda}}$ to $\mathbb{S}^1$ depends on whether $z$ is an endpoint of the lift of the lamination $\lambda$ or not. On the left, if $z$ is an endpoint of a geodesic $l_j$ of $\widetilde{\lambda}$, then we extend $\psi_{\lambda}$ through the interval $[w,z[$ where $w$ is a point on $l_j$. If not, on the right picture, we choose $w$ such that $[w,z[$ does not intersect any $\epsilon-$neighborhood of the lift of $\lambda$.}\label{emaray}
\end{figure}

As corollary of Lemma \ref{4.9} we get the following.
\begin{cor}\label{exx}
The map $\psi_{\lambda}:\mathbb{D}^2\to \mathbb{R}$ defined in \eqref{psil} extends continuously to a function on $\mathbb{S}^1$ with graph invariant by $\rho_{\lambda,+}^{\HP}(\pi(\Sigma)).$
\end{cor}
\begin{remark}\label{pleated}
    It is important to remark that since the concave function $\psi_{\lambda}$ extends to $\mathbb{S}^1$ and it is affine on each connected component of $\mathbb{D}^2\setminus\mathrm{dev}(\widetilde{\lambda})$, then
    $\psi_{\lambda}$ satisfies the following property: 
    For each $z_0$ in $\mathbb{D}^2$, we take $a$ an affine map such that $\psi_{\lambda}(z_0)=a(z_0)$. Then 
    $$\{z\in \mathbb{D}^2, \ a(z)=\psi_{\lambda}(z)\}=\mathrm{Convex}\ \mathrm{hull}\{z\in \mathbb{S}^1,\ a(z)=\psi_{\lambda}(z)\}.$$
\end{remark}
We can now prove Proposition \ref{half}.
\begin{proof}[Proof of Proposition \ref{half}]
By Corollary \ref{exx}, the graph of $\psi_{\lambda}:\mathbb{S}^1\to \mathbb{R}$ is invariant by $\rho_{\lambda}^{\HP}$,
hence the representation $\rho_{\lambda,+}^{\HP}$ is a $\HP$-quasi-Fuchsian representation, this concludes the proof of first part of the proposition.\\
Now we want to prove that the image of the bending map $\mathrm{b}_{\lambda,+}^{\HP}$ is $\partial_+\mathrm{CH}(\Lambda_{\rho_{\lambda,+}^{\HP}})$. For $z\in \mathbb{D}^2$ we define the function
$$ \mathrm{H}(z)=\inf\{  a(z), \ \psi_{\lambda}\leq a\ \mathrm{on } \ \mathbb{S}^1\}.$$
It is proven in \cite[Lemma 2.41]{barbotfillastre} that the graph of the concave function $\mathrm{H}$ is $\partial_+\mathrm{CH}(\Lambda_{\rho_{\lambda,+}^{\HP}})$. Therefore it is enough to prove that 

\begin{equation}
    \mathrm{H}=\psi_{\lambda}.
\end{equation}
First, since $\psi_{\lambda}$ is a concave function, then $\mathrm{H}\leq \psi_{\lambda}$ because $\mathrm{H}$ is pointwise no greater than any concave function with boundary value $\mathrm{\psi}_{\lambda}$ (see \cite[Corollary 4.5]{ConvexAnal}). Now we claim that $\psi_{\lambda}\leq\mathrm{H}$. Suppose by contradiction that $\psi_{\lambda}(z_0)>\mathrm{H}(z_0)$ for some $z_0\in \mathbb{H}^2$. We take $a$ an affine map such that $\psi_{\lambda}(z_0)=a(z_0)$. Then by Remark \ref{pleated}
$$\{z\in \mathbb{D}^2, \ a(z)=\psi_{\lambda}(z)\}=\mathrm{Convex}\ \mathrm{hull}\{z\in \mathbb{S}^1,\ a(z)=\psi_{\lambda}(z)\}$$
But by construction, $\psi_{\lambda}$ and $\mathrm{H}$ agree on $\mathbb{S}^1$, so $z_0$ is contained in the convex hull of $\{z\in \mathbb{S}^1,\ a(z)=\mathrm{H}(z)\} $ which is equal to $\{z\in \mathbb{D}^2, \ a(z)=\mathrm{H}(z) \}$ by \cite[Lemma 4.9]{ConvexAnal}, hence $\psi_{\lambda}(z_0)=\mathrm{H}(z_0)$ contradicting the assumption $\psi_{\lambda}(z_0)>\mathrm{H}(z_0)$. Therefore, we conclude that $\psi_{\lambda}=\mathrm{H}$, and so the image of the positive bending map $\mathrm{b}_{\lambda,+}^{\HP}$ which is the the graph of $\psi_{\lambda}$ is exactly $\partial_+\mathrm{CH}(\Lambda_{\rho_{\lambda,+}^{\HP}})$, as desired.\end{proof}

\subsection{Proof of Theorem \ref{H} (Transition of holonomy) }
The aim of this section is to prove the following Theorem.

\begin{theorem}[Transition of holonomy]\label{H}
  Let $\lambda$, $\mu$ be two weighted multicurves which fill  $\Sigma$. Consider $\rho_{(t\lambda,t\mu)}:\pi_1(\Sigma)\to\isom(\X)$, a family of representations such that 
  \begin{itemize}
      \item For $t>0$, $\rho_{(t\lambda,t\mu)}$ is the  holonomy representation of the hyperbolic convex core structure on $\Sigma\times[0,1]$ for which the bending lamination on the upper (resp. lower) boundary component is $\vert t\vert\lambda$ (resp. $\vert t\vert \mu$).
      \item For $t<0$, $\rho_{(t\lambda,t\mu)}$ is the  holonomy representation of the Anti-de Sitter convex core structure on $\Sigma\times[0,1]$ for which the bending lamination on the upper (resp. lower) boundary component is $\vert t\vert\lambda$ (resp. $\vert t\vert \mu$).
  \end{itemize}
Then after conjugating if needed we have:
  \begin{enumerate}
      \item \begin{equation}\label{reshol}
          \lim_{t \to 0^+} \tau_t \rho_{(t\lambda,t\mu)}\tau_t^{-1}=\lim_{t \to 0^-} \tau_t \rho_{(t\lambda, t\mu)}\tau_t^{-1}.\end{equation}
          \item The common limit in \eqref{reshol} is a $\HP$-quasi-Fuchsian representation whose linear part is given by the Kerckhoff point $k_{\lambda,\mu}$ and the upper boundary component is pleated along $\lambda.$
  \end{enumerate}
\end{theorem}
Before proving Theorem \ref{H}, we need to fix some notations. Denote by $(\mathrm{dev}^+_{(t\lambda,t\mu)},\sigma^+_{(t\lambda,t\mu)})$ the family of complete hyperbolic structures induced on the upper boundary component of the convex core of $\rho_{(t\lambda,t\mu)}.$ 

\begin{proof}[Proof of Theorem \ref{H}]
By Theorem \ref{series} and \ref{ads}, the family $\rho_{(t\lambda,t\mu)}$ converges to $k_{\lambda,\mu}$ up to conjugacy. Therefore, the family $\sigma_{(t\lambda,t\mu)}^+$ converges also to $k_{\lambda,\mu}$. This follows from the fact that the map which associates a hyperbolic or Anti-de Sitter convex core structure the hyperbolic structure on the boundary of the convex core is continuous. As consequence, by the Ehresmann-Thurston Principle \ref{ehresman}, we may assume that the family $\mathrm{dev}^+_{(t\lambda,t\mu)}$ converges uniformly to the developing map associated to $k_{\lambda,\mu}$ that we will note $\mathrm{dev}_{(\lambda,\mu)}^0$. Next by Proposition \ref{holHfuch} and \ref{Mess}, we necessary have for $t$ small enough 
\begin{equation}\label{holonomyformula}  
\rho_{(t\lambda,t\mu)}(\gamma)=
\mathrm{B}^{\X}_{\vert t\vert \lambda,+}(x_0,\gamma x_0)\circ \sigma_{(t\lambda,t\mu)}^+(\gamma).\end{equation} Here the cocycles are constructed from the hyperbolic structure $(\mathrm{dev}_{(t\lambda, t\mu)}^+, \sigma_{(t\lambda,t\mu)}^+)$. Now for each loop $\gamma$, we consider  $l_1^t,\cdots l_n^t$ the geodesics in the support of $\widetilde{\lambda}$ meeting the segment $[x_0,\gamma\cdot x_0]$. Denoting by $\R^{\X}(\vert t\vert a_i,l_i^t)$ the rotation in $\X$ of angle $\vert t\vert a_i$ along the geodesic $\mathrm{dev}^+_{(t\lambda,t\mu)}(l_i^t)$ ($\vert t\vert a_i$ is the weight of $l_i^t$), since $\tau_t$ commutes with $\isom(\mathbb{H}^2)$, we obtain
$$\tau_t \rho_{(t\lambda,t\mu)}\tau_t^{-1}=
\tau_t       \R^{\X}(\vert t\vert a_1,l_1^t)\tau_t^{-1}\circ \cdots   \tau_t \R^{\X}(\vert t\vert a_p,l_p^t)         \tau_t^{-1}\sigma_{(t\lambda,t\mu)}^+(\gamma).$$
Hence it is enough to show that \begin{equation}\label{6}
    \lim_{t \to 0^+}\tau_t \R^{\mathbb{H}^3 }(\vert t\vert  a_i,l_i^t)\tau_t^{-1}=\lim_{t \to 0^-}\tau_t \R^{\ads }(\vert t\vert a_i,l_i^t)\tau_t^{-1}
    \end{equation}
Note that $\mathrm{dev}^+_{(t\lambda,t\mu)}(l_i^t)$ is the axis of the hyperbolic isometry $\sigma_{(t\lambda,t\mu)}^+(\gamma_i)$ for some loop $\gamma_i$ in the support of the weighted multicurve $\lambda$. Since $(\mathrm{dev}_{(t\lambda, t\mu)}^+, \sigma_{(t\lambda,t\mu)}^+)$ converges to $(\mathrm{dev}_{(\lambda, \mu)}^0, k_{\lambda,\mu})$, then $\mathrm{dev}^+_{(t\lambda,t\mu)}(l_i^t)$ converges to the geodesic $\mathrm{dev}^0_{(\lambda,\mu})(l_i^0)$ which is the axis of the hyperbolic isometry $k_{\lambda,\mu}(\gamma_i)$. This concludes the proof of the identity \eqref{6} by Proposition \ref{rotation} and so the first item of the statement. The second item of the statement follows directly from Proposition \ref{half}.\end{proof}

\section{Transition of the convex core}\label{sec5}
We have seen in the previous section that we have a transition at the level of representations. The goal of this section is to prove that the rescaled convex core of both hyperbolic and Anti-de Sitter convex core structures converges to a convex core structure in half pipe geometry.
 This will be done in $2$ steps:
\begin{enumerate}
    \item 
    First we study the convergence of the upper component of the convex core. This is straightforward since by our normalisation the upper component corresponds to the image of the bending map. See Section \ref{6.1} below.

    \item In Section \ref{6.2}, we turn to the convergence of the other boundary component. Here we use some estimates about the width of the convex core obtained by \cite{limitefisch} in the hyperbolic geometry and \cite{SEP19} in the Anti-de Sitter geometry.
    
\end{enumerate}
\begin{notation*}
From now on, we will use the same notation of Theorem \ref{H}, that is:
\begin{itemize}\label{notation}
    \item $\rho_{(t\lambda,t\mu)}$ is the  holonomy representation of the hyperbolic/Anti-de Sitter convex core structure on $\Sigma\times[0,1]$ for which the bending lamination on the upper (resp. lower) boundary component is $\vert t\vert\lambda$ (resp. $\vert t\vert \mu$).
\item $(\mathrm{dev}_{(t\lambda, t\mu)}^{\pm}, \sigma_{(t\lambda,t\mu)}^{\pm})$ is the hyperbolic structure on $\Sigma$ induced on the upper or lower components of $\mathcal{C}_{\X}(\rho_{(t\lambda,t\mu)})$.
\item $(\mathrm{dev}_{(\lambda, \mu)}^0, k_{\lambda,\mu})$ is the hyperbolic structure on $\Sigma$ which is the limit of $(\mathrm{dev}_{(t\lambda, t\mu)}^{\pm}, \sigma_{(t\lambda,t\mu)}^{\pm})$ as $t$ goes to $0.$ 
\item Up to conjugacy by $\isom(\X)$, the representation $\rho_{(t\lambda,t\mu)}$ is given by:
\begin{equation}\label{holonomyformulaa}  
\rho_{(t\lambda,t\mu)}(\gamma)=
\mathrm{B}^{\X}_{\vert t\vert \lambda,+}(x_0,\gamma x_0)\circ \sigma_{(t\lambda,t\mu)}^+(\gamma),\end{equation} where 
$\mathrm{B}^{\X}_{\vert t\vert \lambda,+}$ is the positive bending cocycle associated to the complete hyperbolic structure $(\mathrm{dev}_{(t\lambda, t\mu)}^+, \sigma_{(t\lambda,t\mu)}^+)$.
\item Finally, we denote by $\rho_{(\lambda,\mu)}^{\HP}$ the $\HP$-quasi-Fuchsian representation given by: $$\rho_{(\lambda,\mu)}^{\HP}:=\lim_{t \to 0} \tau_t \rho_{(t\lambda,t\mu)}\tau_t^{-1}.$$

\end{itemize}
\end{notation*}
\subsection{Transition of the upper boundary component}\label{6.1}
The main ingredient to prove that the rescaled limit of the upper boundary component of $\mathcal{C}_X(\rho_{(t\lambda,t\mu)})$ converges is the following Proposition.
\begin{prop}\label{Sx1}
Let $\mathrm{b}_{\vert t\vert\lambda,+}^{\X}$ and $\mathrm{b}_{\lambda,+}^{\HP}$ be the positive bending maps associated to the complete hyperbolic structures $(\mathrm{dev}_{(t\lambda, t\mu)}^+, \sigma_{(t\lambda,t\mu)}^+)$ and $(\mathrm{dev}_{(\lambda, \mu)}^0, k_{\lambda,\mu})$ respectively. Then when $t$ goes to $0$,  $\tau_t\mathrm{b}_{\vert t\vert\lambda,+}^{\X}$ converges uniformly to $\mathrm{b}_{\lambda,+}^{\HP}$ on compact sets of $\widetilde{\Sigma}$.

\end{prop}

\begin{proof}
Let $x_t$ be a family of points in $\widetilde{\Sigma}$ that converges to $x\in \widetilde{\Sigma}$ as $t\to 0$. We claim that $\mathrm{b}_{\vert t\vert \lambda,+}^{\X}(x_t)$ converges to $\mathrm{b}_{\lambda,+}^{\HP}(x)$ and hence we get the uniform convergence on compact set of $\widetilde{\Sigma}$. To prove this, we first observe that since $\mathrm{dev}_{(t\lambda, t\mu)}^+$ converges uniformly to $\mathrm{dev}_{(\lambda, \mu)}^0$ on a compact sets of $\widetilde{\Sigma}$, it follows that $\mathrm{dev}_{(t\lambda, t\mu)}^+(x_t)$ converges to $\mathrm{dev}_{(\lambda, \mu)}^0(x)$. We consider $l_1^t,\dots,l_n^t$ to be the oriented geodesics in the support of $\widetilde{\lambda}$ that intersect the segment $[x_0,x_t]$ (the orientation is explained in Section \ref{section4.2}). Denoting by $\R^{\X}(\vert t\vert a_i,l_i^t)$ the rotation in $\X$ along the geodesic $\mathrm{dev}^+_{(t\lambda,t\mu)}(l_i^t)$ of angle equal to the weight $\vert t\vert a_i$ of $l_i^t$, we also consider $l_1^0,\cdots l_n^0$ the leaves of $\widetilde{\lambda}$ such that $\mathrm{dev}^+_{(t\lambda,t\mu)}(l_i^t)$ converges to $\mathrm{dev}^+_{(\lambda,\mu)}(l_i^0)$. In particular, $l_1^0,\cdots l_n^0$ are leaves that intersect the segment $[x_0,x]$. However, we may have a situation where $x$ lies on some other leaf of $\widetilde{\lambda}$ that we will denote $l_{n+1}^0$. Thus, we distinguish two cases.

\begin{itemize}
    \item If $x$ is not in a new leaf of $\widetilde{\lambda}$, then for $t$ small enough $\B^{\X}_{\vert t\vert\lambda,+}(x_0,x_t)=\B^{\X}_{\vert t\vert\lambda,+}(x_0,x)$, since $\tau_t\mathrm{dev}^+_{(t\lambda,t\mu)}=\mathrm{dev}^+_{(t\lambda,t\mu)}$ then we have
$$\tau_t\mathrm{b}_{\vert t\vert\lambda,+}^{\X}(x_t)=\tau_t \B^{\X}_{\vert t\vert\lambda,+}(x_0,x)\tau_t^{-1}\circ\mathrm{dev}_{(t\lambda,t\mu)}^+(x_t).$$
A computation similar to that in the proof of Theorem \ref{H} shows that  $$ \lim_{t \to 0} \tau_t \B_{\vert t\vert \lambda,+}^{\X}(x_0,x)\tau_t^{-1}=\B_{\lambda,+}^{\HP}(x_0,x),$$ hence $\tau_t\mathrm{b}_{\vert t\vert\lambda,+}^{\X}(x_t)\to \mathrm{b}_{\lambda,+}^{\HP}(x).$
\item If $x$ is on the leaf $l_{n+1}^0$, then for $t$ small enough we have 
$$\tau_t \mathrm{b}_{\vert t\vert \lambda,+}^{\X}(x_t)=
\tau_t       \R^{\X}(\vert t\vert a_1,l_1^t)\tau_t^{-1}\circ \cdots   \tau_t \R^{\X}(\vert t\vert a_n,l_n^t)  \tau_t^{-1}\mathrm{dev}_{(t\lambda,t\mu)}^+(x_t).$$
Passing to the limit we obtain
$$\lim_{t \to 0}\tau_t \mathrm{b}_{\vert t\vert \lambda,+}^{\X}(x_t)= \R^{\HP}(a_1,l_1^0)\circ \cdots  \R^{\HP}( a_n,l_n^0) \mathrm{dev}_{(\lambda,\mu)}^0(x).$$
In the other hand, we have 
$$\mathrm{b}_{\lambda,+}^{\HP}(x)=\R^{\HP}(a_1,l_1^0)\circ \cdots  \R^{\HP}( a_{n+1},l_{n+1}^0) \mathrm{dev}_{(\lambda,\mu)}^0(x).$$
But since $x\in l_{n+1}^0$, then $\R^{\HP}(a_{n+1},l_{n+1}^0)\mathrm{dev}_{\lambda,\mu}^0(x)=\mathrm{dev}_{\lambda,\mu}^0(x).$ 
As consequence $$ \lim_{t \to 0} \tau_t \mathrm{b}_{\vert t\vert\lambda,+}^{\X}(x_t) =\R^{\HP}(a_1,l_1^0)\circ \cdots  \R^{\HP}( a_{n+1},l_{n+1}^0) \mathrm{dev}_{(\lambda,\mu)}^0(x)= \mathrm{b}_{\lambda,+}^{\HP}(x),$$ which concludes the proof.
\end{itemize}\end{proof}
The Proposition \ref{Sx1} implies in particular that the pleated surface $\partial_+\mathrm{CH}(\Lambda_{\rho_{(t\lambda,t\mu)}})$ converges after rescaling by $\tau_t$ to the pleated surface $\partial_+\mathrm{CH}(\Lambda_{\rho_{(\lambda,\mu)}^{\HP}})$ since $\mathrm{b}_{\vert t\vert\lambda,+}^{\X}(\widetilde{\Sigma})=\partial_+\mathrm{CH}(\Lambda_{\rho_{ (t\lambda,t\mu)}})$.

\subsection{Transition of the lower boundary component}\label{6.2}
Let's start this section by recalling some results on the width of convex core in hyperbolic and Anti-de Sitter geometry.

\subsubsection{Width of convex core in Anti-de Sitter space.}
Let $\phi:\mathbb{RP}^1\to \mathbb{RP} ^1$ an orientation preserving homeomorphism of $\mathbb{RP}^1$. Then its graph $\Lambda_{\phi}$ in contained in $\mathbb{RP} ^1\times \mathbb{RP}^1\cong\partial\ads$. Consider $\mathrm{CH}(\Lambda_{\phi})$ the convex hull of $\Lambda_{\phi}$ in $\ads$, then the \textit{width} of $\mathrm{CH}(\Lambda_{\phi})$ is defined as 

$$w(\mathrm{CH}(\Lambda_{\phi})):=\underset{x \in \partial_- \mathrm{CH}(\Lambda_{\phi}), \ y\in \partial_+\mathrm{CH}(\Lambda_{\phi})}{\sup} d_{\ads}(x,y),$$
where $\partial_- \mathrm{CH}(\Lambda_{\phi})$ (resp. $\partial_+ \mathrm{CH}(\Lambda_{\phi})$) denote the two connected components of $\partial \mathrm{CH}(\Lambda_{\phi})$ and $d_{\ads}(x,y)$ is the supremum of the length of timelike paths containing $x $ and $y$. 
In \cite{SEP19}, Seppi gives an upper bound of the width of the convex hull, which only depends on the cross ratio norm of $\phi.$ Recall that given an orientation-preserving homeomorphism $\phi: \mathbb{RP}^1\to \mathbb{RP}^1$, The \textit{cross- ratio  norm} is defined as 
 $$\lVert  \phi \rVert=\underset{\mathrm{cr}(Q)=1\ \ \ \ \ \ \ \ \ \ \ \ \ }{\sup \vert\ln{\mathrm{cr}(h(Q))}    \vert },$$ where $Q=[a,b,c,d]$ is a quadruple of points on $\mathbb{RP}^1$ and $$\mathrm{cr}(Q)=\frac{  (b-a)(d-c) }{ (c-b)(d-a)   }$$ is the cross ratio of $Q$. We say that $\phi$ is \textit{quasisymmetric} if $\lVert \phi \rVert$ is finite. We have the following estimate.
\begin{theorem}[Proposition 3.A \cite{SEP19}]\label{widthads}
Given any quasisymmetric homeomorphism $\phi$ of  $\mathbb{RP}^1$, let $w(\mathrm{CH}(\Lambda_{\phi}))$ be the width
of the convex hull of the graph of $\phi$. Then
$$w(\mathrm{CH}(\Lambda_{\phi}))\leq \arctan\left(\sinh{\frac{  \lVert \phi\rVert_{cr}  }{2}}\right). $$
\end{theorem}
By Thurston's earthquake theorem  any quasisymmetric homeomorphism is the extension to $\mathbb{RP}^1$ of an \textit{earthquake} $\mathrm{E}^{\lambda}$ with $\lambda$ bounded measured lamination (see \cite{Thurston}, \cite{Saric}). Here bounded means with respects to the Thurston norm.  

\begin{defi}
Given a measured geodesic lamination $\lambda$ on $\mathbb{H}^2$, the \textit{Thurston norm} of $\lambda$ is defined as:

$$     \lVert \lambda\rVert_{\mathrm{Th}}:=\underset{I \ \ \ \ \ \  }{\sup\lambda(I)},          $$
where $I$ varies over all geodesic segments of length $1$ transverse to the
geodesic lamination $\lambda.$
\end{defi}
It turn out that the cross-ratio distortion norm and Thurston norm are equivalent.
\begin{theorem}[\cite{earthnorm}]\label{JU}
There exists a universal constant $C>0$ such that for any quasisymmetric homeomorphism $\phi:\mathbb{RP}^1\to\mathbb{RP}^1 $,
$$\frac{1}{C}\lVert \phi \rVert\leq \lVert \lambda \rVert_{\mathrm{Th}}\leq C \lVert  \phi\rVert,$$ where $\phi=\mathrm{E}^{\lambda}|_{\mathbb{RP}^1}.$
\end{theorem}
For an $\ads$- quasi-Fuchsian representation $\rho=(\sigma_l,\sigma_r)$, we define the width of $\mathcal{C}_{\ads}(\rho)$ as the width of the convex core of the unique homeomorphism $\phi$ which conjugates $\sigma_l$ and $\sigma_r$ (see Section \ref{conj}). Now combining Theorem \ref{JU} and \ref{widthads} we get 

\begin{cor}\label{Fwidth}
Let $\lambda$ and $\mu$ two weighted multicurves which fill $\Sigma$. Then there exists $C>0$ (the universal constant $C$ in Theorem \ref{JU}) such that the width of $\mathcal{C}_{\ads}(\rho_{(t\lambda,t\mu)})$ satisfies the following. 
$$w(\mathcal{C}_{\ads}(\rho_{(t\lambda,t\mu)}))\leq\arctan\left( \sinh{\frac{\vert t\vert C\lVert \lambda \rVert_{\mathrm{Th}}}{2}}\right).$$
In particular, there is $C^{'}, \epsilon>0$ depending on $\lambda$ such that if $\vert t\vert\leq\epsilon$, then

$$w(\mathcal{C}_{\ads}(\rho_{(t\lambda,t\mu)}))\leq C^{'}\vert t\vert .$$

 \end{cor}

\subsubsection{Width of convex core in hyperbolic space.}
Given a Jordan curve $\Lambda$ in $\partial\mathbb{H}^3$, let $\mathrm{CH}(\Lambda)$ denote the convex hull of $\Lambda$ in $\mathbb{H}^3$. Denote by $\partial_+\mathrm{CH}(\Lambda)$ (resp. $\partial_-\mathrm{CH}(\Lambda)$) the upper (resp. lower) boundary components of $\partial\mathrm{CH}(\Lambda)$. Then the \textit{width} of $\mathrm{CH}(\Lambda)$ is defined in \cite{Bonsante2020QuasicirclesAW} as:

$$w(\mathrm{CH}(\Lambda)):=\max (\underset{x\in \partial_+\mathrm{CH}(\Lambda) \hspace{1.9cm} }{\sup d_{\mathbb{H}^3}(x,\partial_-\mathrm{CH}(\Lambda))  }, \ \underset{x\in \partial_-\mathrm{CH}(\Lambda)\hspace{1.9cm}  }{\sup d_{\mathbb{H}^3}(x,\partial_+\mathrm{CH}(\Lambda))  }). $$ For a $\mathbb{H}^3$-quasi-Fuchsian representation $\rho$, we define the width of $\mathcal{C}_{\mathbb{H}^3}(\rho)$ as the width of the convex core of the limit set of $\rho.$ We have the following estimate which is not exactly about the width of $\mathcal{C}_{\mathbb{H}^3}(\rho)$ but it is sufficient for the purpose of the paper.
\begin{prop}[Corollary 6.10 of \cite{limitefisch}]\label{suppH}
Let $\lambda$ and $\mu$ be two weighted multicurves which fill $\Sigma$. Then there exists $C>0$ such that: For all $\gamma\in\vert\lambda \vert\cup\vert\mu\vert$, let $\widetilde{\gamma^+}$ be the lift of $\gamma$ to $\partial_+\mathrm{CH}(\Lambda_{\rho_{(t\lambda,t\mu)}})$. Then for all $x\in\widetilde{\gamma^+}$:

$$d_{\mathbb{H}^3}(x,\partial_- \mathrm{CH}(\Lambda_{\rho_{(t\lambda,t\mu)}}))<C\vert t\vert. $$
In particular there is $y_t\in \partial_-\mathrm{CH}(\Lambda_{\rho_{(t\lambda,t\mu)}})$ such that $d_{\mathbb{H}^3}(x,y_t)<C\vert t\vert.$ 
\end{prop}
We now turn to the transition of the lower boundary component of $\mathrm{CH}(\Lambda_{\rho_{(t\lambda,t\mu)}})$. Before moving to the proof, we need the following definition: Let $\mathrm{C}$ be a convex subset of an affine space. We say that a plane $\mathrm{P}$ is a \textit{support plane} of $\mathrm{C}$ (at $x\in \partial \mathrm{C}$) if $\mathrm{P}$ contains $x$, with the property that all of $\mathrm{C}$ is contained in one of the two closed half-spaces bounded by $\mathrm{P}$. We will freely use the following two basic facts: Let $\mathrm{P}_t$ be a family of affine planes and $x_t$ be a family of points in $\mathrm{P}_t$, then
\begin{itemize}
\item If $x_t$ is bounded, then $\mathrm{P}_t$ converges, up to a subsequence to an affine plane $\mathrm{P}$.
\item If the family $\mathrm{P}_t$ is disjoint from a plane $\mathrm{P}$ and $x_t$ converges to a point $x$ in $\mathrm{P}$, then $\mathrm{P}_t$ necessarily converges to $\mathrm{P}$.
\end{itemize}
The same terminology will be used for our three projective geometries. More precisely we have:
\begin{defi}\label{extremal}
Let $\mathrm{C}$ be a convex subset of $\mathbb{H}^3$ (resp. in $\ads$ or $\HP$). In the case of $\ads$, we additionally assume that $\mathrm{C}$ is contained in an affine chart $\ads\setminus \mathrm{Q}$, where $\mathrm{Q}$ is a spacelike plane in $\ads$. We say that a totally geodesic plane $\mathrm{P}$ in $\mathbb{H}^3$ (resp. a spacelike plane in $\ads$ or $\HP$) is a \textit{support plane} of $\mathrm{C}$ (at $x\in \partial \mathrm{C}$) if $\mathrm{P}$ contains $x$, with the property that all of $\mathrm{C}$ is contained in one of the two closed half-spaces of $\mathbb{H}^3\setminus \mathrm{P}$ (resp. $\ads\setminus \mathrm{P}\cup\mathrm{Q}$, $\HP\setminus \mathrm{P}$).
\end{defi}
Also recall that if $\mathrm{P}\cap\mathrm{C}$ is a line then it is called a \textit{bending line}. If $\mathrm{P}\cap\mathrm{C}$ is not a line then we say that $\mathrm{P}$ is an \textit{extremal support plane} of $\mathrm{C}.$ The main ingredients to prove the convergence of $\tau_t\partial_-\mathrm{CH}(\Lambda_{\rho_{(t\lambda,t\mu)}})$ are Proposition \ref{hyperbolicsurface} and \ref{adssurface}. Before stating the propositions, we would like to remind that by construction, the image of the positive bending map $\mathrm{b}_{\vert t\vert\lambda,+}^{\X}$ is exactly $\partial_+\mathrm{CH}(\Lambda_{\rho_{(t\lambda,t\mu)}})$. Hence, the hyperbolic plane $\mathbb{H}^2$ simultaneously serves as a support plane for $\partial_+\mathrm{CH}(\Lambda_{\rho_{(t\lambda,t\mu)}})$ for any $t$ (small enough). 
\begin{prop}\label{hyperbolicsurface}
For $t>0$, let $p_t$ be a family of points in the boundary of 
$\mathbb{H}^2\cap \partial_+\mathrm{CH}(\Lambda_{\rho_{(t\lambda,t\mu)}})$ which converges in $\mathbb{H}^2$. Consider $\mathrm{L}_t$ the geodesic in $\mathbb{H}^3$ such that $$\mathrm{L}_{t}(0)=p_t, \ \ \ \mathrm{L}_t^{'}(0)=-e_3:=(0,0,0,-1).$$
Then there exists a support plane $\mathrm{P}_t$ of $\partial_-\mathrm{CH}(\Lambda_{\rho_{(t\lambda,t\mu)}})$ at the intersection point $\mathrm{L}_t\cap\partial_-\mathrm{CH}(\Lambda_{\rho_{(t\lambda,t\mu)}})$ such that up to taking a subsequence, the family of planes $\tau_t\mathrm{P}_t$ converges to a spacelike plane  $\mathrm{P}_{\infty}$ in $\HP.$ 
\end{prop}
Note that the hypothesis that $p_t$ is contained in the boundary of $\mathbb{H}^2\cap \partial_+\mathrm{CH}(\Lambda_{\rho_{(t\lambda,t\mu)}})$ is necessary to apply Proposition \ref{suppH} which gives an estimate of the distance between $p_t$ and the lower boundary component $\partial_-\mathrm{CH}(\Lambda_{\rho_{(t\lambda,t\mu)}})$.
\begin{proof}
Let $p_t$ be a family of points in a bending line contained in $\mathbb{H}^2\cap \partial_+\mathrm{CH}(\Lambda_{\rho_{(t\lambda,t\mu)}})$ that converge to some $p_{\infty}$ in $\mathbb{H}^2.$ The Proposition \ref{suppH} implies that there is a point $y_t$ in $\partial_-\mathrm{CH}(\Lambda_{\rho_{(t\lambda,t\mu)}})$ such that $$d_{\mathbb{H}^3}(p_t,y_t)\leq Ct,$$ for some constant $C>0$. Let $\mathrm{Q}_t$ be a support plane of $\partial_-\mathrm{CH}(\Lambda_{\rho_{(t\lambda,t\mu)}})$ at $y_t$, we remark that $\mathrm{Q}_t$ is disjoint from $\mathbb{H}^2$ since $\mathbb{H}^2$ is a support plane for $\partial_+\mathrm{CH}(\Lambda_{\rho_{(t\lambda,t\mu)}})$ and any support plane of the upper boundary component $\partial_+\mathrm{CH}(\Lambda_{\rho_{(t\lambda,t\mu)}})$ is disjoint from a support plane of the lower boundary component $\partial_-\mathrm{CH}(\Lambda_{\rho_{(t\lambda,t\mu)}})$. 
Therefore the support plane $\mathrm{Q}_t$ must converge to $\mathbb{H}^2$ since $y_t\to p_{\infty}\in \mathbb{H}^2$ and $\mathrm{Q}_t$ is disjoint from $\mathbb{H}^2$.

\begin{figure}[htb]
\centering
\includegraphics[width=.6\textwidth]{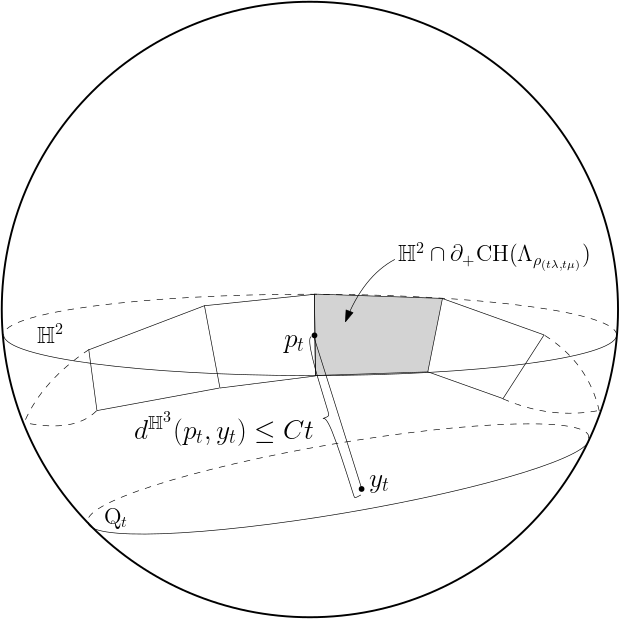}
\caption{Illustration of the beginning of the proof of Proposition \ref{hyperbolicsurface} using the estimate \ref{suppH}.}\label{Seriesestimate}
\end{figure}

We claim now that $\tau_t\mathrm{Q}_t$ converges also to a spacelike plane $\mathrm{Q}_{\infty}$ in $\HP$ up to subsequence. To prove this, it is enough to find a point $z_t\in \mathrm{Q}_t$ such that $\tau_tz_t$ is bounded. We consider $z_t$ the intersection of $\mathrm{Q}_t$ with the geodesic starting from $p_t$ with velocity $-e_3$. Namely 
\begin{equation}\label{verticallineH}
    z_t=[\cosh{(d_{\mathbb{H}^3}(p_t,z_t))}\widetilde{p_t}   -\sinh({d_{\mathbb{H}^3}(p_t,z_t)})e_3],
\end{equation}
where $\widetilde{p_t}$ is the lift of $p_t$ in $\mathcal{H}^3$ .
Let $\alpha(t):=(\alpha_0(t),\alpha_1(t),\alpha_2(t),\alpha_3(t))$ be a unit normal vector with respect to $q_1$ such that $\mathrm{Q}_t$ is the intersection of $\mathbb{H}^3$ with the projectivization of the orthogonal of $\alpha(t)$, here orthogonal with respects to the bilinear form $\langle \cdot, \cdot \rangle_{1,3}$ whose associated quadratic form is $q_{1}.$ We obtain 
\begin{align*}
z_t\in \mathrm{Q}_t&\iff\cosh{(d_{\mathbb{H}^3}(p_t,z_t))}\langle  \widetilde{p_t}  , \alpha(t)\rangle_{1,3} -\sinh({d_{\mathbb{H}^3}(p_t,z_t)})\langle e_3,\alpha(t)\rangle_{1,3}=0.\\
&\iff \tanh(d_{\mathbb{H}^3}(p_t,z_t))=\frac{\langle \widetilde{p_t},\alpha(t)\rangle _{1,3}}{\alpha_3(t)}. 
\end{align*}
By elementary hyperbolic geometry it is easy to check that $$\sinh(d_{\mathbb{H}^3}(p_t,\mathrm{Q}_t))=\vert\langle \widetilde{p_t},\alpha(t)\rangle_{1,3}\vert.$$ 
Since $y_t\in \mathrm{Q}_t$, then $d_{\mathbb{H}^3}(p_t,\mathrm{Q}_t)\leq d_{\mathbb{H}^3}(p_t,y_t)$ hence $\vert\langle \widetilde{p_t},\alpha(t)\rangle_{1,3}\vert\leq \sinh(Ct)$. This implies that for some $C^{'}>0$ we have $$\vert\langle \widetilde{p_t},\alpha(t)\rangle_{1,3}\vert\leq C^{'}t.$$ Note that since $\mathrm{Q}_t\to\mathbb{H}^2$, then necessary $\alpha_3(t)\to 1 $. Hence
$$d_{\mathbb{H}^3}(p_t,z_t)\leq C^{''}t$$ for some constant $C^{''}>0$. Therefore,
\begin{equation}\label{zt}
    \tau_t z_t=[\cosh (d_{\mathbb{H}^3}(p_t,z_t))\widetilde{p_t}-\frac{\sinh(d_{\mathbb{H}^3 }(p_t,z_t))}{t}e_3],
    \end{equation} 
so $\tau_tz_t$ is bounded in $\HP$. Now let's take $x_t$ to be the intersection of $\mathrm{L}_t$ with $\partial_-\mathrm{CH}(\Lambda_{\rho_{(t\lambda,t\mu)}})$ and $\mathrm{P}_t$ a support plane of $\partial_-\mathrm{CH}(\Lambda_{\rho_{(t\lambda,t\mu)}})$ at $x_t$. Clearly $z_t$ lies in the concave side of $\partial_-\mathrm{CH}(\Lambda_{\rho_{(t\lambda,t\mu)}})$ (see Figure \ref{lower}), hence $$d_{\mathbb{H}^3}(p_t,x_t)\leq d_{\mathbb{H}^3}(p_t,z_t)\leq C^{''}t.$$ Therefore a computation similar to \eqref{zt} shows that $\tau_tx_t$ is bounded and hence $\tau_t\mathrm{P}_t$ converges up to subsequence to some plane $\mathrm{P}_{\infty}$. We claim that $\mathrm{P}_{\infty}$ is a spacelike plane of $\HP$. Indeed since $\mathbb{H}^2$ is a support plane of $\partial_+\mathrm{CH}(\Lambda_{\rho_{(t\lambda,t\mu)}})$ then $\mathrm{P}_t$ is disjoint from $\mathbb{H}^2$, then necessary the plane  $\mathrm{P}_{\infty}$ is not vertical, otherwise for sufficiently small $t$, $\mathrm{P}_t$ would have non empty intersection with $\mathbb{H}^2$ which is a contradiction. \end{proof} 

\begin{figure}[htb]
\centering
\includegraphics[width=.7\textwidth]{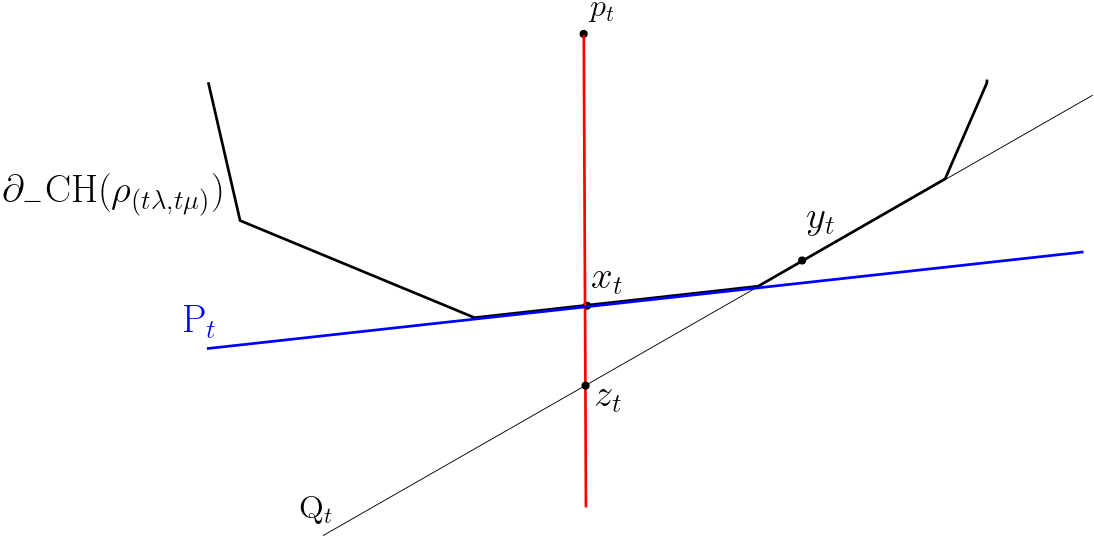}
\caption{Estimating the distance between $x_t$ and $p_t$.}\label{lower}
\end{figure}

We can show the analogous of Proposition \ref{hyperbolicsurface} in the Anti-de Sitter setting.

\begin{prop}\label{adssurface}
For $t<0$, Let $p_t$ be any point on $\mathbb{H}^2\cap \partial_+\mathrm{CH}(\Lambda_{\rho_{(t\lambda,t\mu)}})$ which converges in $\mathbb{H}^2$. Consider $\mathrm{L}_t$ the timelike geodesic in $\ads$ such that 
$$\mathrm{L}_{t}(0)=p_t, \ \ \ \mathrm{L}_t^{'}(0)=-e_3=(0,0,0,-1).$$
Then there exists a spacelike support plane $\mathrm{P}_t$ of $\partial_-\mathrm{CH}(\Lambda_{\rho_{(t\lambda,t\mu)}})$ at the intersection point $\mathrm{L}_t\cap\partial_-\mathrm{CH}(\Lambda_{\rho_{(t\lambda,t\mu)}})$ such that up to taking a subsequence, the family of planes $\tau_t\mathrm{P}_t$ converges to a spacelike plane  $\mathrm{P}_{\infty}$ in $\HP.$
\end{prop}

\begin{proof}
Let $p_t$ be a point in $\partial_+\mathrm{CH}(\Lambda_{\rho_{(t\lambda,t\mu)}})$, it follows from Corollary \ref{Fwidth} that for all $y_t\in \partial_-\mathrm{CH}(\Lambda_{\rho_{(t\lambda,t\mu)}})$
\begin{equation}\label{adsestima}
    d_{\ads}(y_t,p_t)\leq C^{'}\vert t\vert,
\end{equation} for some $C^{'}>0$. We consider $x_t$ the intersection of $\partial_-\mathrm{CH}(\Lambda_{\rho_{(t\lambda,t\mu)}})$ with the time like geodesic starting from $x_0$ with velocity $(0,0,0,1)$. Namely $$x_t=[\cos{(d_{\ads}(x_t,p_t))}\widetilde{p_t}   -\sin({d_{\ads}(x_t,p_t)})e_3],$$ where $q_{-1}(\widetilde{p_t})=-1$ and $p_t=[\widetilde{p_t}]$. Then \begin{equation}\label{verticallineADS}
    \tau_tx_t=[\cos{(d_{\ads}(x_t,p_t))}\widetilde{p_t}   -\frac{\sin({d_{\ads}(x_t,p_t)})}{\vert t\vert }e_3].
\end{equation}
As consequence, $\tau_tx_t$ is bounded, this follows from the estimate \eqref{adsestima}. Now, take $\mathrm{P}_t$ a support plane of $\partial_-\mathrm{CH}(\Lambda_{\rho_{(t\lambda,t\mu)}})$ at $x_t$. Thus up to extracting a sub-sequence, the family $\tau_t\mathrm{P}_t$ converges to a plane that we denote $\mathrm{P}_{\infty}$. Similar to the proof of Proposition \ref{hyperbolicsurface}, we observe that $\mathrm{P}_\infty$ must necessarily be a spacelike plane in $\HP$.
\end{proof}
The lower boundary component of $\mathrm{CH}(\Lambda_{\rho_{(t\lambda,t\mu)}})$ can be described in terms of negative bending cocycle as follow; we fix $y_0$ a point in the support of $\widetilde{\lambda}$ so that $p_t:=\mathrm{dev}_{(t\lambda, t\mu)}^{+}(y_0)$ is contained in the boundary of $\mathbb{H}^2\cap \partial_+\mathrm{CH}(\Lambda_{\rho_{(t\lambda,t\mu)}})$, recall that we require this condition to be able to apply the estimate of Proposition \ref{suppH}. We take $x_t$ the intersection of $\partial_{-}\mathrm{CH}(\Lambda_{\rho_{(t\lambda,t\mu)}})$ with the geodesic starting from $p_t$ and velocity $(0,0,0,-1)$. We take $\mathrm{P}_t$ the extremal support plane of $\partial_{-}\mathrm{CH}(\Lambda_{\rho_{(t\lambda,t\mu)}})$ at $x_t$ (If $x_t$ is in a bending line $l_t$, then there are two extremal support planes, in this case we choose the plane $\mathrm{P}_t$ which is in the left with respect to the orientation induced on $l_t$). By Propositions \ref{hyperbolicsurface} and \ref{adssurface}, there exist points $x_{\infty}^{\pm}$ in $\HP$ and spacelike planes $\mathrm{P}_{\infty}$ of $\HP$ such that up to subsequence we have
$$ \lim_{t \to 0^{\pm}}\tau_tx_t=x_{\infty}^{\pm},  \  \ \lim_{t \to 0^{\pm}}\tau_t\mathrm{P}_t=\mathrm{P}_{\infty}^{\pm}.$$
We will us the fact (Lemma \ref{rcompos} in the Appendix) that one can choose a family of isometries $\mathrm{A}_t^{\X}$ in $\isom(\X)$ which converges up to sebsuequence to the identity ($\X=\mathbb{H}^3$ for $t>0$ and $\X=\ads$ for $t<0$) and
\begin{equation}\label{AT}
 \mathrm{A}_t^{\X}(p_t)=x_t,\ \mathrm{A}_t^{\X}(\mathbb{H}^2)=\mathrm{P}_t.
 \end{equation}
Moreover, up to subsequence
 \begin{equation}\label{A+}
     \lim_{t \to 0^{+}}\tau_t\mathrm{A}_t^{\mathbb{H}^3}\tau_t^{-1}=\mathrm{A}_+,\ \lim_{t \to 0^{-}}\tau_t\mathrm{A}_t^{\ads}\tau_t^{-1}=\mathrm{A}_{-}
\end{equation}
The isometry $\mathrm{A}_t^{\X}$ is defined as an isometry that sends the support plane $\mathbb{H}^2$ to $\mathrm{P}_t$ and sends $p_t$ to $x_t$ (see Figure \ref{Renormalization}). Using this isometry, we can bend from $\mathrm{P}_t$ along $\mu$ to obtain the lower boundary component of $\mathrm{CH}(\Lambda_{\rho_{(t\lambda,t\mu)}})$. More precisely, we have:
\begin{equation}\label{bordinferieure}
    \mathrm{A}_t^{\X}\mathrm{b}_{\vert t\vert\mu,-}^{\X}(\widetilde{\Sigma})=\partial_{-}\mathrm{CH}(\Lambda_{\rho_{(t\lambda,t\mu)}}),
\end{equation}
where $\mathrm{b}_{\vert t\vert\mu,-}^{\X}(x)=\B^{\X}_{\vert t\vert\mu,-}(y_0,x)(\mathrm{dev}_{(t\lambda,t\mu)}^-(x)) $ and $\B^{\X}_{\vert t\vert\mu,-}$ is the negative cocycle associated to the complete hyperbolic structure $(\mathrm{dev}^-_{(t\lambda, t\mu)}, \sigma^-_{(t\lambda,t\mu)})$. 
\begin{figure}[htb]
\centering
\includegraphics[width=.8\textwidth]{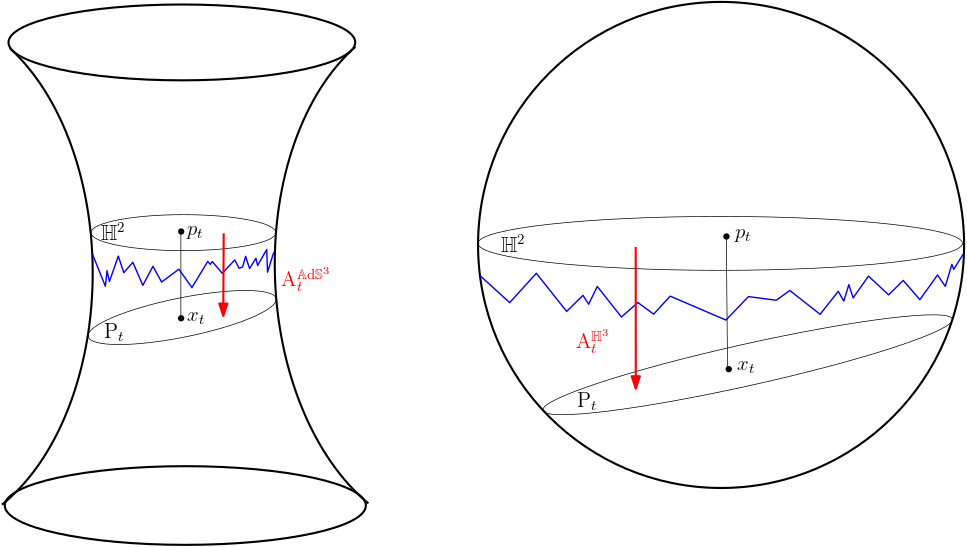}
\caption{An illustrated picture explaining the role of the isometry $\mathrm{A}_t^{\mathrm{X}}$ which sends $p_t$ to $x_t$ and $\mathbb{H}^2$ to $\mathrm{P}_t$.}\label{Renormalization}
\end{figure}
The following Proposition shows that it is not necessary to take a subsequence.
\begin{prop}\label{soussuite}
 Keeping the same notation as above, the limits of $\tau_tx_t$, $\tau_t\mathrm{P}_t$ and $\tau_t\mathrm{A}_t^\X\tau_t^{-1}$ do not depend on the extracted subsequence as $t\to0^+$ and $t\to 0^-.$  
\end{prop}

\begin{proof}
First, let us show that $\tau_tx_t$ has a unique accumulation point. We will focus on the limit from the hyperbolic part ($t>0$), and note that the Anti-de Sitter ($t<0$) case can be proved in the same way.

By contradiction, assume that there are two accumulation points of $\tau_tx_t$, denoted by $x_{\infty,1}$ and $x_{\infty,2}$. Thus, there exist two subsequences, denoted as $(\tau_{t_{k_1}}x_{t_{k_1}})_{k_1}$ and $(\tau_{t_{k_2}}x_{t_{k_2}})_{k_2}$, that converge to $x_{\infty,1}$ and $x_{\infty,2}$, respectively. Since the plane $\mathrm{P}_t$ converges to $\mathbb{H}^2$, then Lemma \ref{rcompos} implies that there are subsequences $(\mathrm{A}_{t_{k_1}}^{\mathbb{H}^3})_{k_1}$ and $(\mathrm{A}_{t_{k_2}}^{\mathbb{H}^3})_{k_2}$ such that $\lim_{t_{k_1} \to 0^+}\tau_{t_{k_1}}\mathrm{A}_{t_{k_1}}^{\mathbb{H}^3}\tau_{t_{k_1}}^{-1} :=\mathrm{A}_1$, $\lim_{t_{k_2} \to 0^+}\tau_{t_{k_2}}\mathrm{A}_{t_{k_2}}^{\mathbb{H}^3}\tau_{t_{k_2}}^{-1} :=\mathrm{A}_2$ and 
\begin{equation}\label{eqss}
    \mathrm{A}_{t_{k_1}}^{\mathbb{H}^3}(p_{t_{k,1}})=x_{t_{k_1}}, \ \mathrm{and}\ \mathrm{A}_{t_{k_2}}^{\mathbb{H}^3}(p_{t_{k,2}})=x_{t_{k_2}}.\end{equation} Recall that $p_t=\mathrm{dev}_{(t\lambda, t\mu)}^{+}(y_0)$ hence $p_t$ converges to $p_{\infty}:=\mathrm{dev}_{(\lambda,\mu)}^0(y_0)$ as $t$ goes to $0$, see Notation \ref{notation}. Thus passing to the limit in \eqref{eqss}, we get
$$\mathrm{A}_1(p_{\infty})=x_{\infty,1}, \  \mathrm{and}\ \mathrm{A}_2(p_{\infty})=x_{\infty,2}.$$
By Proposition \ref{half}, the image of the maps $\mathrm{A}_1\mathrm{b}_{\mu,-}^{\HP}$ and $\mathrm{A}_2\mathrm{b}_{\mu,-}^{\HP}:\widetilde{\Sigma}\to \HP$ is necessarily $\partial_-\mathrm{CH}(\Lambda_{\rho_{(\lambda,\mu)}^{\HP}}),$ where we recall that $\rho_{(\lambda,\mu)}^{\HP}=\lim_{t \to 0} \tau_t \rho_{(t\lambda,t\mu)}\tau_t^{-1}$. We claim now that $x_{\infty,1}=x_{\infty,2}$; indeed this follows from the fact that $\partial_-\mathrm{CH}(\Lambda_{\rho_{(\lambda,\mu)}^{\HP}})$ is a graph over $\widetilde{\Sigma}$ and the fact that both $x_{\infty,1}$ and $x_{\infty,2}$ belong to $\partial_-\mathrm{CH}(\Lambda_{\rho_{(\lambda,\mu)}^{\HP}})\cap \mathrm{F}$, where $\mathrm{F}$ is the fiber in $\HP$ over $p_{\infty}$.
Now, since $\tau_tx_t$ converges and $\mathrm{P}_t$ converges to $\mathbb{H}^2$, Lemma \ref{rcompos} allows us to choose a family of isometries $\mathrm{A}_t^{\mathbb{H}^3}$ such that $\tau_t\mathrm{A}_t^{\mathbb{H}^3}\tau_t^{-1}$ converges to an isometry of $\HP$, and  
$$\mathrm{A}_t^{\mathbb{H}^3}(p_t)=x_t, \ \mathrm{and}\ \mathrm{A}_t^{\mathbb{H}^3}(\mathbb{H}^2)=\mathrm{P}_t.$$
On the other hand, since, $\tau_t$ fixes $\mathbb{H}^2$, then the family of planes $\tau_t\mathrm{P}_t=\tau_t\mathrm{A}_t^{\mathbb{H}^3}\tau_t^{-1}(\mathbb{H}^2)$ will be automatically convergent. This completes the proof.\end{proof}

Note that the choice of the isometry $\mathrm{A}_t^{\X}$ in identity \eqref{AT} is not unique. In fact we can take any isometry of the form $\mathrm{A}_t^{\X}\mathrm{R}$ where $\mathrm{R}$ is an element of $\isom(\mathbb{H}^2)$ such that $ \mathrm{R}(p_t)=p_t  $. The next Lemma shows that we can manage to change $\mathrm{A}_t^{\X}$ in such way that the rescaled limits coincide as $t\to0^+$ and $t\to0^-$.
\begin{lemma}\label{adju}
    Up to composing $\mathrm{A}_t^{\X}$ with an element $\mathrm{R}_t$ in $\isom(\mathbb{H}^2)$ fixing $p_t$ we have:
   \begin{equation}\label{ajusememnt}
       \lim_{t \to 0^{+}}\tau_t\mathrm{A}_t^{\mathbb{H}^3}\tau_t^{-1}= \lim_{t \to 0^{-}}\tau_t\mathrm{A}_t^{\ads}\tau_t^{-1}\end{equation}
\end{lemma}
\begin{proof}
    Let $\mathrm{A}_{+}$, $ \mathrm{A}_-$ be elements of $\isom(\HP)$ such that 
    $$\lim_{t \to 0^{+}}\tau_t\mathrm{A}_t^{\mathbb{H}^3}\tau_t^{-1}=\mathrm{A}_+, \lim_{t \to 0^{-}}\tau_t\mathrm{A}_t^{\ads}\tau_t^{-1}=\mathrm{A}_{-}.$$ 
    As in the proof of Proposition \ref{soussuite}, we keep denoting by $p_{\infty}$ the limit of $p_{t}$ as $t$ goes to $0$. Then from the proof of the Proposition \ref{soussuite}, the limits $\lim_{t \to 0^+}\tau_tx_t$ and $\lim_{t \to 0^-}\tau_tx_t$ are contained in $\partial_-\mathrm{CH}(\Lambda_{\rho_{(\lambda,\mu)}^{\HP}})\cap \mathrm{F}$ where $\mathrm{F}$ is the fiber in $\HP$ over $p_{\infty}$, hence the two limits are equal and we will denote them by $x_{\infty}.$
    Passing to the limit in the identity \eqref{AT}, we get $$\mathrm{A}_{+}^{-1}\mathrm{A}_{-}(p_{\infty})=p_{\infty}, \ \ \mathrm{A}_{+}^{-1}\mathrm{A}_{-}(\mathbb{H}^2)=\mathbb{H}^2.$$
    Therefore we can view $\mathrm{A}_{+}^{-1}\mathrm{A}_{-}$ as a rotation of $\mathbb{H}^2$ fixing $p_{\infty}$. Since $p_t$ converges to $p_{\infty}$ then it not difficult to construct a family of rotations $\mathrm{R}_t$ of $\isom(\mathbb{H}^2)$ that fix $p_t$ and converge to $\mathrm{A}_{+}^{-1} \mathrm{A}_{-}$ (for example, by using Lemma \ref{rcompos} where the family of planes is constant and equal to $\mathbb{H}^2$). As consequence the family of isometries defined by
    $$\begin{cases}
\mathrm{A}_t^{\mathbb{H}^3}\mathrm{R}_t &\text{ if}\ t>0\\
\mathrm{A}_t^{\ads} &\text{ if }\ t<0
\end{cases}$$ 
satisfies the identities \eqref{AT} and \eqref{ajusememnt}. This concludes the proof.
\end{proof}

Combining Propositions \ref{Sx1} and Lemma \ref{adju}, we get the principal result of this section about the transitions of the boundary components of $\mathrm{CH}(\Lambda_{\rho_{(t\lambda,t\mu)}})$.

\begin{theorem}[Transition of pleated surfaces]\label{Sx0}
Let $\mathrm{b}_{\vert t\vert\lambda,+}^{\X}$ and $\mathrm{b}_{\vert t\vert\mu,-}^{\X}$ be the positive and negative bending maps associated to the complete hyperbolic structures $(\mathrm{dev}_{(t\lambda, t\mu)}^-, \sigma_{(t\lambda,t\mu)}^-)$ and $(\mathrm{dev}_{(t\lambda, t\mu)}^+, \sigma_{(t\lambda,t\mu)}^+)$, respectively. Consider also the family $\mathrm{A}_t^{\X}$ obtained in Lemma \ref{adju}. Then:
\begin{itemize}
\item $\tau_t \mathrm{b}_{\vert t\vert\mu,-}^{\X}$ converges uniformly to $\mathrm{b}_{\lambda,+}^{\HP}$ on compact sets of $\widetilde{\Sigma}$, (see Proposition \ref{Sx1}).
\item $\tau_t \mathrm{A_t}^{\X}\mathrm{b}_{\vert t\vert\mu,-}^{\X}$ converges uniformly to $\mathrm{A}^{\HP}\mathrm{b}_{\mu,-}^{\HP}$ on compact sets of $\widetilde{\Sigma}$, where $\mathrm{A}^{\HP}=\lim_{t \to 0^{\pm}}\tau_t\mathrm{A}_t^{\X}\tau_t^{-1}$.
\end{itemize}
Here, $\mathrm{b}_{\lambda,+}^{\HP}$ and $\mathrm{b}_{\mu,-}^{\HP}$ are the positive and negative bending maps associated to the complete hyperbolic structure $(\mathrm{dev}_{(\lambda, \mu)}^0, k_{\lambda,\mu})$.
\end{theorem}
We finish this section by a remark that will be useful in the next Section:
\begin{remark}\label{comutation}
Since the family of isometries $\mathrm{A}_t^X$ is constructed by using Lemma \ref{rcompos}, then it converges to the identity. This implies that the half-pipe isometry $\mathrm{A}^{\HP}:=\lim_{t \to 0^{\pm}}\tau_t\mathrm{A}_t^{\X}\tau_t^{-1}$ has the form $\begin{bmatrix}
\mathrm{Id} & 0 \\
v & 1 
\end{bmatrix}  $ for some $v\in \mathbb{R}^{1,2}$. Therefore, if we denote by $\pi$ the projection $\HP\to\mathbb{H}^2$, then $\pi\circ\mathrm{A}^{\HP}([x,t])=[x]$, for all $[x,t]$ in $\HP$. In particular $\mathrm{b}_{\lambda,+}^{\HP}$ and  $\mathrm{A}^{\HP}\mathrm{b}_{\mu,-}^{\HP}$ have the same projection in $\mathbb{H}^2$. 
\end{remark}

\section{Transition of developing map}\label{TRNDEV}
In this section, our goal is to construct a developing map with holonomy given by $\rho_{(t\lambda,t\mu)}$. To achieve this, we will extend the bending maps obtained in Propositions \ref{Sx1} and \ref{Sx0} along a vector field which is transverse to the pleated surfaces $\partial_{\pm}\mathrm{CH}(\Lambda_{\rho_{(t\lambda,t\mu)}})$. This vector field would not be the normal vector of a pleated surface since the behavior of the equidistant surface in hyperbolic space is different from that in Anti-de Sitter space. Specifically, the equidistant surface obtained by following the normal flow in the convex (resp. concave) side may be singular (resp. not singular) in hyperbolic space, while the opposite situation holds in Anti-de Sitter space. Given a surface $\mathrm{S}$ in $\mathbb{H}^3$ or $\ads$, the singular points of the equidistant surface $\mathrm{S}_r$ at distance $r>0$ from $\mathrm{S}$ are the points which have more than one projection on $\mathrm{S}$. We will address this issue in Proposition \ref{vector}.

Finally, once we control the behaviour of the developing map in a neighborhood $\widetilde{\Sigma}\times\{0\}$ and $\widetilde{\Sigma}\times\{1\}$, we will use a standard argument from the theory of deformation of geometric structures to extend the developing map into $\widetilde{\Sigma}\times[0,1]$.
\subsection{Transverse vector field to the pleated surface}
The main goal of this subsection is to prove Proposition \ref{vector}, which provides a unit vector field that is transverse to $\partial\mathrm{CH}(\Lambda_{\rho_{(t\lambda,t\mu)}})$ and invariant with respect to the action of $\rho_{(t\lambda,t\mu)}$. By "unit vector field" on $\mathbb{H}^3$ (resp. $\ads$), we mean a vector field that has a norm of $1$ (resp. $-1$).
\begin{prop}\label{vector}
There exists a continuous unit vector field $\mathcal{N}_t^{\pm}: \partial_{\pm}\mathrm{CH}(\Lambda_{\rho_{(t\lambda,t\mu)}})\to \mathrm{T}\X$ such that
\begin{itemize}
    \item $\mathcal{N}_t^{\pm}$ is $\rho_{(t\lambda,t\mu)}$-invariant. Namely for each $\gamma\in \pi_1(\Sigma)$ $$\rho_{(t\lambda,t\mu)}(\gamma)^* \mathcal{N}_t^{\pm}=\mathcal{N}_t^{\pm}.$$
    \item $\mathcal{N}_t^{\pm}$ is transverse to $\partial_{\pm}\mathrm{CH}(\Lambda_{\rho_{(t\lambda,t\mu)}})$, that is for each $x\in \partial_{\pm}\mathrm{CH}(\Lambda_{\rho_{(t\lambda,t\mu)}})$ and $\mathrm{P}_t$ a support plane at $x$, the vector $\mathcal{N}_t^{\pm}(x)$ is not contained in $\mathrm{T}_x\mathrm{P}_t$. 
    \item There is $\delta>0$, such that the maps $(x,s)\to \exp_{\mathrm{b}_{t\lambda,+}(x)}^{\X}(s\mathcal{N}_t^+(x))$ and $(x,s)\to \exp_{\mathrm{A}^{\X}_t\mathrm{b}_{t\mu,-}(x)}^{\X}(s\mathcal{N}_t^-(x))$ are local homeomorphisms from $\widetilde{\Sigma}\times[0,\delta]$ to $\X$, where $\exp^X$ denotes the exponential map associated to $\X$.
    \item  $\lim_{t \to 0^{\pm}}\mathcal{N}_t^{\pm}=\mp e_3,$ where we recall that $e_3=(0,0,0,1)$ is a unit normal vector of the plane $\mathbb{H}^2.$
\end{itemize}
 
\end{prop}
Since $\lambda$ and $\mu$ are weighted multicurves, we can focus on understanding the construction of the vector field on the surface obtained by the union of two half-spaces that intersect along a geodesic. More precisely, let $\mathrm{P}$ and $\mathrm{Q}$ be two planes (resp. spacelike planes) in $\mathbb{H}^3$ (resp. in $\ads$) that intersect along a geodesic $l$.  We fix an orientation on $l$ and a vector $\mathrm{V}$ (resp. timelike vector in $\ads$) which is not colinear to the direction of $l$, this gives rise to a well defined notion of a left/right side of $l$ in the planes $\mathrm{P}$ and $\mathrm{Q}$. Then we define a \textit{roof} as a surface $\mathrm{S}$ in $\mathbb{H}^3$ (resp. $\ads$) that consists of two pieces. The first piece is the portion of $\mathrm{P}$ in the, say left of $l$. The second piece is the portion of $\mathrm{Q}$ which is in the right of $l$. See Figure \ref{CPQ}. We will use the notation $\mathrm{S}^{\mathbb{H}^3}(\mathrm{P}, \mathrm{Q})$ (resp.  $\mathrm{S}^{\ads}(\mathrm{P}, \mathrm{Q})$ ) to specify a roof in $\mathbb{H}^3$ (resp. in $\ads$) and to indicate the planes $\mathrm{P}$, $\mathrm{Q}$ used to construct it. Note that the last notations can be potentially confusing, indeed once we choose an orientation, we can construct two roofs starting from the planes $\mathrm{P}$ and $\mathrm{Q}$ depending on which side of $\mathrm{P}\setminus l$ and $\mathrm{Q}\setminus l$ we choose. Therefore, whenever we write $\mathrm{S}^{\mathbb{H}^3}(\mathrm{P}, \mathrm{Q})$ or $\mathrm{S}^{\ads}(\mathrm{P}, \mathrm{Q})$, it should be understood as choosing one of the two possible roofs.

\begin{figure}[htb]
\centering
\includegraphics[width=.6\textwidth]{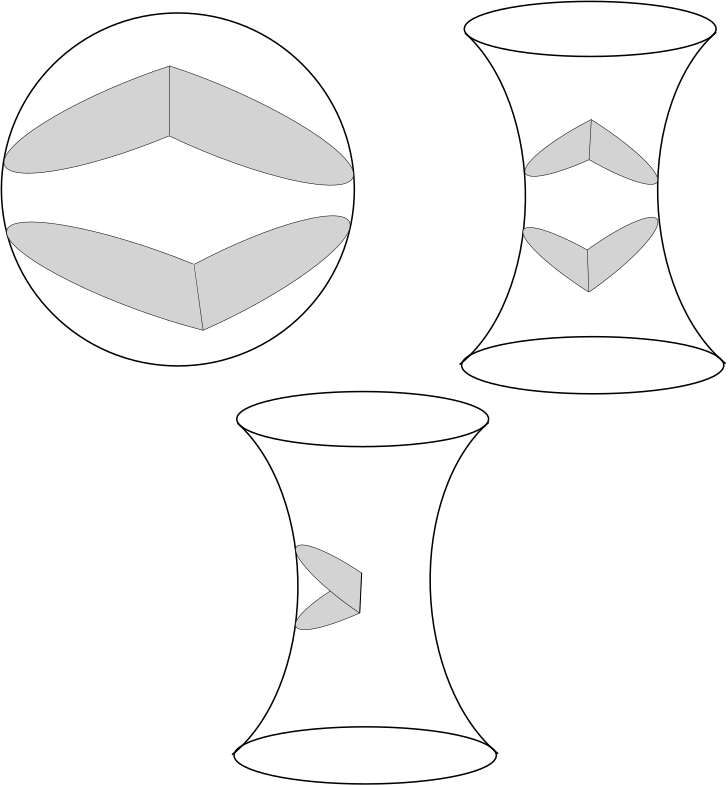}
\caption{On the top, examples of hyperbolic and Anti-de Sitter roofs (shaded part) in an affine chart. In the bottom, a non-example of a roof in $\ads$ that may occur if we allow the vector $\mathrm{V}$ used to define a roof to be spacelike.}\label{CPQ}
\end{figure}

\begin{lemma}\label{AdsV}
      Let $\mathrm{P}$ and $\mathrm{Q}$ be two spacelike planes in $\ads$ which intersect along a geodesic $l$ and $\mathrm{S}^{\ads}(\mathrm{P}, \mathrm{Q})$ a roof in $\ads$. Let $\epsilon>0$ and $U_{l}(\epsilon)$ an $\epsilon$-neighborhood of $l$ in $\mathrm{S}^{\ads}(\mathrm{P}, \mathrm{Q})$ Then, there exists a continuous unit transverse vector field $X_{\mathrm{P},\mathrm{Q},\epsilon}$ on $\mathrm{S}^{\ads}(\mathrm{P}, \mathrm{Q})$ such that 
    \begin{itemize}
    \item The restriction of $X_{\mathrm{P},\mathrm{Q},\epsilon}$ to $\mathrm{P}\cap\mathrm{S}^{\ads}(\mathrm{P}, \mathrm{Q})\setminus U_{l}(\epsilon)$ (resp. $\mathrm{Q}\cap\mathrm{S}^{\ads}(\mathrm{P}, \mathrm{Q})\setminus U_{l}(\epsilon)$) coincide with the unit normal vector $\mathrm{N_{\mathrm{P}}}$ of $\mathrm{P}$ pointing into the convex side of $\mathrm{S}^{\ads}(\mathrm{P}, \mathrm{Q})$ (resp. the unit normal vector $\mathrm{N_{\mathrm{Q}}}$ of $\mathrm{Q}$).
        \item $X_{\mathrm{P},\mathrm{Q}},\epsilon$ invariant by the 1-parameter subgroups of hyperbolic isometry in $\ads$ preserving $l=\mathrm{P}\cap \mathrm{Q}$.
        \item For any $0<\delta<\frac{\pi}{2}$, the map $\mathcal{E}: \mathrm{S}^{\ads}(\mathrm{P}, \mathrm{Q})\times [0,\delta]\to \ads$, $(x,s)\to \exp_x^{\ads}(sX_{\mathrm{P},\mathrm{Q},\epsilon}(x))$ is a local homeomorphism.
       
    \end{itemize}
\end{lemma}
The $\epsilon$-neighborhood of $l$ in $\mathrm{S}^{\ads}(\mathrm{P}, \mathrm{Q})$ refers to the set of points in $\mathrm{S}^{\ads}(\mathrm{P}, \mathrm{Q})$ at distance at most $\epsilon$ from $l$, In this context the distance is defined as the path distance induced on the spacelike surface $\mathrm{S}^{\ads}(\mathrm{P}, \mathrm{Q})$, which is isometric to $\mathbb{H}^2$. By the one-parameter subgroup of hyperbolic isometries preserving $l$, we mean the following: Up to composing by an isometry of $\ads$, assume that $l$ is equal to the spacelike geodesic $\{x_0=x_1=0\}$. Then the subgroup is given by:
\begin{equation}\label{adshyper}
\left\{\begin{bmatrix}
1 & 0 & 0&0 \\
0 & 1 & 0 &0\\
0 & 0 & \cosh(t)& \sinh(t)\\
0&0& \sinh(t)& \cosh (t)
\end{bmatrix}\ \mid \ t\in \mathbb{R}\right\}.
\end{equation}
The family of isometries given in \eqref{adshyper} acts as a translation on $l$. One may also remark that these isometries can be interpreted as rotations that fix pointwise the geodesic $l^*=\{x_3=x_4=0\}$, and this is, in fact, the \textit{dual} of the geodesic $l$. We refer the reader to \cite{adsarticle} for a more detailed exposition about duality in AdS geometry.

\begin{proof}[Proof of Lemma \ref{AdsV}]
We may assume, up to applying an isometry of $\mathrm{Isom}(\ads)$, that $\mathrm{P}=\{[x_0,x_1,x_2,x_3]\in \ads \mid x_3=0\}$ and $\mathrm{Q}=\{[x_0,x_1,x_2,x_3]\in \ads \mid x_3=-\tanh(\theta)x_2\}$, where $\theta$ is the angle between $\mathrm{P}$ and $\mathrm{Q}$. Thus, $l=\mathrm{P}\cap\mathrm{Q}$ is the spacelike geodesic in $\ads$ given by $\{[x_0,x_1,x_2,x_3]\in \ads \mid x_3=x_2=0\}$. We will focus on the roof given by (see Figure \ref{Vectorads}):
\begin{align*}
\mathrm{S}^{\ads}(\mathrm{P},\mathrm{Q})&=\left(\mathrm{P}\cap \left\{[x_0,x_1,x_2,x_3]\in\ads \mid \frac{x_2}{x_0}<0\right\} \right) \\
&\ \ \cup\left(\mathrm{Q}\cap \left\{[x_0,x_1,x_2,x_3]\in\ads \mid \frac{x_2}{x_0}>0\right\} \right).
\end{align*}

Let $\mathrm{R}:=\{[x_0,x_1,x_2,x_3]\in \ads\mid x_1=0\}$ be the timelike plane orthogonal to $\mathrm{P}$ and containing $l$. This allow us to reduce the problem to a two-dimensional problem in $\mathrm{R}\cong \mathbb{A}\mathrm{d}\mathbb{S}^2$. Indeed, once we construct the desired vector field on $\mathrm{R}\cap \mathrm{S}^{\ads}(\mathrm{P},\mathrm{Q})$, we can use a 1-parameter subgroup of hyperbolic isometries in $\ads$ that preserves $l$ to extend the vector field to $\mathrm{S}^{\ads}(\mathrm{P},\mathrm{Q})$. In particular, our vector field will satisfy the second item of the statement. By an elementary computation, we can see that the set $\mathrm{R}\cap \mathrm{S}^{\ads}(\mathrm{P},\mathrm{Q})$ can be parameterized by the piecewise geodesic $\alpha: \mathbb{R}\to \mathrm{R}$ given by: 
$$\begin{array}{ccccc}
  \\
 & & \alpha(t) & = & \begin{cases}
[\cosh(t),0,\sinh(t)\cosh(\theta),-\sinh(t)\sinh(\theta)] &\text{ if}\ t\geq 0 \\
[\cosh(t),0,\sinh(t),0] &\text{ if }\ t\leq 0,
\end{cases} \\
\end{array}$$
Now consider $f: \mathbb{R}\to [0,+\infty[$ a smooth increasing function such that $f=0$ on $]-\infty, 0] $ and $f=\theta $ on $[\epsilon,+\infty[$. Define $X_{\mathrm{P},\mathrm{Q},\epsilon}$ on $\mathrm{R}\cap \mathrm{S}^{\ads}(\mathrm{P},\mathrm{Q})$ as the unit vector given by 
$$X_{\mathrm{P},\mathrm{Q},\epsilon}(\alpha(t))=\left(0,0,\sinh(f(t)),-\cosh(f(t))\right).$$

\begin{figure}[htb]
\centering
\includegraphics[width=.7\textwidth]{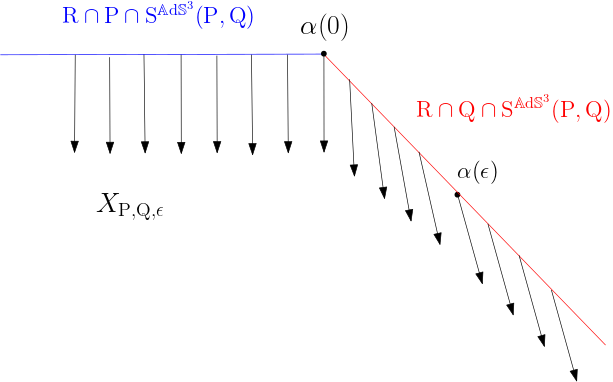}
\caption{A cross section of the roof $\mathrm{S}^{\ads}(\mathrm{P},\mathrm{Q})$ with the transverse vector field $X_{\mathrm{P},\mathrm{Q},\epsilon}$ which is oriented towards the convex side of the roof.}\label{Vectorads}
\end{figure}

One can check that the vector field $X_{\mathrm{P},\mathrm{Q},\epsilon}$ is oriented in the convex side of $\mathrm{S}^{\ads}(\mathrm{P},\mathrm{Q})$, and moreover, $X_{\mathrm{P},\mathrm{Q},\epsilon}$ satisfies the first item of the statement. The goal now is to show that the map $\mathrm{E}: \mathbb{R}\times [0,\delta]\to\mathrm{R}$ given by $\mathrm{E}(t,s)=\exp_{\alpha(t)}^{\ads}(sX_{\mathrm{P},\mathrm{Q},\epsilon}(\alpha(t)))$ is a local homeomorphism.

Using the fact that the derivative of $f$ is non negative, one can apply the inverse function theorem, to check that for all $t\neq 0$, the map $\mathrm{E}$ is a local homeomorphism (in fact it is a local diffeomorphism). The point $t=0$ corresponds to the singular point of $\mathrm{R}\cap\mathrm{S}^{\ads}(\mathrm{P},\mathrm{Q})$. Therefore, the only problem that occurs is when $t=0$. However, we still have local injectivity around this point. To see this, let's examine the $x_2$ coordinate of $\mathrm{E}(t,s)$. We have:
$$\begin{array}{ccccc}
  \\
 & &  & \begin{cases}
 \cos(s)\sinh(t)\cosh(\theta) +\sin(s)\sinh(f(t))&\text{ if}\ t\geq0 \\
\cos(s)\sinh(t)  &\text{ if }\ t\leq0,
\end{cases} \\
\end{array}$$
We observe that the $x_2$ coordinate of $\mathrm{E}(t,s)$ is positive when $t>0$ and negative when $t<0$. Therefore, $\mathrm{E}(t,s)$ is locally injective around $(0,s)$ for any $s\in [0,\delta]$.\end{proof}

\begin{lemma}\label{H3V}
    Let $\mathrm{P}$ and $\mathrm{Q}$ be two planes in $\mathbb{H}^3$ which intersect along a geodesic $l$, and $\theta$ be the angle between $\mathrm{P}$ and $\mathrm{Q}$. We furthermore assume that $0<\theta<\frac{\pi}{2}$. Consider $\mathrm{S}^{\mathbb{H}^3}(\mathrm{P}, \mathrm{Q})$ a roof in $\mathbb{H}^3$. Let $\epsilon>0$ and $U_{l}(\epsilon)$ an $\epsilon$-neighborhood of $l$ in $\mathrm{S}^{\ads}(\mathrm{P}, \mathrm{Q})$. Then, there exists a continuous unit transverse vector field $X_{\mathrm{P},\mathrm{Q},\epsilon}$ on $\mathrm{S}^{\mathbb{H}^3}(\mathrm{P}, \mathrm{Q})$ such that 
    \begin{itemize}
    \item The restriction of $X_{\mathrm{P},\mathrm{Q},\epsilon}$ to $\mathrm{P}\cap\mathrm{S}^{\mathbb{H}^3}(\mathrm{P}, \mathrm{Q})\setminus U_{l}(\epsilon)$ (resp. $\mathrm{Q}\cap\mathrm{S}^{\mathbb{H}^3}(\mathrm{P}, \mathrm{Q})\setminus U_{l}(\epsilon)$) coincides with the unit normal vector $\mathrm{N_{\mathrm{P}}}$ of $\mathrm{P}$ pointing into the convex side of $\mathrm{S}^{\mathbb{H}^3}(\mathrm{P}, \mathrm{Q})$ (resp. the unit normal vector $\mathrm{N_{\mathrm{Q}}}$ of $\mathrm{Q}$ pointing into the convex side of $\mathrm{S}^{\mathbb{H}^3}(\mathrm{P}, \mathrm{Q})$).
        \item $X_{\mathrm{P},\mathrm{Q},\epsilon}$ is invariant by the 1-parameter subgroup of hyperbolic isometries in $\mathbb{H}^3$ preserving $l=\mathrm{P}\cap \mathrm{Q}$.
\item For $\delta>0$ small enough, the map $\mathcal{E}: \mathrm{S}^{\mathbb{H}^3}(\mathrm{P}, \mathrm{Q})\times [0,\delta]\to \mathbb{H}^3$, $(z,s)\to \exp_z^{\mathbb{H}^3}(sX_{\mathrm{P},\mathrm{Q},\epsilon}(z))$ is a local homeomorphism.

    \end{itemize}
\end{lemma}
By $\epsilon$- neighborhood of $l$ in $\mathrm{S}^{\mathbb{H}^3}(\mathrm{P}, \mathrm{Q})$, we mean the following:
$$U_{l}(\epsilon)=\{ x\in  \mathrm{S}^{\mathbb{H}^3}(\mathrm{P}, \mathrm{Q}), \ d_{\mathbb{H}^3}(x,l)<\epsilon   \}.$$

\begin{proof}
    We may assume, up to applying an isometry of $\mathrm{Isom}(\mathbb{H}^3)$, that $\mathrm{P}=\{[x_0,x_1,x_2,x_3]\in \mathbb{H}^3 \mid x_3=0\}$ and $\mathrm{Q}=\{[x_0,x_1,x_2,x_3]\in \mathbb{H}^3 \mid x_3=-\tan(\theta)x_2\}$, where $\theta$ is the angle between $\mathrm{P}$ and $\mathrm{Q}$. Thus, $l=\mathrm{P}\cap\mathrm{Q}$ is the geodesic in $\mathbb{H}^3$ given by $\{[x_0,x_1,x_2,x_3]\in \mathbb{H}^3 \mid x_3=x_2=0\}$. We will focus on the roof given by:
    \begin{align*}
\mathrm{S}^{\mathbb{H}^3}(\mathrm{P},\mathrm{Q})&=\left(\mathrm{P}\cap \left\{[x_0,x_1,x_2,x_3]\in\mathbb{H}^3 \mid \frac{x_2}{x_0}<0\right\} \right) \\
&\ \ \cup\left(\mathrm{Q}\cap \left\{[x_0,x_1,x_2,x_3]\in\mathbb{H}^3 \mid \frac{x_2}{x_0}>0\right\} \right).
\end{align*}
Let $\mathrm{R}:=\{[x_0,x_1,x_2,x_3]\in \ads\mid x_1=0\}$ be the plane orthogonal to $\mathrm{P}$ and containing $l$. This allow us to reduce the problem to a two-dimensional problem in $\mathrm{R}\cong\mathbb{H}^2$. As in the Anti-de Sitter setting, in it is enough to construct the vector field on $\mathrm{R}\cap \mathrm{S}^{\mathbb{H}^3}(\mathrm{P},\mathrm{Q})$. By an elementary computation, we can see that the set $\mathrm{R}\cap \mathrm{S}^{\mathbb{H}^3}(\mathrm{P},\mathrm{Q})$ can be parameterized by the piecewise geodesic $\alpha: \mathbb{R}\to \mathrm{R}$ given by: 
$$\begin{array}{ccccc}
  \\
 & & \alpha(t) & = & \begin{cases}
[\cosh(t),0,\sinh(t)\cos(\theta),-\sinh(t)\sin(\theta)] &\text{ if}\ t\geq 0 \\
[\cosh(t),0,\sinh(t),0] &\text{ if }\ t\leq 0,
\end{cases} \\
\end{array}$$
Now consider $f: \mathbb{R}\to [0,+\infty[$ a smooth increasing function such that $f=0$ on $]-\infty, \frac{\epsilon}{2}] $ and $f=\theta $ on $[\epsilon,+\infty[$. Define $X_{\mathrm{P},\mathrm{Q},\epsilon}$ on $\mathrm{R}\cap \mathrm{S}^{\ads}(\mathrm{P},\mathrm{Q})$ as the unit vector given by
$$X_{\mathrm{P},\mathrm{Q},\epsilon}(\alpha(t))=\left(0,0,-\sin(f(t)),-\cos(f(t))\right).$$
One can check that the vector field $X_{\mathrm{P},\mathrm{Q},\epsilon}$ is oriented in the convex side of $\mathrm{S}^{\mathbb{H}^3}(\mathrm{P},\mathrm{Q})$, and moreover, $X_{\mathrm{P},\mathrm{Q},\epsilon}$ satisfies the first item of the statement. The goal now is to show that there is $\delta>0$ such that the map $\mathrm{E}: \mathbb{R}\times [0,\delta]\to\mathrm{R}$ given by $\mathrm{E}(t,s)=\exp_{\alpha(t)}^{\mathbb{H}^3}(sX_{\mathrm{P},\mathrm{Q},\epsilon}(\alpha(t)))$ is a local homeomorphism. First, using the inverse function theorem, one can  check easily that the restriction of the map $\mathrm{E}$ to each of $]-\infty,0[\times \mathbb{R}$, $]0,\frac{\epsilon}{2}[\times \mathbb{R}$ and $]\epsilon,+\infty[\times\mathbb{R}$ is a local diffeomorphism into its image. Again as in the Anti-de Sitter setting, we still have local injectivity around the singular point $t=0$. Indeed  
the $x_2$ coordinate of $\mathrm{E}(t,s)$ is given by
$$\begin{array}{ccccc}
  \\
 & &  & \begin{cases}
 \cosh(s)\sinh(t)\cos(\theta)&\text{ if}\ 0\leq t\leq \frac{\epsilon}{2} \\
\cosh(s)\sinh(t)  &\text{ if }\ t\leq0,
\end{cases} \\
\end{array}$$
Since the $x_2$ coordinate of $\mathrm{E}(t,s)$ is positive when $0<t\leq \frac{\epsilon}{2}$ (because $\theta\neq\frac{\pi}{2}$) and negative when $t < 0$, then $\mathrm{E}(t,s)$ is locally injective around $(0,s)$ for any $s \in \mathbb{R}$. As a consequence, the restriction of $\mathrm{E}$ to $]-\infty,\frac{\epsilon}{2}[\times\mathbb{R}$ is a local homeomorphism onto its image.
Now, we claim that we can find $\eta > 0$ such that 
\begin{equation}\label{EH}
    \mathrm{E}(]-\infty,\frac{\epsilon}{2}[\times[0,\eta])\cap \mathrm{E}(]\epsilon,+\infty[\times[0,\eta])=\emptyset,\end{equation}
This guarantees that the restriction of $\mathrm{E}(t,s)$ to $\left(]-\infty,\frac{\epsilon}{2}[\cup]\epsilon,+\infty[\right)\times[0,\eta]$ is a local homeomorphism into its image, as illustrated in Figure \ref{Vector}.  To prove the claim, we observe that the $x_2$ coordinates of $\mathrm{E}$ satisfy the following condition:
$$\begin{array}{ccccc}
  \\
 & &  & \begin{cases}
 x_2(t,s)>\cos(\theta)\sinh(\epsilon)-\sinh(s)\sin(\theta)  \ &\text{ if}\ t>\epsilon \\
x_2(t,s)<\cosh(s)\sinh{(\frac{\epsilon}{2})}\cos{(\theta)} &\text{ if }\ t<\frac{\epsilon}{2}.
\end{cases} \\
\end{array}$$
Then one can find a positive constant $\eta$ which depends only on $\theta$ and $\epsilon$ such that for all $0\leq s\leq \eta$, we have
$$\cosh(s)\sinh{(\frac{\epsilon}{2})}\cos{(\theta)}<\cos(\theta)\sinh(\epsilon)-\sinh(s)\sin(\theta),$$ hence $x_2(t_1, s_1)$ cannot be equal to $x_2(t_2, s_2)$ for $t_1 < \frac{\epsilon}{2}$, $t_2 > \epsilon$, and $s_1, s_2 \in [0, \eta]$. This concludes the proof of \eqref{EH}.

\begin{figure}[htb]
\centering
\includegraphics[width=.7\textwidth]{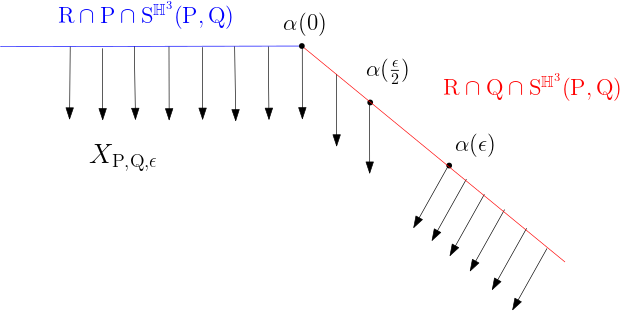}
\caption{The vector field $X_{\mathrm{P},\mathrm{Q},\epsilon}$ drawn in a cross section of the hyperbolic plane $\mathrm{R}\cong\mathbb{H}^2$. From the picture, we observe that the geodesic tangent to $X_{\mathrm{P},\mathrm{Q},\epsilon}$ starting from $\alpha(t_1)$ with $t_1>\epsilon$ does not intersect those starting from $\alpha(t_2)$ with $t_2<\frac{\epsilon}{2}$ for a small time $s<\eta$.}\label{Vector}
\end{figure}

It remains to analyze what happens when $t \in \left[\frac{\epsilon}{2}, \epsilon\right]$. By the inverse function theorem and the compactness of $\left[\frac{\epsilon}{2}, \epsilon\right]$, we can find a uniform $c > 0$ such that $\mathrm{E}$ is a local homeomorphism (diffeomorphism) at each point of $\left[\frac{\epsilon}{2}, \epsilon\right] \times [0, c]$. The proof of Lemma is then completed by taking $\delta = \min(\eta, c)$.\end{proof}

\begin{remark}\label{continuityofvector}
    It is important to note the following properties of the vector field $X_{\mathrm{P},\mathrm{Q},\epsilon}$ constructed in Lemmas \ref{H3V} and \ref{AdsV}:
    \begin{itemize}
        \item First, when $\mathrm{P}=\mathrm{Q}$, the vector field $X_{\mathrm{P},\mathrm{P},\epsilon} $ is just the unit normal vector $\mathrm{N}_{\mathrm{P}}$ orthogonal to $\mathrm{P}.$ 
    \item Second, it is not difficult to show that the construction of the vector field $X_{\mathrm{P},\mathrm{Q},\epsilon}$ is continuous with respect to $\mathrm{P}$ and $\mathrm{Q}$. More precisely, if we have a family of planes $\mathrm{P}_s$ and $\mathrm{Q}_s$ which converge to $\mathrm{P}$ and $\mathrm{Q}$ respectively, then $X_{\mathrm{P}_s,\mathrm{Q}_s,\epsilon}$ converges to $X_{\mathrm{P},\mathrm{Q},\epsilon}$ uniformly on compact sets. 
    \end{itemize}
\end{remark}
Now we have all the tools to prove Proposition \ref{vector}.
\begin{proof}[Proof of Proposition \ref{vector}]
We will only give the proof for the vector field $\mathcal{N}_t^+$ as the proof for $\mathcal{N}_t^-$ can be obtained in the same way. For each curve $\alpha_j$ in the support of $\lambda$, we take $U_j$ as a $2\epsilon$-neighborhood of $\alpha_j$ and denote $V_0$ as the union of $U_j$ for $j=1,\cdots, p$. We can choose $\epsilon$ such that these neighborhoods have disjoint closures.

Let $V_1$ be the union of small neighborhoods around the punctures of $\Sigma$. Now we take simply connected open sets $V_2,\cdots,V_N$ such that $\Sigma=\partial_{+}\mathrm{CH}(\Lambda_{\rho_{(t\lambda,t\mu)}})/\rho_{(t\lambda,t\mu)}(\pi_1(\Sigma))$ is covered by $V_0 \cup V_1 \cup \cdots \cup V_N$. We further assume that $V_2,\cdots, V_N$ are disjoint from $V_1$ and from a $\epsilon$ neighborhood of $\alpha_j$. (Remark that $U_j$ is a $2\epsilon$ neighborhood, not an $\epsilon$ neighborhood of $\alpha_j$.) The goal now is to construct a vector field $X_i^t$ in $\widetilde{V_i}$ which is $\rho_{(t\lambda,t\mu)}$-invariant for $i=0,\cdots, n$ and this completes the construction of $\mathcal{N}_t^+$.

\item[\textbullet]For each $i\geq2$, since $\widetilde{V}_i$ does not intersect the bending lines of $\partial_+\mathrm{CH}(\Lambda_{\rho_{(t\lambda,t\mu)}})$, then we can define $X_i^t$ as the normal vector pointing towards the convex side of $\partial_+\mathrm{CH}(\Lambda_{\rho_{(t\lambda,t\mu)}})$ and this is clearely $\rho_{(t\lambda,t\mu)}-$invariant.
\item[\textbullet] We now turn to the case $i=0$, we have $V_0=\bigcup_{j=1}^{p}U_j$. In order to construct  the vector field $X_0^t$, it is enough to construct  a vector field $X_{0,j}^t$ on $\widetilde{U_j}$ which is $\rho_{(t\lambda,t\mu)}-$invariant for each $j$. So we fix $j$ and a connected component $A_j$ of $\widetilde{U}_j$. Then $A_j$ is a subset of $\mathrm{S}^{\mathbb{H}^3}(\mathrm{P}_t,\mathrm{Q}_t)$ where $\mathrm{P}_t$ and $\mathrm{Q}_t$ are support plane of $\partial_+\mathrm{CH}(\Lambda_{\rho_{(t\lambda,t\mu)}})$ and $\mathrm{P}_t\cap \mathrm{Q}_t=\widetilde{\alpha_j}$. We take $X_{0,j}^t$ to be equal to the vector field $X_{\mathrm{P}_t,\mathrm{Q}_t,\epsilon}$ constructed in Lemmas \ref{H3V} and \ref{AdsV}. Note that by construction $X_{\mathrm{P}_t,\mathrm{Q}_t}$ is invariant by the $1$-parameter subgroup of hyperbolic isometries preserving $\widetilde{\alpha_j}$ and so in particular, it is invariant by $\rho_{(t\lambda,t\mu)}(\alpha_j)$. Therefore we can use the action of $\rho_{(t\lambda,t\mu)}$ to define $X_{0,j}^t$ on the other components of $\widetilde{U_j}$.

Notice also that if $\widetilde{U_j}$ intersects with some $\widetilde{V_i}$, then the value of $X_{0,j}^t$ coincides with $X_i^t$; and this follows from the construction of the vector field $X_{\mathrm{P_t},\mathrm{Q_t},\epsilon}$, which we recall that it coincides with the unit normal vector of $\mathrm{P}_t$ or $\mathrm{Q}_t$ outside an $\epsilon-$neighborhood of $\mathrm{P}_t\cap \mathrm{Q}_t$.
\item[\textbullet] For $i=1$,  we write $V_1$ as the union of punctured disks $D_j$. As before, in order to construct a vector field $X_1^t$ on $\widetilde{V_1}$, it is enough to construct  a vector field $X_{1,j}^t$ on $\widetilde{D_j}$ which is $\rho_{(t\lambda,t\mu)}$-invariant for each $j$.
So we fix $j$ and take $C_j$ a connected component of $\widetilde{D_j}$. Then $C_j$ is contained in the unique support plane $\mathrm{P}_t$ of $\partial_{+}\mathrm{CH}(\Lambda_{\rho_{(t\lambda,t\mu)}})$ containing the puncture and so we can just take the  constant vector field equal to the normal of $\mathrm{P}_t$. This is invariant by the parabolic isometry fixing $\mathrm{P}_t$ (the holonomy of loop around the puncture), hence we get a vector field $X_{1,j}^t$ on $\widetilde{D_j}$ using the action of $\rho_{(t\lambda,t\mu)}(\pi_1(\Sigma))$. Therefore we have constructed $X_1^t$ on $V_1$, and this completes the construction of $\mathcal{N}_t^+$.

By construction, $\mathcal{N}_t^+$ is a continuous (and smooth outside the bending line) unit vector field that is $\rho_{(t\lambda,t\mu)}$-invariant and transverse to $\partial_{+}\mathrm{CH}(\Lambda_{\rho_{(t\lambda,t\mu)}})$. Furthermore the convergence of $\mathcal{N}_t^{+}$ to $-e_3$ follows from the Remark \ref{continuityofvector} and from the fact that $\mathcal{N}_t^{+}$ points in the convex side of $\partial_{+}\mathrm{CH}(\Lambda_{\rho_{(t\lambda,t\mu)}})$ as well as from the convergence of the pleated surface $\partial_{+}\mathrm{CH}(\Lambda_{\rho_{(t\lambda,t\mu)}})$ to $\mathbb{H}^2$ as proven in Lemma \ref{Sx1}. The fact that the exponential of the vector field $\mathcal{N}_t^+$ is a local homeomorphism follows from Lemmas \ref{H3V} and \ref{AdsV} since the only potential issue could arise when we are on the bending lines. However, by construction $\mathcal{N}_t^+$ coincides locally with $X_{\mathrm{P},\mathrm{Q},\epsilon}$.
\end{proof}
\subsection{Developing map of convex core structure}
Let us decompose the surface $\Sigma$ as the union of $\Sigma_{\mathfrak{p}}$ and $\Sigma_{\mathfrak{c}}$, where $\Sigma_{\mathfrak{p}}$ is a subset of $\Sigma$ consisting of the union of small neighborhoods around punctures, and $\Sigma_{\mathfrak{c}}$ is the complement of $\Sigma_{\mathfrak{p}}$ in $\Sigma$. The subset $\Sigma_{\mathfrak{p}}$ could be empty if the surface $\Sigma$ is closed. We start this subsection by constructing a developing map $\widetilde{\Sigma_\mathfrak{c}}\times [0,1]\to\HP$ which is $\rho_{(\lambda,\mu)}^{\HP}$-equivariant. To achieve this, let us define the function $\psi$ as follows
\begin{equation}
    \begin{array}{ccccc}
\psi & : & \widetilde{\Sigma} & \to & \mathbb{R} \\
 & & x & \mapsto &  \mathrm{L}(\mathrm{b}_{\lambda,+}^{\HP}(x))-\mathrm{L}(\mathrm{A}^{\HP}\mathrm{b}_{\mu,-}^{\HP}(x)), \\
\end{array}
\end{equation}
where $\mathrm{A}^{\HP}$ is the same Half-pipe isometry defined in Proposition \ref{Sx0} and $\mathrm{L}$ is the function defined in the Section \ref{halfpipe_structure}.
Since the cocycles $\mathrm{b}_{\lambda,+}^{\HP}$ and $\mathrm{A}^{\HP}\mathrm{b}_{\mu,-}^{\HP}$ are $\pi_1(\Sigma)$-equivariant and they have the same projection on $\mathbb{H}^2$ (see Remark \ref{comutation}), we can use the formula of the function $\mathrm{L}$ given in \eqref{formule_L} to show that $\psi$ is $\pi_1(\Sigma)$-invariant. This means that for all $\gamma\in \pi_1(\Sigma)$ and $x\in \widetilde{\Sigma}$, we have
\begin{equation}\label{psi}
    \psi(\gamma\cdot x)=\psi(x).
\end{equation}  
\begin{prop}\label{HP3_convexcore}
There exists a developing map $\mathrm{Dev}_{(\lambda,\mu)}^{\HP}:\widetilde{\Sigma_\mathfrak{c}}\times [0,1]\to \HP$ which is equivariant with respect to the $\HP$-quasi-Fuchsian structure given by $\rho_{(\lambda,\mu)}^{\HP}$, such that for all $x\in \widetilde{\Sigma_{\mathfrak{c}}}$:
\begin{itemize}
    \item $\mathrm{Dev}_{(\lambda,\mu)}^{\HP}(x,1)=\mathrm{b}_{\lambda,+}^{\HP}(x)$.
    \item $\mathrm{Dev}_{(\lambda,\mu)}^{\HP}(x,0)=\mathrm{A}^{\HP}\mathrm{b}_{\mu,-}^{\HP}(x)$,
\end{itemize}
\end{prop}
\begin{proof}
Let us describe the map $\mathrm{Dev}_{(\lambda,\mu)}^{\HP}$ in the model $\mathbb{H}^2\times \mathbb{R}$ explained in Section \ref{halfpipe_structure}, (see identity \eqref{L}). We define the map using an affine interpolation between $\mathrm{b}_{\lambda,+}^{\HP}$ and $\mathrm{A}^{\HP}\mathrm{b}_{\mu,-}^{\HP}$. Namely

\begin{equation}\label{29}
    \begin{array}{ccccc}
\mathrm{Dev}_{(\lambda,\mu)}^{\HP} & : & \widetilde{\Sigma_\mathfrak{c}}\times[0,1] & \to & \HP= \mathbb{H}^2\times \mathbb{R}\\
 & & (x,s) & \mapsto & (\mathrm{dev}_{(\lambda,\mu)}^0(x), s\mathrm{L}(\mathrm{b}_{\lambda,+}^{\HP}(x))+(1-s)\mathrm{L}(\mathrm{A}^{\HP}\mathrm{b}_{\mu,-}^{\HP}(x))) \\
\end{array}\end{equation}
where $\mathrm{dev}_{(\lambda,\mu)}^0$ is the limit of $\mathrm{dev}_{(t\lambda,t\mu)}^\pm$.
Note that the $\mathbb{H}^2$ component of both $\mathrm{b}_{\lambda,+}$ and $\mathrm{b}_{\mu,-}$ is equal to $\mathrm{dev}_{(\lambda,\mu)}^0$ (see identity \eqref{forme_de_rotation}) and hence by Remark \ref{comutation}, the $\mathbb{H}^2$ part of $\mathrm{A}^{\HP}\mathrm{b}_{\mu,-}$ is also $\mathrm{dev}_{(\lambda,\mu)}^0$.
This implies that the map $\mathrm{Dev}_{(\lambda,\mu)}^{\HP}$ satisfies the items of the statement.
Now we will prove that $\mathrm{Dev}_{(\lambda,\mu)}^{\HP}$ is equivariant with respect to the representation $\rho_{(\lambda,\mu)}^{\HP}$. First let $\sigma $ be the linear part of $\rho_{(\lambda,\mu)}^{\HP}$ and $\tau$ be a $\sigma$-cocycle such that \begin{equation}  
\rho_{(\lambda,\mu)}^{\HP}:=\mathrm{Is}(\sigma,\tau)=\left [
   \begin{array}{c c c}
      \sigma & 0 \\
      
      ^T\tau\mathrm{J}\sigma & 1\\ 
   \end{array}
\right].
\end{equation}
Let $x\in \widetilde{\Sigma_{\mathfrak{c}}}$ and $\gamma\in \pi_1(\Sigma_{\mathfrak{c}})$, we denote by $\widetilde{\mathrm{dev}_{(\lambda,\mu)}^0}$ the lift of $\mathrm{dev}_{(\lambda,\mu)}^0$ to $\mathcal{H}^2$, namely $\mathrm{dev}_{(\lambda,\mu)}^0=[\widetilde{\mathrm{dev}_{(\lambda,\mu)}^0}]$, we have
\begin{align}
\mathrm{Dev}_{(\lambda,\mu)}^{\HP}(\gamma\cdot x,s)&= \left([\widetilde{\mathrm{dev}_{(\lambda,\mu)}^0(\gamma\cdot x)}], \mathrm{L}(\mathrm{b}_{\lambda,+}^{\HP}(\gamma\cdot x))+(s-1)\psi(\gamma\cdot x)\right)   \\
&=\left([\sigma(\gamma)\widetilde{\mathrm{dev}_{(\lambda,\mu)}^0(x)}], \mathrm{L}\left(\rho_{(\lambda,\mu)}^{\HP}(\gamma)\mathrm{b}_{\lambda,+}^{\HP}(x)\right)+(s-1)\psi( x)\right)\label{39}\\
&= \left([\sigma(\gamma)\widetilde{\mathrm{dev}_{(\lambda,\mu)}^0(x)}], \mathrm{L}(\mathrm{b}_{\lambda,+}^{\HP}(x))+\langle \tau(\gamma),\widetilde{\mathrm{dev}_{(\lambda,\mu)}^0(x)} \rangle_{1,2}+(s-1)\psi(x)\right)\label{40}\\
&=\rho_{(\lambda,\mu)}^{\HP}(\gamma)\mathrm{Dev}_{(\lambda,\mu)}^{\HP}(x).
\end{align}
In equation \eqref{39}, we used the fact that $\psi$ is $\pi_1(\Sigma)$-invariant and $\mathrm{dev}_{(\lambda,\mu)}^0$ is equivariant with respect to its holonomy representation $\sigma.$ In equation \eqref{40}, we used the $\pi_1(\Sigma)$-equivariance of the positive bending cocyle $\mathrm{b}_{\lambda,+}^{\HP}$ and the following basic observation: for all $[x,t]\in \HP$ with $x\in \mathcal{H}^2$ and for any $(\mathrm{A},\mathrm{v})\in \mathrm{O}_{0}(1,2)\ltimes\mathbb{R}^{1,2}$, we have

$$\mathrm{L}\left(\mathrm{Is}(\mathrm{A},\mathrm{v})([x,t])\right)=\mathrm{L}([x,t])+\langle \mathrm{v},x\rangle_{1,2}.$$\end{proof}

\begin{remark}
    The way we define $\mathrm{Dev}_{(\lambda,\mu)}^{\HP}$ in \eqref{29} still makes sense even in $\widetilde{\Sigma_\mathfrak{p}}\times[0,1]$. However we will explain later in Proposition \ref{611} how to define $\mathrm{Dev}_{(\lambda,\mu)}^{\HP}$ on $\widetilde{\Sigma_\mathfrak{p}}\times[0,1]$ in a way which is more convenient for our purpose. 
\end{remark}
The next goal is to construct a family of developing maps $\mathrm{Dev}_{(t\lambda,t\mu)}:\widetilde{\Sigma_{\mathfrak{c}}}\times [0,1]\to \mathrm{X}$ that are $\rho_{(t\lambda,t\mu)}$-equivariant and which converge to $\mathrm{Dev}_{(\lambda,\mu)}^{\HP}$ after rescaling. Let $\mathrm{A}_t^{\X}$ be the same isometry as in Proposition \ref{Sx0} and denote by $\mathrm{N}_t^{\pm}:\widetilde{\Sigma}\to\mathrm{T}\X$ the vector fields given by
$$\mathrm{N}_t^+=\mathcal{N}_t^+\circ\mathrm{b}_{\vert t\vert\lambda,+}^{\X}, \ \ \mathrm{and}\ \ \mathrm{N}_t^-=\mathcal{N}_t^-\circ\mathrm{A}_t^{\X}\mathrm{b}_{\vert t\vert\mu,-}^{\X}.$$ The following Lemma describes $\mathrm{Dev}_{(t\lambda,t\mu)}$ in a neighborhood of the pleated surfaces.

\begin{lemma}\label{extensiondev}
Let $\delta>0$ be the same constant as in Proposition \ref{vector}. Then there is $0<\delta^{'}<\delta$ such that the map defined as $$\mathrm{Dev}_{(t\lambda,t\mu)}(x,s)=\begin{cases}
\exp^{\X}_{\mathrm{b}_{\vert t\vert\lambda,+}(x)}{((s-1)\vert t\vert \psi(x)\mathrm{N}_t^+(x))} &\text{ if}\ (x,s)\in \widetilde{\Sigma_{\mathfrak{c}}}\times [1-\delta^{'},1] \\
\mathrm{A}_t^{\X}\exp^{\X}_{\mathrm{b}_{\vert t\vert\mu,-}(x)}{( s\vert t\vert\psi(x) \mathrm{N}_t^-(x))}  &\text{ if }\ (x,s)\in \widetilde{\Sigma_{\mathfrak{c}}}\times [0,\delta^{'}] \\
\end{cases}$$ satisfies the following 

\begin{itemize}
\item $\tau_t\mathrm{Dev}_{(t\lambda,t\mu)}$ is a local homeomorphism which is $\rho_{(t\lambda,t\mu)}$-equivariant.
\item  $\tau_t\mathrm{Dev}_{(t\lambda,t\mu)}$ converges uniformly on compact sets to the developing map $\mathrm{Dev}_{(\lambda,\mu)}^{\HP}$ constructed in Proposition \ref{HP3_convexcore}. 

\end{itemize}
\end{lemma}

\begin{proof}
Let us prove the convergence of $\tau_t\mathrm{Dev}_{(t\lambda,t\mu)}$ to $\mathrm{Dev}_{(\lambda,\mu)}^{\HP}$,  we will only provide the proof for the convergence in the region $\widetilde{\Sigma_{\mathfrak{c}}}\times[1-\delta^{'},1]$, as the same proof applies to $\widetilde{\Sigma_{\mathfrak{c}}}\times [0,\delta^{'}]$. We compute:
\begin{itemize}
    \item For $t>0$, so that $\X=\mathbb{H}^3$, we take $\widetilde{\mathrm{b}_{\vert t\vert\lambda,+}(x)}\in \mathcal{H}^3$ the lift of $\mathrm{b}_{\vert t\vert\lambda,+}(x)$, we obtain
    \begin{align*}
\exp^{\X}_{\mathrm{b}_{\vert t\vert\lambda,+}(x)}{((s-1)\vert t\vert \psi(x)\mathrm{N}_t^+(x))}&=\biggl[\cosh{\left(\vert t\vert (s-1)\psi(x)\right)}\widetilde{\mathrm{b}_{\vert t\vert\lambda,+}(x)}\\
&+\sinh\left(\vert t\vert (s-1) \psi(x)\right)\mathrm{N}_t^{+}(x)\biggr].
\end{align*}
    
     \item For $t<0$ so that $\X=\ads$, denoting by the same notation $\widetilde{\mathrm{b}_{\vert t\vert\lambda,+}(x)}\in \mathcal{A}\mathrm{d}\mathcal{S}^3$ the lift of $\mathrm{b}_{\vert t\vert\lambda,+}(x)$, then we have
     \begin{align*}
\exp^{\X}_{\mathrm{b}_{\vert t\vert\lambda,+}(x)}{\left((s-1)\vert t\vert \psi(x)\mathrm{N}_t^+(x)\right)}&=  \biggl[\cos{\left(\vert t\vert (s-1)\psi(x)\right)}\widetilde{\mathrm{b}_{\vert t\vert\lambda,+}(x)}  \\
& +\sin\left(\vert t\vert (s-1) \psi(x)\right)\mathrm{N}_t^{+}(x)\biggr].
\end{align*}
 Hence 
 \begin{align*}
\lim_{t \to 0^{\pm}}\tau_t \exp^{\X}_{\mathrm{b}_{\vert t\vert\lambda,+}(x)}{((s-1)\vert t\vert \psi(x)\mathrm{N}_t^+(x))}&=[\widetilde{\mathrm{dev}_{(\lambda,\mu)}^0(x)},
\mathrm{L}(\mathrm{b}_{\lambda,+}^{\HP}(x)) +(s-1)\psi(x)]\\
&=\mathrm{Dev}_{(\lambda,\mu)}^{\HP}(x,s),
\end{align*}
\end{itemize}
where $\mathrm{dev}_{(\lambda,\mu)}^0$ is the limit of $\mathrm{dev}_{(t\lambda,t\mu)}^\pm$ as $t\to0$ (See the proof of Proposition \ref{HP3_convexcore}). Note that the uniform convergence on compact sets of $\tau_t\mathrm{Dev}_{(t\lambda,t\mu)}$ follows form the uniform convergence of the cocyle $\mathrm{b}_{\vert t\vert\lambda,+}$ to $\mathrm{b}_{\lambda,+}^{\HP}$ proved in Proposition \ref{Sx1} and form the construction of the transverse vector field $\mathrm{N}_t^+$ which depends continuously on $t$. See Remark \ref{continuityofvector}.\\ The $\rho_{(t\lambda,t\mu)}$-equivariance of $\mathrm{Dev}_{(t\lambda,t\mu)}$ follows from the $\rho_{(t\lambda,t\mu)}$-equivariance of the bending maps and the vector fields $\mathrm{N}_t^{\pm}$. 

We proceed now to prove that $\mathrm{Dev}_{(t\lambda,t\mu)}$ is a local homeomorphism. First remark that since $\psi$ is $\pi_1(\Sigma_{\mathfrak{c}})-$invariant and $\Sigma_{\mathfrak{c}}$ is compact then $\psi$ is bounded above by some positive real number $a$, let $\eta:=\min(\delta,a)$, then for $t$ small enough, the restriction of the map $\mathrm{Dev}_{(t\lambda,t\mu)}$ to each of $\widetilde{\Sigma_{\mathfrak{c}}}\times [1-\eta,1]$ and $\widetilde{\Sigma_{\mathfrak{c}}}\times [0,\eta]$ is already a local homeomorphism, this follows from the construction of the vector field $\mathrm{N}_t^\pm$ in Proposition \ref{vector} and the Lemmas \ref{H3V} and \ref{AdsV}.
Now, we claim that we can choose $\delta^{'}<\eta$ small enough such that for any $t\in (-\delta^{'},\delta^{'})$ \begin{equation}
    \mathrm{Dev}_{(t\lambda,t\mu)}(\widetilde{\Sigma_{\mathfrak{c}}}\times[1-\delta^{'},1])\cap \mathrm{Dev}_{(t\lambda,t\mu)}(\widetilde{\Sigma_{\mathfrak{c}}}\times[0,\delta^{'}])=\emptyset. \label{intersection_dev}
    \end{equation} 
This ensures that $\mathrm{Dev}_{(t\lambda,t\mu)}$ is a local homeomorphism from the union $\widetilde{\Sigma_{\mathfrak{c}}}\times ([0,\delta^{'}]\cup[1-\delta^{'},1])$ to its image. To prove the claim, assume by contradiction that there exists a sequence $\delta_n\to 0$ such that the intersection \eqref{intersection_dev} is non empty. Hence, there are $x_{n,1},\ x_{n,2}\in \widetilde{\Sigma_{\mathfrak{c}}}$, $s_{n,1}\in[0,\delta_n]$ and $s_{n,2}\in [1-\delta_n,\delta_n]$ such that:
$$\mathrm{Dev}_{(t_n\lambda,t_n\mu)}(x_{n,1},s_{n,1})=\mathrm{Dev}_{(t_n\lambda,t_n\mu)}(x_{n,2},s_{n,2}).$$ Using the facts that $\pi_1(\Sigma_{\mathfrak{c}})$ acts cocompactly on $\widetilde{\Sigma_{\mathfrak{c}}}$ and that $\mathrm{Dev}_{(t\lambda,t\mu)}$ is $\rho_{t(\lambda,t\mu)}-$ equivariant, we can assume that $x_{n,1}$ and $x_{n,2}$ converge up to subsequence to $x_1$ and $x_2$ respectively. 
Therefore, we have 
\begin{align*}
   \mathrm{Dev}_{(\lambda,\mu)}^{\HP}(x_1,0) & = \lim_{ n\to +\infty}\tau_{t_n}\mathrm{Dev}_{(t_n\lambda,t_n\mu)}(x_{n,1},s_{n,1})\\
   & = \lim_{ n\to +\infty}\tau_{t_n}\mathrm{Dev}_{(t_n\lambda,t_n\mu)}(x_{n,2},s_{n,2}) \\
   &=\mathrm{Dev}_{(\lambda,\mu)}^{\HP}(x_2,1)
\end{align*}
This gives a contradiction since $\mathrm{Dev}_{(\lambda,\mu)}^{\HP}(x_1,0)$ is contained in the upper boundary component $ \partial_-\mathrm{CH}(\Lambda_{\rho_{(\lambda,\mu)}^{\HP}})$ and $\mathrm{Dev}_{(\lambda,\mu)}^{\HP}(x_2,1)$ is contained in the lower boundary component $\partial_+\mathrm{CH}(\Lambda_{\rho_{(\lambda,\mu)}^{\HP}})$. 
\end{proof}
We now move on to the principal result of this section.
\begin{theorem}[Transition of geometric structures]\label{exdev} 
    Let $\lambda$ and $\mu $ be two weighted multicurves which fill  $\Sigma$ and consider $\rho_{(t\lambda,t\mu)}$ the family of representations as in Theorem \ref{H}. Then there is a family of developing maps $\mathrm{Dev}_{(t\lambda,t\mu)}:\widetilde{\Sigma}\times[0,1]\to \X$ such that the structure $(\tau_t\mathrm{Dev}_{(t\lambda,t\mu)},\tau_t\rho_{(t\lambda,t\mu)}\tau_t^{-1})$ converges as $t\to 0$ to the Half-pipe convex core structure  $(\mathrm{Dev}_{(\lambda,\mu)}^{\HP},\rho_{(\lambda,\mu)}^{\HP})$.
\end{theorem}
Before proving Theorem \ref{exdev} we need the following result of Siebenmann:
\begin{theorem}\cite{TransTop}\label{sib}
    Let $Y$ be a locally compact Hausdorff topological space and let $U$ be an open set of $Y$. Consider $K$ a compact set 
such that $K\subset U$. Then if $h: U \to Y$ is an embedding close to the inclusion $i: U\hookrightarrow Y$ for the compact-open topology, then there is a homeomorphism $h^{'}: Y\to Y$ equal to $h$ on $K$ and equal to the identity outside $U$. Moreover $h^{'}$ is close to the identity map on $Y$.
\end{theorem}
\begin{proof}[Proof of Theorem \ref{exdev}]
    The convergence at the level of holonomies is already proved in Theorem \ref{H}. We need only to check the transition at the the level of developing map. For that it is enough to extend the developing map obtained in Lemma \ref{extensiondev} to the entire $\widetilde{\Sigma}\times[0,1].$ First let us focus on the extension of $\mathrm{Dev}_{(t\lambda,t\mu)}$on $\widetilde{\Sigma_{\mathfrak{c}}}\times[0,1]$.
To prove this consider $$\mathcal{B}= (\Sigma_{\mathfrak{c}}\times [0,\delta^{'}])\cup (\Sigma_{\mathfrak{c}}\times [1-\delta^{'},1])$$ a compact co-dimension $0$ manifold in $N_{\mathfrak{c}}:=\Sigma_{\mathfrak{c}}\times[0,1]$, where $\delta^{'}$ is the same constant defined in Lemma \ref{extensiondev}.
Let $\mathcal{B}_T$ be a collar neighborhood of $\mathcal{B}$ in $\Sigma\times\mathbb{R}$ and $N_T=N_{\mathfrak{c}}\cup\mathcal{B}_T$. By Ehresmann-Thurston Theorem \ref{ehresman}, there is a developing map $\mathcal{D}_t$ defined on $\widetilde{N_T}$ that is equivariant with respect to $\tau_t\rho_{(t\lambda,t\mu)}\tau_t^{-1}$. Furthermore we may assume that $\mathcal{D}_t$ converges to $\mathrm{Dev}_{(\lambda,\mu)}^{\HP}$ in the $\mathcal{C}^0$ topology on $N_{\mathfrak{c}}$. Now $\tau_t\mathrm{Dev}_{(t\lambda,t\mu)}$ and $\mathcal{D}_t$ may not agree on $\mathcal{B}.$ 
However it is well known that for a compact manifold with boundary $M$, if two developing maps are close in the $\mathcal{C}^k$ topology and have the same holonomy, then they differ by composition with a $\mathcal{C}^k$ embedding defined on a slightly thinner manifold $M_0\subset M$. Moreover this embedding is close to the identity. See \cite[Theorem I.1.7.1 and page 47]{Epstein}.  

Applying this fact to $\tau_t\mathrm{Dev}_{(t\lambda,t\mu)}$ and $\mathcal{D}_t$ which are close in the $\mathcal{C}^0$ topology (they both converge to $\mathrm{Dev}_{(\lambda,\mu)}^{\HP}$) we get an embedding $f_t:\mathcal{B}_T^{'}\to \mathcal{B}_T\subset N_T $ defined on a smaller collar neighborhood $\mathcal{B}_T^{'}$ of $\mathcal{B}$ ($\mathcal{B}_T^{'}\subset\mathcal{B}$) covered by $\widetilde{f_t}$ such that 
$$\mathcal{D}_t\circ \Tilde{f_t}=\tau_t\mathrm{Dev}_{(t\lambda,t\mu)} \ \ \mathrm{on \ \mathcal{B}_T^{'}}.$$
Moreover $f_t$ is close to the inclusion (for 
the compact-open topology). Hence we can use
Theorem \ref{sib} to say that there is a 
global homeomorphism $\phi_t:N_T\to N_T$ close
to the identity which is equal to $f_t$ on 
$\mathcal{B}_T^{'} $ and equal to the identity outside of $\mathcal{B}_T^{'}$. Let 
$\widetilde{\phi_t}$ be the lift of $\phi_t$, then the map $\mathcal{D}\mathrm{ev}_t:=\mathcal{D}_t\circ\Tilde{\phi_t}$ satisfies the following:
\begin{itemize}
\item $\mathcal{D}\mathrm{ev}_t$ is $\rho_{(t\lambda,t\mu)}$-equivariant. 
\item $\mathcal{D}\mathrm{ev}_{t|\mathcal{B}}=\tau_t\mathrm{Dev}_{( t\lambda,t\mu   )} $.
\item $\lim_{t \to 0^{\pm}} \mathcal{D}\mathrm{ev}_t=\mathrm{Dev}_{(\lambda,\mu)}^{\HP}$ on $\widetilde{\Sigma_{\mathfrak{c}}}$.
 \end{itemize}
Hence we can define $\mathrm{Dev}_{(t\lambda,t\mu)}$ to be equal to $\tau_t^{-1}\mathcal{D}
\mathrm{ev}_t$ on 
$\widetilde{\Sigma_{\mathfrak{c}}}\times[0,1]$.\\ 

So it remains to define the extension on $\widetilde{\Sigma_{\mathfrak{p}}}\times [\delta^{'},1-\delta^{'}]$. First let us identify $\widetilde{\Sigma}$ with $\mathbb{H}^{2}$ (using a complete hyperbolic structure $(\mathrm{dev},\sigma)$). We consider $V_i$ a neighborhood around the $i^{th}$ puncture disjoint from the support of $\lambda$ and $\mu$, we take a component $C_i$ of the lift of $\widetilde{V_i}$ in $\mathbb{H}^2$ which is a disk tangent to $\partial \mathbb{H}^2$ at $p_i$. Here $p_i$ is the fixed point of the parabolic isometry $\sigma(\gamma_i)$, where $\gamma_i$ is a loop around the puncture representing $\partial V_i$. Since the map $\mathrm{Dev}_{(t\lambda,t\mu)}$ is already defined on $\partial C_i\times[0,1]$ (because $\partial C_i\subset\widetilde{\Sigma_{\mathfrak{c}}}$) then we can define $C_t^+=\mathrm{Dev}_{(t\lambda,t\mu)}(\partial C_i\times \{1\} )$ and $C_t^-=\mathrm{Dev}_{(t\lambda,t\mu)}(\partial C_i\times \{0\} )$. Note that since the weighted multicurves $\lambda$ and $\mu$ consist of non peripheral curves then, there is a support plane $\mathrm{P}_t^{\pm}$ of $\partial_{\pm}\mathrm{CH}(\Lambda_{\rho_{(t\lambda,t\mu)}})$ such that $C_t^{\pm}\subset\mathrm{P}_t^{\pm}$. Moreover $\mathrm{C}_t^+\cap\mathrm{C}_t^-=q_t$ where $q_t$ is the fixed point of the parabolic isometry $\rho_{(t\lambda,t\mu)}(\gamma_i)$. 

We now claim that the rescaled point $\tau_tq_t$ converges to a fixed point of the parabolic isometry $\rho_{(\lambda,\mu)}^{\HP}(\gamma)$. Indeed the rescaled support planes $\tau_t\mathrm{P}_t^{\pm}$ converge to a spacelike support plane $\mathrm{P}^{\pm}$ of $\partial_{\pm}\mathrm{CH}(\lambda_{\rho^{\HP}_{(\lambda,\mu)}})$. This convergence can be deduced form the fact that the rescaled pleated surfaces $\tau_t\partial_{\pm}\mathrm{CH}(\lambda_{\rho_{(t\lambda,t\mu)}})$ converge to $\partial_{\pm}\mathrm{CH}(\lambda_{\rho^{\HP}_{(\lambda,\mu)}})$ as shown in Propositions \ref{Sx1} and \ref{Sx0}. However, since $\mathrm{P}^+$ and $\mathrm{P}^-$ are support planes at points near the punctures, they should intersect at infinity at a single point. This implies that the intersection point is the limit of $\tau_tq_t$.

Next we will use the subgroup $H_i$ of $\isom(\mathbb{H}^2)$ consisting of hyperbolic isometries fixing $p_i$ to extend $\mathrm{Dev}_{(t\lambda,t\mu)}$ inside $C_i\times[\delta^{'},1-\delta^{'}]$. More precisely for each $x\in C_i$ we take $l_i$ to be the geodesic in $\mathbb{H}^2$ that passes through $x$ with end point $q_i$. Let $x_i$ be the intersection point of $l_i$ with $\partial C_i$. We consider $A_i$ the unique element in $H_i$ with axis $l_i$ that sends $x$ to $x_i$. We denote by $d:=d_{\mathbb{H}^2}(y,x)$ the translation length of $A_i$. Now we take $B_i$ the unique hyperbolic isometry in $\isom(\X)$ such that
\begin{itemize}
    \item $q_t$ is a fixed point of $B_i$.
    \item The axis of $B_i$ is the geodesic starting at $\mathrm{Dev}_{(t\lambda,t\mu)}(x_i,s)$ and ending at $q_t$.
    \item The translation length of $B_i$ is equal to $d.$
\end{itemize}
Then, we define 
\begin{equation}\label{devpunc}
    \mathrm{Dev}_{(t\lambda,t\mu)}(x,s):=B_i\mathrm{Dev}_{(t\lambda,t\mu)}(x_i,s).\end{equation} 
Once we extend $\mathrm{Dev}_{(t\lambda,t\mu)}$ in $C_i\times [\delta^{'},1-\delta^{'}]$ we can use the action of $\rho_{(t\lambda,t\mu)}(\pi_1(\Sigma))$ to extend $\mathrm{Dev}_{(t\lambda,t\mu)}$ on the other components of $\widetilde{V_i}\times[\delta^{'},1-\delta^{'}]$. We then define the extension on $\widetilde{\Sigma}\times[0,1]$ by repeating this procedure for all punctures. The only remaining part is to prove the convergence of the rescaled developing map which will conclude the proof.
\end{proof}  
\begin{prop}\label{611}
The map $\mathrm{Dev}_{(\lambda,t\mu)}^{\HP}$ constructed in Proposition \ref{HP3_convexcore} extends to $\widetilde{\Sigma}\times[0,1]$. More precisely, for all $(x,s)$ in $\widetilde{\Sigma}\times[\delta^{'}, 1-\delta^{'}]$, we have \begin{equation}\label{hpunc}\mathrm{Dev}_{(\lambda,t\mu)}^{\HP}(x,s):=\lim_{t \to 0^{\pm}}\mathrm{Dev}_{(t\lambda,t\mu)}(x,s).\end{equation} 
\end{prop}
Note that the left hand side of identity \eqref{hpunc} is well defined by the formula \eqref{devpunc} in the proof of Theorem \ref{exdev}. Moreover we have convergence in $\left(\widetilde{\Sigma_{\mathfrak{p}}}\times\left([0,\delta^{'}]\cup[1-\delta^{'},1]\right)\right)\cup \left(\widetilde{\Sigma_{\mathfrak{c}}}\times[0,1]\right)$ by Lemma \ref{extensiondev} and Theorem \ref{exdev}. Therefore it is enough to prove the convergence for $(x,s)$ in $\widetilde{\Sigma_{\mathfrak{p}}}\times[\delta^{'},1-\delta^{'}]$ and this follows from the following Claim. 

\begin{claim}\label{claim1}
Let $x_t$ be a family of points in $\mathbb{H}^3$ (resp. or in $\ads$) such that $\lim_{t \to 0} \tau_tx_t=x$. Let $p_t$ be a family of points in $\partial \mathbb{H}^3$ (resp. in $\partial\ads$) such that $\lim_{t \to 0} \tau_tp_t=p$. Let $L_t$ be a geodesics in $\mathbb{H}^3$ (resp. spacelike geodesics in $\ads$) starting at $x_t$ and ending at $p_t$ parameterized by arc length. Then the limit as $t\to0$ of $\tau_t L_t(d) $ exists for all $d>0$.
\end{claim}
Using the same notations as in the Proof of Theorem \ref{exdev}, one can see that the family of points $\mathrm{Dev}_{(t\lambda,t\mu)}(x_i,s)$ and $q_t$ satisfy hypotheses of the Claim. Since by construction, $\mathrm{Dev}_{(t\lambda,t\mu)}(x,s)$ lies on the geodesic starting at $\mathrm{Dev}_{(t\lambda,t\mu)}(x_i,s)$ with endpoint $q_t$ and at a distance $d$ form $\mathrm{Dev}_{(t\lambda,t\mu)}(x_i,s)$, then the convergence of $\tau_t\mathrm{Dev}_{(t\lambda,t\mu)}(x,s)$ is a consequence of the Claim.
\begin{proof}[Proof of Claim \ref{claim1}]
Let $\widetilde{x_t}$ be a lift of $x_t$ such that $q_1(x_t)=-1$ (resp or $q_{-1}(x_t)=-1$), and let $v_t$ be a unit vector in $\mathrm{T}_{x_t}\mathbb{H}^3$ (resp. $\mathrm{T}_{x_t}\ads$) such that $p_t=[x_t+v_t]$. Then the geodesic $L_t$ in $\mathbb{H}^3$ (resp. spacelike geodesic in $\ads$) is given by: 
$$L_t(d)=[\cosh{(d)}\widetilde{x_t}+\sinh{(d)}v_t].$$ 
Since $\tau_tx_t$ and $\tau_tp_t$ converge then $\tau_tv_t$ also converges. This implies the convergence of $\tau_tL_t(d).$
\end{proof}

\section{Double of convex core structure}\label{sec7}
\subsection{Singular structure}
We briefly recall the  notion of singularity in the three geometries of our interest. For more detailed exposition, we refer the reader to \cite[Chapter 4]{danciger_thesis}. Let $M$ be an oriented three-manifold and $L$ a link in $M$, namely $L$ is a finite  disjoint union of embedded circles $K_i$ in $M$. For each $K_i$ we consider $T_i$ a tubular neighborhood of $K_i$. 
\begin{defi}
  A Hyperbolic cone structure on $M$ with a singularity at $L$ is a hyperbolic structure $(\mathrm{Dev},\rho)$ on $M\setminus L$ such that : 
 \begin{itemize}
     \item The developing map $\mathrm{Dev}$ on $\widetilde{T_i\setminus K_i }$ extends to the universal cover $\widetilde{T}$ branched over $K_i$.
     \item The developing map sends $ \widetilde{K_i} $ to a geodesic $l_{K_i}$ in $\mathbb{H}^3$.
     \item  The holonomy of a meridian curve $m$ encircling $K_i$ is a rotation of angle $\theta_i\in [0,2\pi]$ along $l_{K_i}$.
 \end{itemize}
 The angle $\theta_i$ is called the cone angle along $K_i$.
\end{defi}
\begin{example}\label{cone_exemple}
     Take $\mathrm{P}$, $\mathrm{Q}$ two planes in $\mathbb{H}^3$ that intersect along a geodesic $l$, assume that the angle between $\mathrm{P}$ and $\mathrm{Q}$ is $\theta$. The local model of a hyperbolic cone structure is obtained by gluing the faces of $\mathrm{S}^{\mathbb{H}^3}(\mathrm{P},\mathrm{Q})$ by a rotation around $l$. See Figure \ref{Cone} below.

\begin{figure}[htb]
\centering
\includegraphics[width=.7\textwidth]{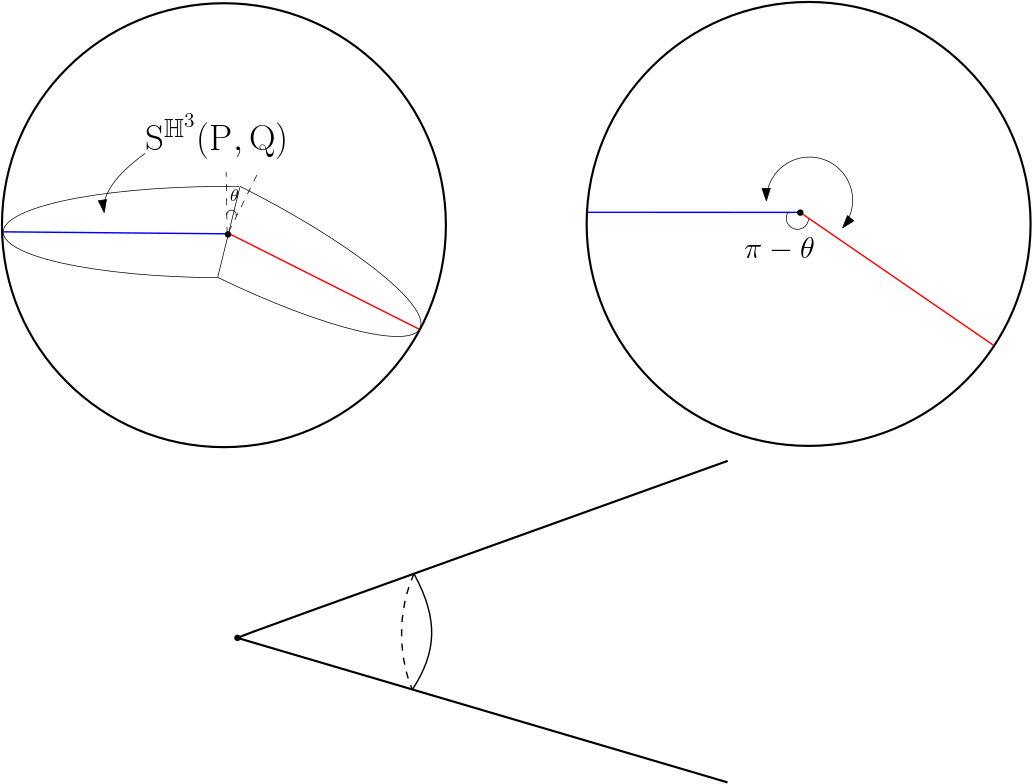}
\caption{On the left, we have a model of a hyperbolic $3$-manifold with cone singularity obtained by gluing two halfplanes in $\mathrm{S}^{\mathbb{H}^3}(\mathrm{P},\mathrm{Q})$ that are bounded by a geodesic $l$ using a rotation around $l$. On the right, we have a perpendicular cross section drawn in the Poincar\'e disk model of $\mathbb{H}^2$.}\label{Cone}
\end{figure}

\end{example}

We have also an analogous definition for Anti-de Sitter manifolds.
\begin{defi}
  An Anti- de Sitter cone structure on $M$ with a spacelike singularity at $L$ is a $\ads$-structure $(\mathrm{Dev},\rho)$ on $M\setminus L$ such that 
 \begin{itemize}
     \item The developing map $\mathrm{Dev}$ on $\widetilde{T_i\setminus K_i }$ extends to the universal cover $\widetilde{T_i}$ branched over $K_i$.
     \item The developing map sends $\widetilde{K_i}$ to a spacelike geodesic $l_{K_i}$ in $\ads$.
     \item  The holonomy of a meridian curve $m$ encircling $K_i$ is a rotation of angle $\theta >0$ along $l_{K_i}$.
 \end{itemize}
 The Lorentzian angle $\theta$ is called the \textit{mass} of the singularity $K_i$.
\end{defi}
The local model of an Anti-de Sitter cone structure can be given as in Example \ref{cone_exemple} by gluing the faces of $\mathrm{S}^{\ads}(\mathrm{P},\mathrm{Q})$ using a rotation around the spacelike geodesic $l=\mathrm{P}\cap \mathrm{Q}$.  
Note that there are various types of singularities in Anti-de Sitter manifold described in \cite{collision}. Our definition of singularity corresponds to what they referred as Tachyon singularity. Finally, we define Half-pipe cone structure with spacelike singularities which can be seen as the infinitesimal hyperbolic cone structure and Anti-de Sitter cone structure. See \cite[Proposition $23$]{danciger_thesis}.   
\begin{defi}
  A Half-pipe cone structure on $M$ with singularity at $L$ is a $\HP$-structure $(\mathrm{Dev},\rho)$ on $M\setminus L$ such that: 
 \begin{itemize}
     \item The developing map $\mathrm{Dev}$ on $\widetilde{T_i\setminus K_i }$ extends to the universal cover $\widetilde{T_i}$ branched over $K_i$.
     \item The developing map sends $\widetilde{K_i}$ to a spacelike geodesic $l_{K_i}$ in $\HP$.
     \item  The holonomy of a meridian curve $m$ encircling $K_i$ is a Half-pipe rotation of angle $\theta >0$ along $l_{K_i}$.
 \end{itemize}
 \end{defi}
An important example of structure with singularity that interest us are the structure obtained by doubling the convex core of a $\mathbb{H}^3$-quasi-Fuchsian representation (and their analogue in $\ads$ and $\HP$). See Section \ref{dh} below.

\subsection{Doubling the holonomy}\label{dh}
Having established the transition at the level of holonomy for a hyperbolic and Anti-de Sitter convex core structure with convex core pleated along $\vert t\vert\lambda$ and $\vert t\vert \mu$, we now discuss the transition of the doubled structure. First, let us clarify the notion of doubled structure. Let $N=\Sigma\times[-1,1]$ and consider the map $\tau$ defined as follows:
$$\begin{array}{ccccc}
\tau & : & N & \to & N \\
 & & (x,s) & \mapsto & (x,-s) \\
\end{array}$$
$\tau $ is an orientation reversing involution which fixes pointwise $\Sigma\times\{0\}$ and switches between $\Sigma\times\{-1\}$ and $\Sigma\times\{1\}$. 
Then we define the \textit{double} of $N$ by $$\mathcal{D}(N):=\Sigma\times [-1,1]/(x,1)\sim (x,-1).$$
Clearly $\mathcal{D}(N)$ is homeomorphic to $\Sigma\times \mathbb{S}^1$. When $N$ is endowed with a convex core structure $(\mathrm{dev},\rho)$ which is hyperbolic (resp. Anti- de Sitter or half pipe structure) and the boundary of the convex core is pleated along weighted multicurves $\lambda$ and $\mu$, then $\mathcal{D}(N)\setminus \left(\vert \lambda\vert\times\{1\} \cup\vert\mu\vert\times \{0\} \right)$ has a natural hyperbolic (resp. Anti-de Sitter or Half-pipe structure) with cone singularities along $\vert \lambda \vert \times \{1\}$ and $\vert\mu\vert\times \{0\}$. Moreover if $\alpha$ is a closed curve contained in the support of $\lambda$ or $\mu$ with weight $a>0$, then the cone angle around $\alpha$ is $2(\pi-a)$ (resp. $-2a$ in the case of Anti-de Sitter or Half-pipe structure), see Remark \ref{holonomy_of_meridian}. We refer to this as the \textit{doubled convex core structure} induced by $(\mathrm{dev},\rho)$.\\ 
We will explain in this subsection how to construct the holonomy representation of the doubled convex core structure induced by $(\mathrm{dev},\rho)$.
To simplify the notation, we will denote by $L$ the union $\vert \lambda \vert \times \{1\}\cup \vert \mu\vert \times \{0\}$ and by $M$ the manifold $\mathcal{D}(N)\setminus L$. The fundamental group of $M$ is described in details in \cite{choiesereis}. We recall here the construction for the convenience of the reader.\\
Let $\Sigma_0,\cdots \Sigma_q$ be the connected components of $\Sigma\setminus\left(\vert \lambda\vert\cup\vert \mu\vert\right)$, for each $i$ we choose a base point $x_i$ in $\Sigma_i$ and a path $\beta_i$ from $x_0$ to $x_i$. Then the paths $e_i:=\beta_i\tau(\beta_i)^{-1}$ project to a loop in $M$, see Figure \ref{double}.

\begin{figure}[htb]
\centering
\includegraphics[width=.7\textwidth]{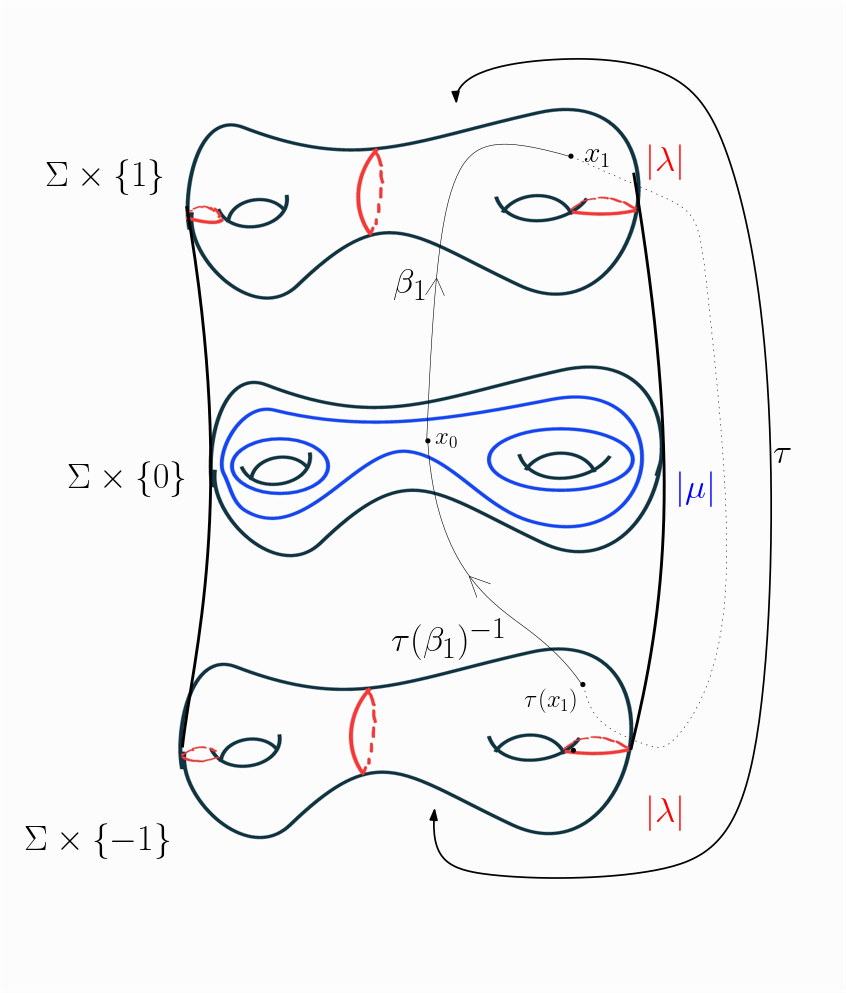}
\caption{A picture illustrating the construction of the path $e_i=\beta_i\tau(\beta_i)^{-1}.$}\label{double}
\end{figure}

For each $i$, let $M_i$ be the manifold obtained by gluing $\Sigma_i$ to $\tau(\Sigma_i)$. The fundamental group of $M$ is obtained by induction on $i$ as follow.
First by Van Kampen's theorem $$\pi_1(M_0,x_0)=\pi_1(N,x_0)\ast_{\pi_1(\Sigma_0,x_0)}\pi_1(\tau(N),\tau(x_0)).$$
Now assume that we have $\pi_1(M_{i-1},x_{0})$, then the fundamental group of $\pi_1(M_{i},x_{0})$ is the HNN-extension of $\pi_1(M_{i-1},x_{0})$ relative to $G_i:=\mathrm{Im}(\pi_1(\Sigma_{i},x_{i})\hookrightarrow \pi_1(M_{i},x_{0}))$. Namely
$$\pi_1(M_{i},x_{0})=\langle \pi_1(M_{i-1},x_{0})), e_i \ \vert \ e_i^{-1}\gamma e_i=\tau(\gamma), \ \forall \gamma\in G_i\rangle. $$
Now we can give an explicit description of the holonomy for the doubled convex core structure associated to $(\mathrm{dev},\rho)$ that we will denote by $\widehat{\rho}$. But before that, since the holonomy of certain elements of $\pi_1(M)$ will be a composition of reflections, let us recall the definition of reflection:

\begin{defi}
A reflection $r$ in $\mathbb{H}^3$ (resp. $\ads$, $\HP$) is a non-trivial involution of $\mathbb{H}^3$ (resp. $\ads$, $\HP$) that fixes point-wise a plane in $\mathbb{H}^3$ (resp. $\ads$, $\HP$).
\end{defi}
\begin{remark}
   Given a plane $\mathrm{P}$ in $\mathbb{H}^3$ (resp. in $\ads$), then there is a unique reflection $r$ fixing the plane $\mathrm{P}$. However, this result is false in $\HP$. Indeed, if the plane $\mathrm{P}$ is \textit{degenerate} (i.e., $\mathrm{P}$ contains a fiber), then it is proven in \cite[Proposition 4.15]{riolo_seppi} that there is a one-parameter family of reflections in $\HP$ which fix the plane $\mathrm{P}$ pointwise. However, if $\mathrm{P}$ is spacelike, then there is a unique reflection fixing point-wise $\mathrm{P}$, and by duality, this reflection coincides with the Minkowski transformation given by $x \to -x+v$ for some $v\in \mathbb{R}^{1,2}$, which is an isometry of $\mathbb{R}^{1,2}$ that reverses the orientation. For more details, see \cite[Section 4.5]{riolo_seppi}.
\end{remark}

Let $r_0$ be the reflection along the totally geodesic plane which is a support plane for $\partial_+\mathrm{CH}(\Lambda_{\rho})$ at $\mathrm{dev}(x_0)$. Define the representation \begin{equation}\label{36}
    \rho\ast r_0\rho r_0^{-1}: \pi_1(\Sigma,x_0)\ast \pi_1(\tau (\Sigma),\tau(x_0))\to \mathrm{Isom}(\X),\end{equation} such that its restriction to $\pi_1(\Sigma,x_0)$ and $\pi_1(\tau (\Sigma),\tau(x_0))$ are $\rho$ and $r_0\rho r_0^{-1}$ respectively.
Note that $\rho(\gamma)$ commutes with $r_0$ for all $\gamma\in \pi_1(\Sigma_0)$ and hence $\rho\ast r_0\rho r_{0}^{-1 }$ descends to a representation $\rho^0:\pi_1(M_0,x_0)\to \isom(\X).$\\
Now suppose that at the $i^{th}$ stage we have constructed the holonomy  $\rho^{i-1}:\pi_1(M_{i-1},x_0)\to \isom(\X)$ an we want to extend it to $\rho^i:\pi_1(M_i,x_0)\to\isom(\X)$. To do that it is enough to describe the holonomy of the loop $e_i$. It is not difficult to see that   
\begin{equation}\label{hol_extra_curve}
\rho^i(e_i)=r_ir_0\end{equation}
where $r_i$ is the reflection along the support plane at $\mathrm{dev}(x_i)$. Moreover such support plane is fixed by $\rho(\pi_1(\Sigma_i,x_0))$ hence we have a well defined representation on $\pi_1(M_i,x_0)$. Finally the holonomy $\widehat{\rho}$ will be equal to $\rho^i$ obtained inductively.

Now, let us come back to our convex core structures $(\mathrm{Dev}_{(t\lambda,t\mu)},\rho_{(t\lambda,t\mu)})$ and $(\mathrm{Dev}_{(\lambda,\mu)}^{\HP},\rho_{(\lambda,\mu)}^{\HP})$ constructed in Section \ref{TROFHOL} and \ref{TRNDEV}. Then we get the following. 
\begin{prop}\label{dhh}
Let $\widehat{\rho}_{(t\lambda,t\mu)}:\pi_1(M)\to \isom(\X)$ and $\widehat{\rho}^{\HP}_{(\lambda,\mu)}: \pi_1(M)\to \isom(\HP)$ be the holonomy of the doubled convex core structure induced by $(\mathrm{Dev}_{(t\lambda,t\mu)},\rho_{(t\lambda,t\mu)})$ and $(\mathrm{Dev}_{(\lambda,\mu)}^{\HP},\rho_{(\lambda,\mu)}^{\HP})$ respectively. Then the path of rescaled representations $\tau_t\widehat{\rho}_{(t\lambda,t\mu)}\tau_t^{-1}$ converges as $t\to 0^{\pm}$ to $\widehat{\rho}^{\HP}_{( \lambda,\mu)}$.
\end{prop}
In order to prove the Proposition, we need the following claim:
\begin{claim}\label{reflection}
Let $\mathrm{P}_t$ be a family of planes in $\mathbb{H}^3$ (resp. spacelike planes in $\ads$) defined for $t>0$ (resp. $t<0$) such that $\lim_{\vert t\vert \to 0}\tau_t\mathrm{P}_t=\mathrm{P}_0 $, where $\mathrm{P}_0$ is a spacelike plane in $\HP$. Let $r_t$ be a family of reflections along $\mathrm{P}_t$. Then

$$\lim_{t \to 0^+} \tau_t r_t\tau_t^{-1}=\lim_{t \to 0^-} \tau_t r_t\tau_t^{-1} =r_0.$$
Where $r_0$ is the Half-pipe reflection along $\mathrm{P}_0$.

\end{claim}
\begin{proof}
To prove the Claim, let $\alpha_t=(\alpha_0(t),\alpha_1(t),\alpha_2(t),\alpha_3(t))^T$ be a unit normal vector of $\mathrm{P}_t$ with respect to the quadratic form $q_1$ (for $t>0$) or $q_{-1}$ (for $t<0$). Then the hyperbolic or Anti-de Sitter reflection along $\mathrm{P}_t$ can be written in the standard basis as
\begin{equation}\label{formulereflec}
r_t=\mathrm{Id}\mp2J_{\pm}\alpha_t\alpha_t^T,
\end{equation}
where $J_{\pm}$ is the matrix $\mathrm{diag}(-1,1,1,1)$ if $t>0$ and $\mathrm{diag}(-1,1,1,-1)$ if $t<0.$ An elementary argument shows that  $\tau_t\mathrm{P}_t$ is the orthogonal of $(-\frac{\alpha_0}{t},\frac{\alpha_1}{t},\frac{\alpha_2}{t},\alpha_3)^{T}  $ with respect to the quadratic form $q_t$. Since $\tau_t\mathrm{P}_t$ converges to the spacelike plane $\mathrm{P}_0$, then $(-\frac{\alpha_0(t)}{t},\frac{\alpha_1(t)}{t},\frac{\alpha_2(t)}{t},\alpha_3(t))^{T}$ converges to $(a_0,1)^{T}$. A direct computation shows that 
$$\lim_{t \to 0^\pm}\tau_t r_t\tau_t^{-1}=\begin{pmatrix}
\mathrm{Id} & 0  \\
-2a_0 & -1 
\end{pmatrix},$$
which is the Half-pipe reflection $r_0$ along $\mathrm{P}_0$.
\end{proof}

\begin{proof}[Proof of Proposition \ref{dhh} ]
From the description of the fundamental group of $M$ and the Theorem \ref{H}, it is enough to check the limit only for the paths $e_i$. By identity \eqref{hol_extra_curve}, we have
$$\tau_t\widehat{\rho}_{(t\lambda,t\mu)}(e_i)\tau_t^{-1}=(\tau_tr_{t,i}\tau_t^{-1})(\tau_t r_{t,0}\tau_t^{-1}).$$
where $r_{t,i}$, $r_{t,0}$ are the reflections along the support planes of $\partial \mathrm{CH}(\Lambda_{(t\lambda,t\mu)})$ at $\mathrm{Dev}_{(t\lambda,t\mu)}(x_i)$ and $\mathrm{Dev}_{(t\lambda,t\mu)}(x_0)$ respectively. Moreover, the rescaled of those support planes converge because $\tau_t\mathrm{Dev}_{(t\lambda,t\mu)}$ converge to $\mathrm{Dev}_{(\lambda,\mu)}^{\HP}$ (See Theorem \ref{exdev}), then the conclusion follows from Claim \ref{reflection}.
\end{proof}

\subsection{Doubling the developing map }
Let $(\mathrm{dev},\rho)$ be a convex core structure on $N=\Sigma\times[0,1]$. We will look more closely at the developing map of the doubled convex core structure induced by $(\mathrm{dev},\rho)$, that we will denote by $\widehat{\mathrm{dev}}$. Let us denote by $\Gamma$ the fundamental group of $M$ and for each $i=0,\cdots q$, $\Gamma_i$ the fundamental group of the manifold $M_i$ obtained by gluing $\Sigma_i$ to $\tau(\Sigma_i)$ as explained in subsection \ref{dh}. Let $\widetilde{M_i}$ be a copy of the universal cover of $M_i$. As for the holonomy representation, the developing map $\widehat{\mathrm{dev}}$ of the doubled convex core structure induced by $(\mathrm{dev},\rho)$ will be constructed by induction on $i$. We will follow \cite[page 10]{Bendingdoubledev} where an explicit formula for a developing map of a gluing of two manifolds is given, see also \cite[page 33]{projective_non_hyperbolic}. For $i=0$, the universal cover of $M_0$ can be described combinatorially as

$$\widetilde{M_0}=\left(\Gamma_0\times \widetilde{N}\right)/\pi_1(N,x_0)\sqcup \left(\Gamma_0\times \widetilde{\tau(N)}\right)/\pi_1(\tau(N),\tau(x_0)) $$
where $\alpha \in \pi_1(N,x_0)$ (resp. in $\pi_1(\tau(N),\tau(x_0))$) acts on $\Gamma \times\widetilde{N}$ (resp. on $\Gamma \times\widetilde{\tau(N)}$) by
$$\alpha \cdot (\gamma, x) = (\gamma\alpha^{-1}, \alpha\cdot x).$$ Additionally, if $x\in \widetilde{N}\cap\widetilde{\tau(N)} $ then we
identify the point $(\gamma, x) \in \Gamma \times\widetilde{N}$ with the point $(\gamma, x) \in \Gamma \times\widetilde{\tau(N)}$. The action of $\Gamma_0$ on $\widetilde{M_0}$ is given by
$$\gamma_1\cdot[\gamma_2,x]=[(\gamma_1\gamma_2,x)].$$

Now let us define $\mathrm{dev}_{\tau}: \widetilde{\tau(N)}\to \X$ a developing map for $\tau(N)$ such that $\mathrm{dev}_{\tau}\circ \tau= r_0\mathrm{dev}$ where $r_0$ is a reflection along the totally
geodesic plane which is a support plane for $\partial_+\mathrm{CH}(\Lambda_{\rho})$ at $\mathrm{dev}(x_0)$ as in \eqref{36}. Therefore 
define a developing map $\mathrm{dev}^0:\widetilde{M_0}\to \X$ by
$$\mathrm{dev}^0[\gamma,x]=\begin{cases}
\hat{\rho}(\gamma)\mathrm{dev}(x) &\text{ if}\ x\in \widetilde{N} \\
\hat{\rho}(\gamma)\mathrm{dev}_{\tau}(x)  &\text{ if }\ x\in \widetilde{\tau(N)} \\
\end{cases}.$$
Now suppose inductively that we have constructed  a developing map $\mathrm{dev}^{i-1}$ and we want to extend it to $\mathrm{dev}^{i}:\widetilde{M_i} \to \X$. Here the universal cover of $M_i$ can be described combinatorially as 
$$\widetilde{M_i}=\Gamma_i\times \widetilde{M_{i-1}}/\Gamma_{i-1}$$
where $\alpha \in \Gamma_{i-1}$ acts on $\Gamma_i \times\widetilde{M_{i-1}}$ by
$$\alpha \cdot (\gamma, x) = (\gamma\alpha^{-1}, \alpha\cdot x).$$  The action of $\Gamma_i$ on $\widetilde{M_i}$ is given by
$$\gamma_1\cdot[\gamma_2,x]=[(\gamma_1\gamma_2,x)].$$ The developing map $\mathrm{dev}^i: \widetilde{M_i}\to \X$ is defined as follow:
\begin{equation}\label{formuledevdouble}\mathrm{dev}^{i}([\gamma,x])=\widehat{\rho}(\gamma)\mathrm{dev}^{i-1}(x).\end{equation} The developing map $\widehat{\mathrm{dev}}$ will be equal to $\mathrm{dev}^i$ obtained inductively. 

\begin{prop}\label{7.8}
Let $\widehat{\mathrm{Dev}}_{(t\lambda,t\mu)}:\widetilde{M}\to \X$ and $\widehat{\mathrm{Dev}}^{\HP}_{(\lambda,\mu)}: \widetilde{M}\to \HP$ be the developing maps of the doubled convex core structures induced by $(\mathrm{Dev}_{(t\lambda,t\mu)},\rho_{(t\lambda,t\mu)})$ and $(\mathrm{Dev}_{(\lambda,\mu)}^{\HP},\rho_{(\lambda,\mu)}^{\HP})$ respectively. Then the path $\tau_t\widehat{\mathrm{Dev}}_{(t\lambda,t\mu)}$
 converges as $t\to 0^{\pm}$ to $\widehat{\mathrm{Dev}}^{\HP}_{( \lambda,\mu)}$.
\end{prop}
\begin{proof}
The proof follows directly by applying Proposition \ref{dhh} on the convergence of $\tau_t\widehat{\rho}_{(t\lambda,t\mu)}\tau_t^{-1}$ and the formula \eqref{formuledevdouble} of $\widehat{\mathrm{dev}}$.
\end{proof}
\begin{figure}[htb]
\centering
\includegraphics[width=.7\textwidth]{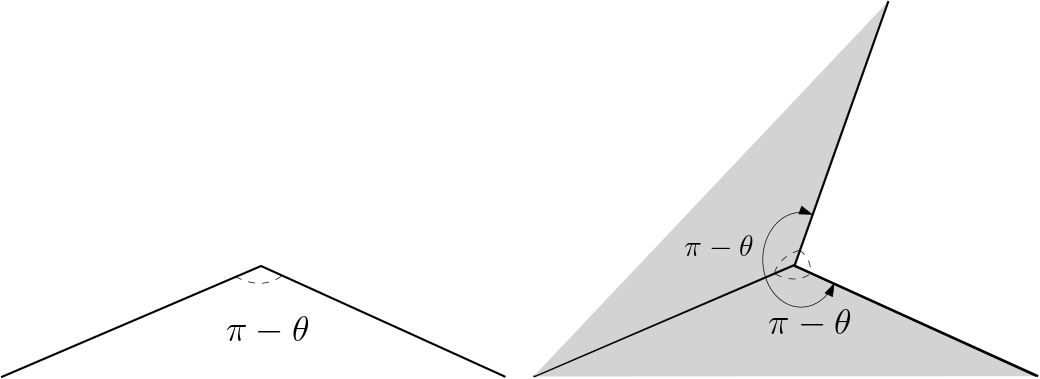}
\caption{A two-dimensional picture illustrating the doubling procedure in hyperbolic space. On the left, two lines with an exterior angle $\theta$ form a wedge of angle $(\pi-\theta)$. On the right, we first glue two copies of the left wedge along an edge to form another wedge with angle $2(\pi-\theta)$ (the grey part). Then, the edges of the grey wedge are glued by a rotation of angle $2(\pi-\theta)$.}\label{gluing}
\end{figure}
\begin{remark}\label{holonomy_of_meridian}
It is worth remarking that the holonomy of a meridian curve $\gamma\in \pi_1(M)$ around a curve $\alpha\times\{*\}$ is a rotation. Indeed, assume that $\alpha$ is in the support of $\lambda$, and let $\mathrm{P}_t$ and $\mathrm{Q}_t$ be two extremal support planes such that $\mathrm{P}_t\cap\mathrm{Q}_t=\alpha_t^+$, where $\alpha_t^+$ denotes the lift of the geodesic representative of $\alpha$ in $\partial_+\mathrm{CH}(\Lambda_{\rho_{(t\lambda,t\mu)}})/\rho_{(t\lambda,t\mu)}(\pi_1(\Sigma))$. Then, the holonomy of $\gamma$ is given by $r_{\mathrm{P}_t}r_{\mathrm{Q}_t}$, where $r_{\mathrm{P}_t}$ and $r_{\mathrm{Q}_t}$ denote the reflections along the planes $\mathrm{P}_t$ and $\mathrm{Q}_t$ respectively. Now, since the angle between $\mathrm{P}_t$ and $\mathrm{Q}_t$ is equal to $\vert t\vert \lambda(\alpha)$, a computation similar to that of the proof of Claim \ref{reflection} (more precisely, see equation \eqref{formulereflec}) shows that $r_{\mathrm{P}_t}r_{\mathrm{Q}_t}$ is a hyperbolic rotation of angle $2(\pi-\vert t\vert\lambda(\alpha))$ if $t>0$, and an Anti-de Sitter rotation of angle $-2\vert t\vert\lambda(\alpha)$ if $t<0$, see Figure \ref{gluing}. Therefore, by Proposition \ref{rotation}, $\tau_t \left(r_{\mathrm{P}_t}r_{\mathrm{Q}_t}\right)\tau_t^{-1}$ converges to a Half-pipe rotation of angle $-2\lambda(\alpha)$, which is the holonomy $\widehat{\rho}^{\HP}_{(\lambda,\mu)}(\gamma)$.

\end{remark}

\subsection{Geometry of cusps}
In this section, we will describe the geometry of the doubled convex core structure $(\widehat{\mathrm{Dev}}_{(t\lambda,t\mu)},\widehat{\rho}_{(t\lambda,t\mu)})$ near the punctures. Before that, let us examine the geometry before doubling the convex core. We will keep the same notations as in the proof of Theorem \ref{exdev}. We identify $\widetilde{\Sigma}$ with $\mathbb{H}^2$ using a complete hyperbolic structure $(\mathrm{dev},\sigma)$. Let $V_i$ be a neighborhood around the $i^{th}$ puncture, and let $C_i$ be a connected component of the lift of $\widetilde{V_i}$ in $\mathbb{H}^2$ such that $\partial C_i$ is a horocycle centered at $p_i$. Recall that $p_i$ is the fixed point of the parabolic isometry $\sigma(\gamma_i)$, where $\gamma_i$ is a loop around the puncture representing $\partial V_i$. Then, $C_t^+=\mathrm{Dev}_{(t\lambda,t\mu)}(\partial C_i\times \{1\} )$ and $C_t^-=\mathrm{Dev}_{(t\lambda,t\mu)}(\partial C_i\times \{0\} )$ are contained in the support planes $\mathrm{P}_t^{+}$ and $\mathrm{P}_t^-$ of $\partial\mathrm{CH}(\Lambda_{\rho_{(t\lambda,t\mu)}})$, respectively. Moreover, these two support planes intersect at a unique point $q_t$, which is also the intersection point between $C_t^+$ and $C_t^-$. Denote by $q\in \partial\HP\setminus{[0,0,0,1]}$ the limit of $\tau_tq_t$ as $t\to 0^{\pm}$. We consider $H_t$ the region in $\X$ defined by
\begin{itemize}
\item If $t>0$
\begin{equation}
H_t=\mathbb{P}\left(\{x\in \mathcal{H}^3 \ \mid \ \langle x,\widetilde{q_t}\rangle_{1,3}>a \}\right),\end{equation}

\item If $t<0$
\begin{equation}H_t=\mathbb{P}\left(\{x\in \mathcal{A}\mathrm{d}\mathcal{S}^3 \ \mid \ \langle x,\widetilde{q_t}\rangle_{2,2}>a \}\right),\end{equation}

\item If $t=0$
\begin{equation}H_0=\mathbb{P}\left(\{x\in \mathcal{HP}^3 \ \mid \ \langle x,\widetilde{q}\rangle_{1,2,0}>a \}\right),\end{equation}
\end{itemize}
where $a$ is a negative real number and $\widetilde{q_t}$, $\widetilde{q}$ are points in $\mathbb{R}^4$ which project to $q_t$ and $q$ respectively (remark that the boundary $\partial H_t$ of $H_t$ is a horosphere in $\X$). We can assume that $\vert a\vert $ is small enough so that $H_t$ does not intersect the bending lines of $\partial\mathrm{CH}(\Lambda_{\rho_{(t\lambda,t\mu)}})$.

Hence, the subgroup of $\rho_{(t\lambda,t\mu)}(\pi_1(\Sigma))$ preserving $H_t$ is generated by the parabolic isometry $\rho_{(t\lambda,t\mu)}(\gamma_i)$. We consider the region obtained by truncating $\mathrm{Dev}_{(t\lambda,t\mu)}(C_i\times[0,1])$ with $H_t$, this region is invariant by the subgroup generated by $\rho_{(t\lambda,t\mu)}(\gamma_i)$ and converges after rescaling by $\tau_t$ to the region obtained by truncating $\mathrm{Dev}_{(\lambda,\mu)}^{\HP} (C_i\times[0,1])$ with $H_0$. (Remark that in the upper half-space model of $\mathbb{H}^3$, such a region is isometric to $\{(x,y,z)\in \mathbb{R}^3 \ \mid z>1,\  \vert x\vert<c\}$ for some $c>0$).

We now analyze the situation after doubling. Recall that $M=\mathcal{D}(N)\setminus L\cong(\Sigma\times\mathbb{S}^1)\setminus L$. The holonomy of a curve close to the $i^{th}$ puncture, representing $\mathbb{S}^1$ (namely $\{y_i\}\times\mathbb{S}^1$, for $y_i\in V_i$), is given by $r_t^-r_t^+$, where $r_t^-$ and $r_t^+$ are reflections in $\mathbb{H}^3$ (resp. in $\ads$) if $t>0$ (resp. if $t<0$) along the planes $\mathrm{P}_t^-$ and $\mathrm{P}_t^+$. Since $\tau_t\mathrm{P}_t^{\pm}$ converges to spacelike planes $\mathrm{P}^{\pm}$, Claim \ref{reflection} implies that $\tau_tr_t^{+}\tau_t^{-1}$ and $\tau_tr_t^-\tau_t^{-1}$ converge to $r_0^+$ and $r_0^-$, where $r_0^+$ and $r_0^-$ denote the Half-pipe reflections along $\mathrm{P}^+$ and $\mathrm{P}^-$, respectively.

On the other hand, $r_t^-r_t^+$ (resp. $r_0^-r_0^+$) preserves the horosphere $H_t$ (resp. $H_0$). This follows from the fact that $r_t^{\pm}$ (resp. $r_0^{\pm}$) fix the intersection point $q_t\in \mathrm{P}_t^+\cap\mathrm{P}_t^-$ (resp. $q_0\in \mathrm{P}^+\cap\mathrm{P}^-$).

Consequently, the subgroup of $\widehat{\rho}_{(t\lambda,t\mu)}(\pi_1(M))$ (resp. $\widehat{\rho}^{\HP}_{(\lambda,\mu)}(\pi_1(M))$) preserving $H_t$ is isomorphic to $\mathbb{Z}^2$. The first generator is given by the parabolic isometry $\widehat{\rho}_{(t\lambda,t\mu)}(\gamma_i)$ (resp. $\widehat{\rho}^{\HP}_{(\lambda,\mu)}(\gamma_i)$), and the second generator is given by $r_t^-r_t^+$ (resp. $r_0^-r_0^+$). To see that these generators are linearly independent in $\mathbb{Z}^2$, it is enough to remark that $\widehat{\rho}_{(t\lambda,t\mu)}(\gamma_i)$ (resp. $\widehat{\rho}^{\HP}_{(\lambda,\mu)}(\gamma_i)$) fixes the support planes $\mathrm{P}_{t}^{\pm}$ (resp. $\mathrm{P}^{\pm}$), while $r_t^-r_t^+$ (resp. $r_0^-r_0^+$) does not fix them. Therefore, the quotient by the group generated by $\widehat{\rho}_{(t\lambda,t\mu)}(\gamma_i)$ and $r_t^-r_t^+$ of the truncation of $\widehat{\mathrm{Dev}}_{(t\lambda,t\mu)}(C_i\times\mathbb{S}^1)$ with the $H_t$ is isometric to a cusp in a hyperbolic manifold (if $t>0$) and in an Anti-de Sitter manifold (if $t<0$). Moreover, it converges to the quotient by the group generated by $\widehat{\rho}^{\HP}_{(\lambda,\mu)}(\gamma_i)$ and $r_0^-r_0^+$ of the region obtained by truncating $\widehat{\mathrm{Dev}}^{\HP}_{(\lambda,\mu)}(C_i\times\mathbb{S}^1)$ with $H_0$, which is isometric to a cusp in a Half-pipe manifold.\\

We now have all the tools to prove the main Theorem \ref{mainthm}:
\begin{proof}[Proof of the main Theorem \ref{mainthm}]
Let $\mathcal{P}_t$ be the real projective structure on $M=(\Sigma\times\mathbb{S}^1)\setminus L$ given by $(\tau_t\widehat{\mathrm{Dev}}_{t\lambda,t\mu},\tau_t\widehat{\rho}_{(t\lambda,\mu)}\tau_t^{-1})$. These structures are conjugate to the hyperbolic (resp. AdS) doubled convex core structure with bending data $(\vert t\vert \lambda,\vert t\vert \mu)$ if $t>0$ (resp. $t<0$). By Proposition \ref{7.8}, $\mathcal{P}_t$ converges to the Half-pipe doubled convex core structure with bending data $(\lambda,\mu)$ that we denote by $\mathcal{P}_0$. Remark \ref{holonomy_of_meridian} shows that the value of the cone angle around the link $L$ is exactly as stated in the statement of the main Theorem. Finally, the discussion above shows the cusped structure near the punctures, this completes the proof.
\end{proof}

\section{Appendix.}\label{9}

\subsection{Technical Lemma}
We provide here the proofs of the following property of elements in $\isom(\mathbb{H}^3) $ and $\isom(\ads)$.
\begin{lemma}\label{rcompos}
For $t>0$, let $x_t$ be a family of point in $\mathbb{H}^3$ which converges to $p_{\infty}$ in $\mathbb{H}^2$. Consider $\mathrm{P}_t$ a family of planes in $\mathbb{H}^3$ containing $x_t$ such that 
     $\tau_tx_t\to x_{\infty}$ and $\mathrm{P}_t\to\mathbb{H}^2$ as $t\to0^+$. Then, there exists a family of isometries $\B_t$ in $\isom(\mathbb{H}^3)$ such that: 
\begin{itemize}
    \item $\B_t(x_t)=p_{\infty}$.
    \item $\B_t(\mathrm{P}_t)=\mathbb{H}^2$.
    \item The family $\B_t$ converges to the identity. Moreover the rescaled limit $\tau_t\B_t\tau_t^{-1}$ converges to an isometry in $\isom(\HP)$.
\end{itemize}
The same statement holds in Anti-de Sitter space for spacelike planes.
\end{lemma}

\begin{proof}
Since the action of $\isom(\mathbb{H}^2)$ is transitive on $\mathbb{H}^2$ and $\tau_t$ commutes with $\isom(\mathbb{H}^2)$, then we can assume that $p_{\infty}=[1,0,0,0]$. The construction of $\B_t$ is then divided into four steps.
\begin{itemize}
    \item[1] First, let $\B_t^1$ be an isometry of $\mathbb{H}^2$ such that $\B_t^1(x_t)=[1,0,0,x_t^3]$. Since $x_t$ converges, the family $\B_t^1$ can be chosen convergent. Hence $\tau_t\B_t^1\tau_t^{-1}=\B_t^1$ is also convergent. Thus we can assume that $x_t$ lies on the geodesic $\{x_1=0, x_2=0\}\cap \mathbb{H}^3$.
    
    \item[2]Next, we construct an isometry $\B_t^2$ that maps $y_t:=\B_t^1(x_t)$ to $[1,0,0,0]$. To do this, let $d_t$ be the hyperbolic distance between $[1,0,0,0]$ and $y_t$. Since $y_t$ is in the geodesic $\{x_1=0, x_2=0\}\cap \mathbb{H}^3$ then
    $$y_t=[\cosh(d_t),0,0,\sinh(d_t)].$$
    In particular $x_t^3=\tanh{d_t}$. Define now the isometry
    $$    \B_t^2=\begin{bmatrix}
\cosh(d_t) & 0 & 0& -\sinh(d_t)\\
0 & 1 &0 & 0 \\
0&0&1&0\\
-\sinh(d_t)&0&0&\cosh(d_t)
\end{bmatrix},$$
we have $\B_t^2(y_t)=[1,0,0,0]$. Since $\tau_ty_t$ converges then $\frac{d_t}{t}$ converge, thus by a direct computation one can see that $\tau_t\B_t^2\tau_t^{-1}$ is convergent. We conclude that the isometry $\tau_t\B_t^2\B_t^1\tau_t^{-1}$ converge to an isometry in $\HP$, moreover $\B_t^2\B_t^1(x_t)=[1,0,0,0].$ 

\item[3] We define $\mathrm{Q}_t=\B_t^2\B_t^1(\mathrm{P}_t)$ and $l_t=\mathrm{Q}_t\cap \mathbb{H}^2$. If $\mathrm{Q}_t$ is equal to $\mathbb{H}^2$ then we are done. Otherwise, $l_t$ is a family of geodesics that converges to a geodesic in $\mathbb{H}^2$ (because $\B_t^2\B_t^1$ converges). Thus, we can find an isometry $\B_t^3\in \isom(\mathbb{H}^2)$ such that $\B_t^3(l_t)=\mathbb{H}^3\cap \{x_2=x_3=0\}$. Again, we may assume that $\B_t^3$ is convergent and so $\tau_t\B_t^3\tau_t^{-1}=\B_t^3$ also converges.

\item[4] In the last step, we define $\mathrm{Q}_t^3:=\B_t^{3}(\mathrm{Q}_t)$ and consider $\alpha_t$ the angle between $\mathrm{Q}_t^3$ and $\mathbb{H}^2$. Then the isometry:
$$    \B_t^4=\begin{bmatrix}
1 & 0& 0& 0\\
0 & \cos(\alpha_t)  &\sin(\alpha_t)  & 0 \\
0&-\sin(\alpha_t) &\cos(\alpha_t) &0\\
0&0&0&1
\end{bmatrix}$$
 is a rotation in $\mathbb{H}^3$ which sends $\mathrm{Q}_t^3 $ to $\mathbb{H}^2$. Since $\mathrm{P}_t\to \mathbb{H}^2$ then $\mathrm{Q}_t=\B_t^2\B_t^1(\mathrm{P}_t)$ converges to $\mathbb{H}^2$, this implies that $\alpha_t$ converges to $0$ and so $\tau_t\B_t^4\tau_t^{-1}$ is convergent.
\end{itemize}
 Finally, we define $\B_t=\B_t^4\B_t^3\B_t^2\B_t^1$. By construction, $\B_t(x_t)=p_{\infty}$, $\B_t(\mathrm{P}_t)=\mathbb{H}^2$ and $\B_t$ converges to the identity. Moreover, the families $\tau_t\B_t^i\tau_t^{-1}$ are convergent for $i=1,\cdots,4$. Hence the family $\tau_t\B_t\tau_t^{-1} = \tau_t\B_t^4\tau_t^{-1}\tau_t\B_t^3\tau_t^{-1}\tau_t\B_t^2\tau_t^{-1}\tau_t\B_t^1\tau_t^{-1}$ is also convergent. Thus, we have constructed a family of isometries $\B_t$ satisfying the conditions of the Lemma.
Note that the same proof holds when we take isometry in $\ads$. We need only change the formula of rotation and translation by the analogue formula in $\ads$. Namely we exchange $\sin$ and $\cos$ with $\sinh$ and $\cosh$ respectively. The details are left to the reader. 
\end{proof}

\color{black}
\bibliographystyle{alpha}
\bibliography{sample.bib}

\end{document}